\documentclass[12pt]{amsart}
\usepackage{amsmath,amssymb,latexsym,MnSymbol}
\usepackage{fullpage}
\usepackage{amscd}
\theoremstyle{plain}
\newtheorem{theorem}{Theorem}

\newtheorem{proposition}{Proposition}
\newtheorem{lemma}{Lemma}

\newtheorem{conjecture}{Conjecture}

\def\Q{\mathbb{Q}}
\def\A{\mathbb{A}}
\def\C{\mathbb{C}}

\def\Mat{\text{Mat}}
\def\diag{\text{diag}}
\def\Hom{\text{Hom}}
\numberwithin{equation}{section}
\begin{document}

\title{On the Whittaker range of the generalized metaplectic theta lift}
\author{Solomon Friedberg and David Ginzburg}
\date{September 10, 2021}
\address{Friedberg:  Department of Mathematics, Boston College, Chestnut Hill, MA 02467-3806, USA}
\email{solomon.friedberg@bc.edu}
\address{Ginzburg: School of Mathematical Sciences, Tel Aviv University, Ramat Aviv, Tel Aviv 6997801,
Israel}
\email{ginzburg@post.tau.ac.il}
\thanks{This work was supported by the BSF,
grant number 2020011 and by the NSF, grant numbers DMS-1801497 and DMS-2100206 (Friedberg).}
\begin{abstract}
The classical theta correspondence, based on the Weil representation, allows one to lift automorphic representations on symplectic groups or their double covers to automorphic representations on
special orthogonal groups.
It is of interest to vary the orthogonal group and describe the behavior in this theta tower (the Rallis tower).
In prior work, the authors obtained an extension of the classical theta correspondence to higher degree metaplectic covers of symplectic and special orthogonal groups that is based on the tensor product
of the Weil representation with another small representation.  In this work we study the existence of generic lifts in the resulting theta tower.  In the classical case, 
there are two orthogonal groups that may support
a generic lift of an irreducible cuspidal automorphic representation of a symplectic group.
We show that in general the Whittaker range consists of $r+1$ groups for the lift from the $r$-fold cover of a symplectic group.  We 
also give a period criterion for the genericity of the lift at each step of the tower.
\end{abstract}
\maketitle
\setcounter{MaxMatrixCols}{20}

\section{Introduction}\label{intro}

The classical theta correspondence gives a systematic way to lift automorphic representations on one group to automorphic representations on another.
This correspondence is based on restricting the theta representation on the metaplectic double cover of a symplectic group to a reductive dual pair in the symplectic group (or, more accurately, its
inverse image in the double cover).
In \cite{F-G3} the authors introduced an extension of the notion of reductive dual pair that allows the construction of correspondences using the tensor product of two theta
representations. We used this to construct a theta lift between genuine automorphic representations on higher degree metaplectic covers of
symplectic and orthogonal groups. Our lift matches the classical theta lift for the trivial cover of the orthogonal group, as in that case one of the theta representations in the tensor product is trivial.   
In this paper we study the question of when the generalized theta lift of a given automorphic representation is globally generic.  
We shall show that the theory for the classical theta lift is the first case of a broader theory that, in general, involves more possibilities.

Let $r\geq1$ be an odd integer, $F$ denote a global field which contains a full set of $r$-th roots of unity $\mu_r$, and ${\A}$ denote the ring of adeles of $F$. 
Let $Sp_{2n}^{(r)}(\A)$ denote the $r$-fold metaplectic cover of the symplectic group $Sp_{2n}({\A})$, which is a topological central extension of $Sp_{2n}(\A)$ by $\mu_r$.
The construction of such a covering group goes back to Matsumoto \cite{Mat}.
Let $\pi^{(r)}$ denote a
genuine irreducible cuspidal automorphic representation of $Sp_{2n}^{(r)}(\A)$. 
In \cite{F-G3}
 the authors constructed, via the generalized theta lift described above, 
an automorphic representation $\sigma^{(\kappa r)}$ of the $\kappa r$-fold metaplectic cover $SO_k^{(\kappa r)}({\A})$ of the split special orthogonal group
$SO_k({\A})$, with $\kappa=1$ if $k$ is even and $\kappa=2$ if $k$ is odd.  
The authors studied the Rallis tower (i.e.\ the lift for fixed $\pi^{(r)}$ and varying $k$), 
proved that the first nontrivial occurrence in the tower is cuspidal, and showed that the unramified local lift is functorial in the equal rank case.

In this paper we 
concentrate on the tower of split even orthogonal groups $SO_{2k}$, so that $\kappa=1$; we recall the construction of the lift in this case
in equation \eqref{lift2}.  Denote the automorphic representation so-obtained by  $\sigma_{n,k}^{(r)}$.  Let $\psi$ be a fixed nontrivial additive character of $F\backslash \A$.
The main objects of study here are the Whittaker coefficients attached to $\sigma_{n,k}^{(r)}$. 
These coefficients are introduced below in equation \eqref{wh1} and denoted ${\mathcal W}_{k,\delta}(f)$. Here $\delta\in F^*$, and $f$ is a function in the 
space of $\sigma_{n,k}^{(r)}$.  
The class of $\delta$ modulo $(F^*)^2$ indexes an orbit
of generic characters modulo the action of the rational split torus by conjugation. (One could also obtain these orbits by setting $\delta=1$ and varying $\psi$ in \eqref{whit-char}.)
If ${\mathcal W}_{k,\delta}(f)$ is not zero for some choice of data and some $\delta$, we say that $\sigma_{n,k}^{(r)}$ is (globally) generic.  
In this paper we will give conditions on $\pi^{(r)}$ and $k$ so that ${\mathcal W}_{k,\delta}(f)$ is not zero for some $f$ and relate this question to certain periods of $\pi^{(r)}$.
We will determine the locations $k$ in the Rallis tower such that $\sigma_{n,k}$ can be generic,  the $k$ for which it must be generic provided that $\pi^{(r)}$ is generic, and the $k$ for which it can never be generic.
In \cite{F-G3} we  did not prove that a nonvanishing lift of $\pi^{(r)}$ always occurs, but we will establish this for generic automorphic representations here.

Before stating our main results, it will be useful to describe the situation in the case when $r=1$, that is in the case of the classical global theta correspondence. 
Let $\theta_{4kn}^{(2)}$ be a theta series defined on the classical metaplectic group $Sp_{4kn}^{(2)}({\A})$, the double cover of the symplectic group $Sp_{4nk}(\A)$. 
See for example \cite{G-R-S2}, Section 1, part {\bf 6}.
These theta series
depend on a choice of $\psi$ as above and a Schwartz function $\phi$.  (We always suppress $\phi$ from the notation, but sometimes write 
$\theta_{4kn}^{(2),\psi}$.)
Let $\iota_1$ be the tensor product embedding $\iota_1:SO_{2k}\times Sp_{2n}\to Sp_{4nk}$.  The double cover splits over the image of $\iota_1$, as in Kudla \cite{Ku1}.  (See also Sweet \cite{Sw}.)
Let $\pi$ be an irreducible cuspidal automorphic representation of $Sp_{2n}(\A)$.
Denote by $\sigma_{n,k}$ the representation of $SO_{2k}(F)\backslash SO_{2k}({\A})$ generated by all functions $f(h)$ obtained by using $\theta_{4kn}^{(2)}$ as an integral kernel,
\begin{equation*}
f(h)=\int\limits_{Sp_{2n}(F)\backslash Sp_{2n}({\A})}
\varphi(g)\theta_{4kn}^{(2)}(\iota_1(h,g))\,dg,
\end{equation*}
with $\varphi$ in $\pi$ and all Schwartz functions $\phi$.
The map from $\pi$ to $\sigma_{n,k}$ is the classical theta correspondence. See for example Howe \cite{Ho}.

The properties of the classical theta correspondence of concern to us here are the following, established by Ginzburg, Rallis and Soudry \cite{G-R-S1}.
\begin{theorem}\label{classic-theorem}
 Suppose that $r=1$, so that $\sigma_{n,k}$ is the classical theta lift of the irreducible cuspidal automorphic representation $\pi$ on $Sp_{2n}(\A)$ to $SO_{2k}(\A)$.
\begin{enumerate}
\item Suppose that $k\neq n, n+1$.  Then ${\mathcal W}_{k,\delta}(f)$ is zero for all choices of data and all $\delta$, that is, the representation $\sigma_{n,k}$ is not generic.
\item\label{classic-one-theorem-part2}  Suppose that $k=n+1$.  Then the coefficient ${\mathcal W}_{k,\delta}(f)$ is not zero for some choice of data if and only if $\pi$ is a 
generic representation with respect to $\delta$.
\item Suppose that $k=n$.  Then the coefficient ${\mathcal W}_{k,\delta}(f)$ is not zero for some choice of data if and only if a certain period integral depending on $\delta$
(given by \eqref{intro2} below with $m=1$) is nonzero
for some choice of data.
\end{enumerate}
\end{theorem}
The condition
on genericity in part~\ref{classic-one-theorem-part2} once again reflects that the specific choice of generic character for $\pi$ is related to $\delta\in F^*$, and is needed since
there is more than one orbit of generic characters for $Sp_{2n}(F)\backslash Sp_{2n}(\A)$ under the action by conjugation of the standard maximal torus of $Sp_{2n}(F)$.
For more details see \cite{G-R-S1},~Theorem 2.1.

In view of Theorem~\ref{classic-theorem}, we refer to the numbers $k=n,n+1$ as the {\sl Whittaker range} of the theta lift from $Sp_{2n}$ to $SO_{2k}$.  Also, it follows that
a generic representation appears in the tower of theta lifts of $\pi$, i.e., the representation $\sigma_{n,k}$ is generic for some $k$, some $\delta$, and some choice of data, if and only if $\pi$ is generic.

In this paper we extend these properties to the general case, that is, all $r\geq1$, $r$ odd. 
We require two hypotheses. The first, the Orbit Conjecture (Conjecture~\ref{conj1} in Section~\ref{basic} below), 
concerns the Fourier coefficients associated to the theta representation $\Theta_{r-1}^{(r)}$ on the $r$-fold cover of $Sp_{r-1}$. 
It holds when $r=3$, $F=\Q(e^{2 \pi i /3})$ by work of Patterson \cite{Pat}. We give the context for this conjecture in Section~\ref{basic}, and an unconditional result in 
Proposition~\ref{prop7} below.
The second, the Descent Conjecture, concerns the descent of the theta representation $\Theta_{2l}^{(r)}$.  It is formulated in Section~\ref{4G},  Conjecture~\ref{conj20}, below.
A slightly stronger version is also proposed in \cite{F-G1}, Conjecture~4.2.
Replacing the Descent Conjecture by a local version, which is known, we may obtain (Proposition~\ref{weaker-prop-4}) slightly weaker results without assuming this.

 Let $\pi^{(r)}$ and $\sigma_{n,k}^{(r)}$ be as above. 
We shall establish the following result.
\begin{theorem}\label{theorem-general}  Let $r\geq1$ be an odd integer.  Suppose that the Orbit and Descent Conjectures are satisfied.  Then the following statements hold.
\begin{enumerate}
\item Suppose that $k\neq n-\frac{r-1}{2}, n-\frac{r-3}{2},\ldots, n+\frac{r+1}{2}$.
Then ${\mathcal W}_{k,\delta}(f)$ is zero for all choices of data and all $\delta$, that is, the representation $\sigma^{(r)}_{n,k}$ is not generic.
\item
Suppose that $k=n+\frac{r+1}{2}$.  Then the coefficient ${\mathcal W}_{k,\delta}(f)$ is not zero for some choice of data if and only if $\pi^{(r)}$ is a generic representation with respect to $\delta$.
\item Suppose that $k$ is one of the values $n-\frac{r-1}{2}, n-\frac{r-3}{2},\ldots, n+\frac{r-1}{2}$. Then the coefficient ${\mathcal W}_{k,\delta}(f)$ is not zero for some choice of data if and only if a certain period integral,
depending on $k$ and $\delta$ (and given in Proposition~\ref{general-two} below), is nonzero for some choice of data.
\end{enumerate}
\end{theorem}
Note that Theorem~\ref{theorem-general} asserts that the Whittaker range expands from $k=n,n+1$ for the classical theta lift from $Sp_{2n}$ to $SO_{2k}$ , i.e.\ the case $r=1$, to 
$k=n-\frac{r-1}{2},n-\frac{r-3}{2},\ldots, n+\frac{r+1}{2}$ for the generalized theta lift from from $Sp_{2n}^{(r)}$ to $SO_{2k}^{(r)}$. 
We will describe the related period integrals momentarily.  In the general case a new phenomenon appears:
the shape of the period integral changes as we pass from the interval $2\le k\le \frac{r+1}{2}$ to the interval $\frac{r+3}{2}\le k\le n+\frac{r-1}{2}$; the first interval does not appear if $r=1$.
Also, it follows that, as in the classical case,
a generic representation appears in the tower of theta lifts of $\pi^{(r)}$, i.e., the representation $\sigma_{n,k}^{(r)}$ is generic for some $k$, some $\delta$, and some choice of data, if and only if $\pi^{(r)}$ is generic.

We remark that if $r>1$ we do not know whether or not the full Whittaker range is attained, equivalently
whether or not for each $k$ in the Whittaker range there exist an automorphic representation on $Sp_{2n}^{(r)}(\A)$ such that the corresponding period is nonzero for some choice of data
and some $\delta$. We suspect that this is the case.  
We also remark that one may fix the automorphic representation on the orthogonal group and consider the tower of symplectic groups, introduced by Rallis \cite{R}.  For $r=1$ and even orthogonal groups the
question of nonvanishing of the theta lift is treated by Roberts \cite{Ro}.

Before outlining the proof for general $r$, let us review the proof of Theorem~\ref{classic-theorem} and also the connection to period integrals in the classical case.  We break the proof of this result into four parts. 
The first two are:
\begin{proposition}\label{classic-one}  Suppose that $r=1$, so that $\sigma_{n,k}$ is the classical theta lift of $\pi$ on $Sp_{2n}$ to $SO_{2k}$.
\begin{enumerate}
\item{[Classical-1]}\label{classic-one-prop-part1} Suppose that $k\ge n+2$. Then ${\mathcal W}_{k,\delta}(f)$ is zero for all choices of data and all $\delta$.
\item{[Classical-2]} When $k=n+1$, the coefficient ${\mathcal W}_{k,\delta}(f)$ is not zero for some choice of data if and only if 
$\pi$ is a generic representation with respect to $\delta$.
\end{enumerate}
\end{proposition}\noindent
The proof of Proposition~\ref{classic-one} is based on a direct calculation of the Whittaker integral ${\mathcal W}_{k,\delta}(f)$ by unfolding the theta series and using root exchange. 
The vanishing in part~\ref{classic-one-prop-part1} makes use 
of the cuspidality of $\varphi$.  

To describe the remaining steps, we introduce the following notation. For $0\le m\le n$, let $U_{n,m}$ denote the unipotent radical of the parabolic subgroup of $Sp_{2n}$ whose Levi part is 
$GL_1^{m}\times Sp_{2(n-m)}$. When $m=0$, we define $U_{n,m}$ to be the trivial group. When $n=m$, the group $U_{n,m}$ is the maximal unipotent subgroup of $Sp_{2n}$. 
Fix $\psi$ a nontrivial additive character of $\A$ trivial on $F$, and 
let $\psi_{U_{n,m}}$ denote the character  of $U_{n,m}(F)\backslash U_{n,m}(\A)$ given on $u=(u_{i,j})\in U_{n,m}(\A)$ by 
\begin{equation}\label{character-U-nm}\psi_{U_{n,m}}(u)=\psi(u_{1,2}+u_{2,3}+\cdots+u_{m-1,m}).
\end{equation}
The group $U_{n,m}$ has a structure of a generalized Heisenberg group, and one can define a homomorphism $l$ from $U_{n,m}$ onto the Heisenberg group ${\mathcal H}_{2(n-m)+1}$. 
(See \eqref{the-hom-ell} below.)
Let $\theta_{2m}^{(2),\psi^\delta}$ be a theta function on $Sp_{2m}^{(2)}({\A})$ formed using the Weil representation with additive character $\psi^{\delta}$ given by $\psi^{\delta}(x)=\psi(\delta x)$.
The function $\theta_{2m}^{(2),\psi^\delta}$ is a function on
the semidirect product $\mathcal{H}_{2m+1}(\A)\rtimes Sp_{2m}^{(2)}({\A})$.
We then introduce the period integral
\begin{multline}\label{intro2}
\int\limits_{Sp_{2m}(F)\backslash Sp_{2m}({\A})}
\int\limits_{U_{n,n-m}(F)\backslash U_{n,n-m}({\A})}
\varphi\left(u\begin{pmatrix}I_{n-m}&&\\&h&\\&&I_{n-m}\end{pmatrix}\right)\theta_{2m}^{(2),\psi^\delta}(l(u)h)\overline{
\theta_{2m}^{(2),\psi^\delta}}(h)\\ \psi_{U_{n,n-m}}(u)\,du\,dh.
\end{multline}
In the integrand, the product of the two theta functions is not a genuine function of $h$, and 
so gives a function on the group $Sp_{2n}(\A)$  itself.
We have
\begin{proposition}\label{classic-two} Suppose that $r=1$, so that $\sigma_{n,k}$ is the classical theta lift of $\pi$ on $Sp_{2n}$ to $SO_{2k}$.
\begin{enumerate}
\item{[Classical-3]} \label{c2-1} Assume $2\le k\le n$, and write $m=n-k+1$. Then  $1\le m \leq n-1$. 
The Fourier coefficient ${\mathcal W}_{k,\delta}(f)$ of the representation $\sigma_{n,k}$ is  zero for all choices of data if and only if the period integral \eqref{intro2} is zero for all choices of data.
\item{[Classical-4]} \label{c2-2} Suppose that $k\le n-1$. Then the representation 
$\sigma_{n,k}$ is not generic. 
\end{enumerate}
\end{proposition}

The proof of Proposition~\ref{classic-two}, part~\ref{c2-1},  follows from a direct calculation of the integral ${\mathcal W}_{k,\delta}(f)$
To deduce part~\ref{c2-2}, one then uses the theory of Fourier-Jacobi coefficients as established by Ikeda \cite{I1}.
Indeed, it follows from this reference that $\theta_{2m}^{(2),\psi^\delta}(h)$ is the residue at a point $s_0$ of the Eisenstein series $E^{(2),\psi^\delta}(\cdot,s)$ defined on $Sp_{2m}^{(2)}({\A})$ 
that is induced from the trivial representation of $GL_m({\A})$. (The induction depends on a Weil factor that in turn depends on the choice of additive character $\psi^\delta$.) 
Hence if \eqref{intro2} is not zero for some choice of data, then the integral
\begin{multline}\label{intro3}
\int\limits_{Sp_{2m}(F)\backslash Sp_{2m}({\A})}
\int\limits_{U_{n,n-m}(F)\backslash U_{n,n-m}({\A})}
\varphi\left(u\begin{pmatrix}I_{n-m}&&\\&h&\\&&I_{n-m}\end{pmatrix}\right)\theta_{2m}^{(2),\psi^\delta}(l(u)h)\overline{
E^{(2),\psi^\delta}(h,s)}\\
\psi_{U_{n,n-m}}(u)\,du\,dh
\end{multline}
is not zero for $\text{Re}(s)$ large. However, one may unfold this integral as in \cite{G-R-S2}, and doing so, one obtains an integral with
the Whittaker coefficient of the identity representation of $GL_m({\A})$ as an inner integration. Since this coefficient is zero if $m\ge 2$, part~\ref{c2-2} follows.

The proof of Theorem~\ref{theorem-general} will be given by generalizing Propositions~\ref{classic-one} and \ref{classic-two} to the case of general $r$.  To generalize Proposition~\ref{classic-one} we will show
\begin{proposition}\label{general-one}
Let $r\geq1$ be an odd integer.  Then the following statements hold.
\begin{enumerate}
\item{[General-1]}\label{g3-1}
Suppose that $k\ge n+\frac{r+3}{2}$. Then ${\mathcal W}_{k,\delta}(f)$ is zero for all choices of data and all $\delta$.
\item{[General-2]}\label{g3-2}
Suppose that the Orbit Conjecture is true.  
When $k=n+\frac{r+1}{2}$, the coefficient ${\mathcal W}_{k,\delta}(f)$ is not zero for some choice of data if and only if $\pi^{(r)}$ is a generic representation with respect to $\delta$.
\end{enumerate}
\end{proposition}

Proposition~\ref{prop7} below formulates the dependence of the genericity
on $\delta$ precisely, and gives an unconditional statement as well. The second step is to generalize Proposition~\ref{classic-two}.    We will show
\begin{proposition}\label{general-two}
Let $r\geq1$ be an odd integer.  Suppose that the Descent Conjecture holds.  
Then the following statements hold.
\begin{enumerate}
\item{[General-3]} \label{g4-1}
Suppose that $2\le k\le \frac{r+1}{2}$. Then the coefficient ${\mathcal W}_{k,\delta}(f)$ is not zero for some choice of data if and only if there exists a choice of data such that the period integral
\begin{multline}\label{intro4}
\int\limits_{Sp_{2n}(F)\backslash Sp_{2n}({\A})}\ \  
\int\limits_{U_{n+\frac{r+1}{2}-k,\frac{r+1}{2}-k}(F)\backslash 
U_{n+\frac{r+1}{2}-k,\frac{r+1}{2}-k}({\A})}
\overline{\varphi^{(r)}}(h)\\ \theta_{2n}^{(2),\psi^\delta}(l_{2n}(u)h) \theta_{2n+r-2k+1}^{(2r),\psi^\delta}\left(u\begin{pmatrix}I_{\tfrac{r+1}{2}-k}&&\\&h&\\&&I_{\tfrac{r+1}{2}-k}\end{pmatrix}\right)
\psi_{U_{n+\frac{r+1}{2}-k,\frac{r+1}{2}-k}}(u)\,du\,dh
\end{multline}
is not zero. Here the function $\theta_{2n+r-2k+1}^{(2r),\psi^\delta}$ is in the space of the theta representation $\Theta_{2n+r-2k+1}^{(2r),\psi^\delta}$ that is obtained from the residues of 
Eisenstein series on $Sp_{2n+r-2k+1}^{(2r)}(\A)$.
The remaining notation will be defined below. 
\smallskip

Suppose instead that $\frac{r+3}{2}\le k\le n+\frac{r-1}{2}$. Then the coefficient ${\mathcal W}_{k,\delta}(f)$ is not zero for some choice of data if and only there exists a choice of data such that 
the period integral
\begin{multline}\label{intro5}
\int\limits_{Sp_{2n-2k+r+1}(F)\backslash Sp_{2n-2k+r+1}({\A})}\ \  
\int\limits_{U_{n,k-\frac{r+1}{2}}(F)\backslash 
U_{n,k-\frac{r+1}{2}}({\A})}
\overline{\varphi^{(r)}}\left(u\begin{pmatrix}I_{k-\tfrac{r+1}{2}}&&\\&h&\\&&I_{k-\tfrac{r+1}{2}}\end{pmatrix}\right)\\ 
\theta_{2n+r-2k+1}^{(2),\psi^\delta}(l_{2n+r-2k+1}(u)h) \theta_{2n+r-2k+1}^{(2r),\psi^\delta}(h)\psi_{U_{n,k-\frac{r+1}{2}}}(u)\,du\,dh
\end{multline}
is not zero.  
\item{[General-4]} \label{g4-2}
Suppose that $k\le n-\frac{r+1}{2}$. Then the representation  $\sigma_{n,k}^{(r)}$ is not generic.
\end{enumerate}
\end{proposition}

In the integrals above, the unipotent groups embed in their covers by means of the trivial section $u\mapsto (u,1)$.
The integrands are independent of the central subgroup as a function of $h$, so integration over the adelic quotients of the symplectic groups makes sense.
(In greater detail, we integrate in $h$ over the full $2r$-fold cover of the indicated symplectic group and in evaluating $\varphi^{(r)}$ and $\theta^{(2)}$ we first project $h$
from the $2r$-fold cover to the $r$-fold and $2$-fold covers, resp., as in \cite{F-G3}, Section~3.  But as the integrand is ultimately independent of $\mu_{2r}$ it is more natural to regard
it as simply a function on the group.). Also, the theta representation $\Theta_{2n+r-2k+1}^{(2r),\psi^\delta}$ is defined in \cite{F-G1}, pg.\ 93; the dependence on the additive character is 
explained in \cite{F-G2}, pg.\ 1926.

Though this work focusses on global genericity, 
one may also study genericity in the context of the local generalized theta correspondence  (\cite{F-G3}, Section 6). 
See for example Baki\'c \cite{Bak} for the classical theta correspondence. We remark that some of the methods
developed here could also be transported to the local situation. 

We now describe the proofs in brief and also the structure of this paper.
In Section~\ref{basic} we introduce the groups of concern, and discuss the Fourier coefficients of the theta representation.
We state Theorem~\ref{theta05} (proved in \cite{F-G3}) that gives information about these coefficients, and 
the Orbit Conjecture that is used in the sequel. 
Section~\ref{gend-theta-lift} describes the construction of the generalized theta lift of \cite{F-G3}.
Then in Section~\ref{whcoeff} we introduce the family of Whittaker coefficients that are to be studied, and compute them for $f\in\sigma_{n,k}^{(r)}$.  To do so, we unfold
the classical theta function and then use root exchange, as formulated
in Section 7 of Ginzburg, Rallis and Soudry \cite{G-R-S4}, extensively.  We also use the vanishing of the Fourier coefficients for the theta representation
that are attached to certain unipotent orbits (Theorem~\ref{theta05}, part~\ref{theta05-1}).
This allows us to establish an inductive process that terminates in the integral $L(j)$ given by \eqref{wh4}
with $j=\min(k-2,(r-3)/2)$. We then analyze each case for this minimum.

In Section~\ref{case1}, the case that $\min(k-2,(r-3)/2)=k$ (i.e., $k\leq(r+1)/2$) is studied; note that this case does not arise in the classical theta correspondence.
Here we use further root exchanges and the results of Ikeda mentioned above to arrive at the period \eqref{intro4}. The case that $\min(k-2,(r-3)/2)=(r-3)/2$ (i.e., $k>(r+1)/2$) is studied
in Section~\ref{case2}.  The analysis here is roughly similar but uses additional Fourier expansions over certain unipotent subgroups of symplectic groups.
We conclude that in this case the Whittaker coefficient $\mathcal{W}_{k,\delta}(f)$ is nonzero if and only the integral $\mathcal{L}(j)$ given by \eqref{end1} is nonzero when 
$j=\min(n,k-(r+1)/2)$.   This allows us to prove Proposition~\ref{general-one}, part~\ref{g3-1}, using once again the vanishing of the Fourier coefficients of the theta representation  
(Theorem~\ref{theta05}, part~\ref{theta05-1}), and Proposition~\ref{general-one}, part~\ref{g3-2} after using
the non-vanishing that is the Orbit Conjecture. 
It also allows us to establish Proposition~\ref{general-two}, part~\ref{g4-1}, assuming the Descent Conjecture 
(whose discussion is deferred to Section~\ref{4G}).
The concluding Section~\ref{4G} discusses the Descent Conjecture and the relation of this conjecture to a strong multiplicity one statement, establishes Proposition~\ref{general-two}, part~\ref{g4-2} under it, 
and explains how we can show 
part~\ref{g4-2} without the Descent Conjecture by imposing an additional condition on $\pi^{(r)}$ (Proposition~\ref{weaker-prop-4}).

\section{Notation, Fourier Coefficients of the Theta Representation}\label{basic}
In this section we fix some notations, and then discuss the Fourier coefficients of the theta representation.
Let $\Mat_{a\times b}$ be the algebraic group of all matrices of size $a\times b$ and write simply $\Mat_a$ for $\Mat_{a\times a}$.
For $m\geq1$ let 
 $J_m\in \Mat_m$ be the matrix 
 $$J_m=\begin{pmatrix}&&1\\&\udots&\\1&&\end{pmatrix}.$$
 
 Let $Sp_{2m}$ denote the symplectic group
$$Sp_{2m}=\left\{g\in GL_{2m}~\mid~ ^t\!g \begin{pmatrix}0&J_m\\-J_m&0\end{pmatrix} g=\begin{pmatrix}0&J_m\\-J_m&0\end{pmatrix}\right\}$$
and for $k\geq2$ let $SO_k$ denote the split special orthogonal group
$$SO_k=\left\{ g\in GL_k~ \mid ~^t\!g\, J_k\,g=J_k\right\}.$$
Let ${\mathcal H}_{2l+1}$ be the Heisenberg group in $2l+1$ variables, realized as in \cite{F-G3} as all elements of the form $(X,Y,z)$ where $X,Y\in \Mat_{1\times l}$ and $z\in \Mat_1$. 
This group may be embedded into $Sp_{2l+2}$ as in \cite{F-G3}, p.\ 1537.
We also define
$\Mat_a^0=\{ Z\in \Mat_{a}\ :\ Z^tJ_{a}=J_{a}Z\}$ and $\Mat_a^{00}=\{ Z\in \Mat_{a}\ :\ Z^tJ_{a}=-J_{a}Z\}$.

Let $e_{i,j}$ denote the square matrix, whose size will be clear by context, with $(i,j)$-th entry $1$ and all other entries $0$. When we work with
$Sp_{2l}$, we let $e_{i,j}'=e_{i,j}\pm e_{2l-j+1,2l-i+1}$ with 
the sign chosen so that $e_{i,j}'\in Sp_{2l}$.
When we work with $SO_{2k}$ we let
$e_{i,j}'=e_{i,j}- e_{2k-j+1,2k-i+1}$ for all choices of $(i,j)$.  

Let $F$ be a global field and ${\A}$ denote its ring of adeles. 
Let $\Theta_{2l}^{(2)}$ denote the theta representation on the group ${\mathcal H}_{2l+1}({\A})\rtimes Sp_{2l}^{(2)}({\A})$; this representation 
depends on a choice of a nontrivial additive character of $F\backslash {\A}$. See \cite{G-R-S2}, Section 1, part {\bf 6}, for the definition and the action of the Weil representation.  
Fix $r\geq1$ odd and suppose that $F$ contains a full group of $r$-th roots of unity, $\mu_{r}$, and fix an embedding $\epsilon:\mu_r\hookrightarrow \C^\times$. 
Let $G$ denote one of the groups $Sp_{2l}$ or $SO_{2k}$, and $G^{(r)}({\A})$ the $r$-fold metaplectic covering group of $G(\A)$. 
We work with functions
and representations that are genuine with respect to $\epsilon$.
Let $\Theta_{2l}^{(r)}$ denote the theta representation on $Sp_{2l}^{(r)}({\A})$.  The functions in this space are obtained as residues of Eisenstein series on $Sp_{2l}^{(r)}({\A})$.
For more details, see \cite{F-G1} Section 2, where we give basic properties of these representations. These are similar to the properties of theta 
representations on covers of the general linear group, treated in \cite{K-P}. 

To each unipotent orbit of the group $Sp_{2l}$ one may associate a set of Fourier coefficients, as described in \cite{G}.
Unipotent groups lift canonically to a central extension (\cite{M-W}, Appendix I), so one may also consider Fourier coefficients of functions defined on any cover of $Sp_{2l}$.
 Given an automorphic representation $\pi$
we let $\mathcal{O}(\pi)$ denote the set of unipotent orbits $\mathcal{O}$ that are maximal with respect to the property that there is a function $f\in \pi$ that has a nonzero Fourier
coefficient for $\mathcal{O}$.  The unipotent orbits are indexed by partitions of $2l$ such that every odd part has even multiplicity (see \cite{C-M}).
In this paper we are 
mainly interested in the set ${\mathcal O}(\Theta_{2l}^{(r)})$. 

The structure of the set  ${\mathcal O}(\Theta_{2l}^{(r)})$ is described in \cite{F-G3}, Section 4.
Let $r>1$ be an odd integer, and write $2l=\alpha r+\beta$ where $0\le \beta< r$.
Let ${\mathcal O}_c(\Theta_{2l}^{(r)})$ denote the unipotent orbit
$${\mathcal O}_c(\Theta_{2l}^{(r)})=\begin{cases}(r^\alpha \beta)&\text{if $\alpha$ is even}\\
(r^{\alpha -1}(r-1)(\beta+1))&\text{if $\alpha$ is odd.}\end{cases}$$
 Conjecture~1 of \cite{F-G3} (also made in \cite{F-G1}) states that the set ${\mathcal O}(\Theta_{2l}^{(r)})$ is a singleton, and is precisely the set
$\{{\mathcal O}_c(\Theta_{2l}^{(r)})\}$.  That is, all Fourier coefficients of functions in $\Theta_{2l}^{(r)}$ which are attached to unipotent orbits that are greater than or not comparable to
${\mathcal O}_c(\Theta_{2l}^{(r)})$ are zero, and for some function in $\Theta_{2l}^{(r)}$ there is a nonzero coefficient attached to this orbit.
We mention that the local analogue of this conjecture has been extended to other groups and described conceptually by Gao and Tsai \cite{G-T}.

In \cite{F-G3}, we establish the following result concerning the unipotent orbit of $\Theta_{2l}^{(r)}$.

\begin{theorem}\label{theta05}
\begin{enumerate}
\item\label{theta05-1}
For all positive integers $l$, if $\mathcal{O}\in {\mathcal O}(\Theta_{2l}^{(r)})$, then $\mathcal{O}\le {\mathcal O}_c(\Theta_{2l}^{(r)})$.
\item \label{theta05-2} Assume that $l=0,1,2,r-3,r-2,r-1$. Let $n$ denote a non-negative integer, and assume that if $l=0$, then $n\ge 1$. 
Then ${\mathcal O}(\Theta_{2(l+nr)}^{(r)})=\{{\mathcal O}_c(\Theta_{2(l+nr)}^{(r)})\}$.
In particular, if $r=3$ or $r=5$ then ${\mathcal O}(\Theta_{2l}^{(r)})=\{{\mathcal O}_c(\Theta_{2l}^{(r)})\}$ holds for all $l$.
\end{enumerate}
\end{theorem}

Thus the vanishing properties of Fourier coefficients implied by \cite{F-G3}, Conjecture~1, are known, 
but the non-vanishing properties are known in full only for $r=3,5$.  
We do not need this non-vanishing in general, but only in the case $2l=r-1$, but there to give the optimal results we need more.  In that case, \cite{F-G3}, Conjecture~1 asserts that
$\Theta_{r-1}^{(r)}$ is generic, i.e.\ that there is some function in the space of $\Theta_{r-1}^{(r)}$ that has a nonvanishing Whittaker coefficient with respect to some Whittaker character of 
$Sp_{r-1}$.  Such characters are given by 
$$e_{\delta}(u)=\psi(u_{1,2}+\dots+u_{r_1-1,r_1}+\delta u_{r_1,r_1+1})$$
with $r_1=(r-1)/2$ and $\delta\in F^*$. 
We conjecture this nonvanishing for every class of Whittaker characters modulo the action of the rational split torus.

\begin{conjecture}[Orbit Conjecture]\label{conj1} For each class in $F^*/(F^*)^2$, there is a representative $\delta\in F^*$ and a function $\theta$ in $\Theta_{r-1}^{(r)}$ 
such that the Whittaker integral of $\theta$ with respect to the character $e_{\delta}$ is nonzero.
\end{conjecture}

We remark that by a result of Gao \cite{Gao}, for any choice of additive character, or equivalently for any $\delta\in F^*$,
the local Whittaker functional for each nonarchimedean
local component of $\Theta_{r-1}^{(r)}$ exists and is unique up to scalars.  We also mention that for even degree covers an analogous statement is not
expected to hold (after all, the process of induction in the even cover case requires the Weil factor, which depends on a choice of additive character).

Next we describe the groups at hand in more detail, following the same notation as in \cite{F-G3}.   
All parabolic subgroups we use are standard parabolic subgroups whose unipotent radical consists of upper triangular unipotent matrices.
Let $P_{a,b,c}$ denote the parabolic subgroup of $Sp_{2(ab+c)}$ whose Levi part is $GL_{a}\times\ldots\times GL_{a}\times Sp_{2c}$ where $GL_{a}$ 
appears $b$ times. Let $U_{a,b,c}$ denotes the unipotent radical of $P_{a,b,c}$. 

We have the following matrices and subgroups of $U_{a,b,c}$.
For $1\le i\le b-1$, let
\begin{equation}\label{mat1}
u_{a,b,c}^i(X_i)=\begin{pmatrix} I_\alpha&&&&&&\\ &I_a&X_i&&&&\\ &&I_a&&&&\\ &&&I_\beta&&&\\
&&&&I_a&X_i^*&&\\ &&&&&I_a&\\ &&&&&&I_\alpha\end{pmatrix},\qquad
X_i\in \Mat_{a}.
\end{equation}
Here $\alpha=a(i-1)$,  and $\beta=2(ab+c)-2a(i+1)$.  Here and below, we use an asterisk to indicate that the entries are determined by the condition that the matrix is symplectic
(here the entries of $X_i^*$ are determined by those of $X_i$.)
We denote the subgroup consisting of all matrices $u_{a,b,c}^i(X_i)$ with $X_i\in \Mat_{a}$ by $U_{a,b,c}^i$. 
Let
\begin{equation}\label{mat2}
u_{a,c}'(Y,Z)=\begin{pmatrix} I_{2a(b-1)}&&&&\\ &I_{a}&Y&Z&\\ &&I_{2c}&Y^*&\\ &&&I_{a}&\\
&&&&I_{2a(b-1)}\end{pmatrix},\qquad Y\in \Mat_{a\times 2c}, Z\in \Mat_a^0,
\end{equation}
and let $U_{a,c}'$ be the subgroup $U_{a,b,c}$ of matrices of this form.

Each matrix $u\in U_{a,b,c}$ may be factored uniquely as
$$u=u_{a,b,c}^i(X_i)u'$$ 
with $u_{a,b,c}^i(X_i)\in U_{a,b,c}^i$ and where $u'\in U_{a,b,c}$ is such that all its $(p,j)$ entries are zero
for $\alpha+1\le p\le \alpha+a$ and  $\alpha+a+1\le j\le \alpha+2a$. We refer to $u_{a,b,c}^i(X_i)$ 
as the $i$-th coordinate of $u$. Similarly, we define the $u_{a,c}'(Y,Z)$ coordinate of $u$.
Every $u\in U_{a,b,c}$ has a  factorization of the form 
\begin{equation}\label{mat3}
u=u_{a,c}'(Y,Z)\prod_{i=1}^{b-1} u_{a,b,c}^i(X_i)u_1
\end{equation}
where $u_1\in U_{a,b,c}$ has zeroes in the first $b$ $a\times a$ blocks directly above the main diagonal and in positions $Y$ and $Z$ in \eqref{mat2} 
(and thus is also zero in the corresponding last $b$ blocks directly above the main diagonal and in $Y^*$), and
the matrices $X_i, Y$ and $Z$ are uniquely determined. Using this factorization, we define a character $\psi_{U_{a,b,c}}$ of $U_{a,b,c}(F)\backslash U_{a,b,c}({\A})$ by
$$\psi_{U_{a,b,c}}(u)=\psi(\text{tr}(X_1+\cdots +X_{b-1})).$$

The group $U_{2k,n}'$  has the structure of a generalized Heisenberg group. We make use of a homomorphism 
$l_{4kn}:U_{2k,r_1,n}\to {\mathcal H}_{4kn+1}$ given as follows. 
First, on the center 
of $U_{2k,n}'$, which consists of all matrices \eqref{mat2} such that $Y=0$, let $l_{4kn}(u_{2k,n}'(0,Z))=(0,0,\text{tr}'(Z))$. 
Here, for $Z=(Z_{i,j})\in \Mat_{2k}^0$ we write  
$$\text{tr}'(Z)=Z_{1,1}+Z_{2,2}+\ldots +Z_{k,k}.$$
Since $U_{2k,n}'$ modulo its center is isomorphic to ${\mathcal H}_{4kn+1}$ modulo its center, one may extend $l_{4kn}$ to $U_{2k,n}'$ by 
taking any isomorphism between these two quotients. We postpone a more precise description of which extension we choose until Section~\ref{whcoeff} where it is needed
(see below \eqref{wh2}). 
After defining $l_{4kn}$ on $U_{2k,n}'$ we extend it trivially to $U_{2k,r_1,n}$.  

\section{The Generalized Theta Lifting}\label{gend-theta-lift}

In this section we review the basic integral construction introduced in \cite{F-G3}, Section 1. 
We require two embeddings.  First, as in the classic theta lift, let $\iota_1:SO_{2k}\times Sp_{2n}\to Sp_{4nk}$ be the tensor product embedding.  Second, for given $r$, let $\iota_2: SO_{2k}\times Sp_{2n}\to
Sp_{2n+2k(r-1)}$ be the embedding 
\begin{equation}\label{iota-2}
\iota_2(h,g)=\diag(h,\dots,h,g,h^*,\dots,h^*)
\end{equation} where each of the matrices $h,h^*$ is repeated $r_1$ times and $h^*$ is chosen so that the matrix is in the
symplectic group.  We recall that $r_1=(r-1)/2$.  

These maps may be extended to covering groups, but this requires some care. As explained in \cite{F-G3}, Section 2, by properly choosing cocycles that realize the covering groups,
composing with the projections from 
$SO_{2k}^{(r)}(A)$ to $SO_{2k}(\A)$ and from $Sp_{2n}^{(r)}(\A)$ to $Sp_{2n}(\A)$,
and using the splitting of Kudla \cite{Ku1} and Sweet \cite{Sw}, the map $\iota_1$ may be extended to a map
$\iota_1:SO_{2k}^{(r)}(\A)\times Sp_{2n}^{(r)}(\A)\to Sp_{4nk}^{(2)}(\A)$
(we use the same notation for this extension).  Also, the map $\iota_2$ may be extended to a map
$\iota_2:SO_{2k}^{(r)}(\A)\times Sp_{2n}^{(r)}(\A)\to Sp_{2n+2k(r-1)}^{(r)}(\A).$

Let $\theta_{2n+2k(r-1)}^{(r)}$ be a vector in the space of the representation $\Theta_{2n+2k(r-1)}^{(r)}$,  and $\theta_{4kn}^{(2),\psi}$ be a vector in the space of the 
representation $\Theta_{4kn}^{(2)}$ with additive character $\psi$. 
Let $h\in SO_{2k}^{(r)}({\A})$ and $g\in Sp_{2n}^{(r)}({\A})$.
Consider the Fourier coefficient 
\begin{equation}\label{lift1}
\int\limits_{U_{2k,r_1,n}(F)\backslash U_{2k,r_1,n}({\A})}
\theta_{4kn}^{(2),\psi}(l_{4kn}(u)\iota_1(h,g))
\theta_{2n+2k(r-1)}^{(r)}(u\iota_2(h,g))\psi_{U_{2k,r_1,n}}(u)\,du.
\end{equation}
The idea is then to use this Fourier coefficient as an integral kernel.

Let $\pi^{(r)}$ denote a genuine irreducible cuspidal automorphic representation of $Sp_{2n}^{(r)}({\A})$, and 
let $\varphi^{(r)}$ be a vector in the space of $\pi^{(r)}$. Let $\sigma_{n,k}^{(r)}$ denote the representation of $SO_{2k}^{(r)}({\A})$ generated by all functions $f(h)$ given by the integrals
\begin{multline}\label{lift2}
\int\limits_{Sp_{2n}(F)\backslash Sp_{2n}^{(r)}({\A})}
\int\limits_{U_{2k,r_1,n}(F)\backslash U_{2k,r_1,n}({\A})}
\overline{\varphi^{(r)}(g)} \\
\theta_{4kn}^{(2),\psi}(l_{4kn}(u)\iota_1(h,g))
\theta_{2n+2k(r-1)}^{(r)}(u\iota_2(h,g))\psi_{U_{2k,r_1,n}}(u)\,du\,dg
\end{multline}
as $\varphi^{(r)}$, $\theta_{4kn}^{(2),\psi}$ and $\theta_{2n+2k(r-1)}^{(r)}$ vary over their representation spaces.
This defines a mapping from the set of irreducible cuspidal genuine automorphic representations of 
$Sp_{2n}^{(r)}({\A})$ to the set of genuine representations of the quotient $SO_{2k}(F)\backslash SO_{2k}^{(r)}({\A})$. 
This is the map that we shall study in the sequel.  We remark that as a function of $g=(g_1,\zeta)\in Sp_{2n}^{(r)}({\A})$, with $g_1\in. Sp_{2n}(\A)$, $\zeta\in\mu_r$,
the integrand in \eqref{lift2} is independent of $\zeta$, so the
integrand can also be pushed down to a function on the group $Sp_{2n}({\A})$ and the integration taken over the adelic quotient of this group.

\section{The Whittaker coefficients of the representation 
$\sigma_{n,k}^{(r)}$}\label{whcoeff}

In this Section we compute the Whittaker coefficients of the representation $\sigma_{n,k}^{(r)}$. 
Let $V_{2k}$ denote the maximal unipotent subgroup of $SO_{2k}$.
Let $\delta\in F^*$.  For $v=(v_{i,j})\in V_{2k}(\A)$, define
\begin{equation}\label{whit-char}\psi_{V_{2k},\delta}(v)=\psi(v_{1,2}+v_{2,3}+\ldots +\delta v_{k-1,k}+v_{k-1,k+1}).
\end{equation}
Then $\psi_{V_{2k},\delta}$ is a Whittaker character of the quotient
$V_{2k}(F)\backslash V_{2k}({\A})$. Up to conjugation by the rational split torus all generic characters of this quotient are of this form,
and in view of this action it suffices to consider $\delta\in F^*/(F^*)^2.$
Our goal in this section is to study the integrals
\begin{equation}\label{wh1}
{\mathcal W}_{k,\delta}(f)=\int\limits_{V_{2k}(F)\backslash V_{2k}({\A})}f(vh)
\psi_{V_{2k},\delta}(v)\,dv,
\end{equation}
for $f(h)$ a function in the space of the representation 
$\sigma_{n,k}^{(r)}$.

We start by fixing some notations.  
Given $u\in U_{2k,r_1,n}$, let
\begin{equation}\label{mat4}
u=u_{2k,n}'(Y,Z)\prod_{i=1}^{r_1-1} u_{2k,r_1,n}^i(X_i)u_1
\end{equation}
be its factorization as in equation \eqref{mat3}.  
Write  $Y\in \Mat_{2k\times 2n}$ as $Y=\begin{pmatrix} Y_1\\ Y_2\end{pmatrix}$ where $Y_1,Y_2\in \Mat_{k\times 2n}$. 
Let $U_{2k,r_1,n}^0$ denote the subgroup of $U_{2k,r_1,n}$ consisting of all matrices of the form \eqref{mat4} such that $Y_2=0$. 

For every $0\le j\le \min(k-2, r_1-1)$ we  define a unipotent subgroup $U_j$ of $Sp_{2(n+k(r-1)-jr)}$ as follows. Let 
$P_{2k,r_1,n}^j$ denote the standard parabolic subgroup of $Sp_{2(n+k(r-1)-jr)}$ whose Levi part is $GL_{2k-2j-1}^j\times GL_{2k-2j}^{r_1-j}\times Sp_{2n}$. 
Let $U_{2k,r_1,n}^{1,j}$ denote the standard unipotent radical of $P_{2k,r_1,n}^j$. Similarly to \eqref{mat4}, an element $u\in U_{2k,r_1,n}^{1,j}$ has a factorization 
\begin{equation}\label{mat5}
u=u_{2k-2j,n}'(Y,Z)u_{a_j,b_j,c_j,d_j}^{j}(X_j)
\prod_{\underset{i\ne j}{i=1}}^{r_1-1} u^i_{a_i,b_i,c_i,d_i}(X_i)u_1
\end{equation}
with $Y\in \Mat_{(2k-2j)\times 2n}$, $Z\in \Mat_{2k-2j}^0$, $X_j\in \Mat_{(2k-2j-1)\times (2k-2j)}$, $X_i\in \Mat_{2k-2j-1}$ if $1\le i\le j-1$, and  $X_i\in \Mat_{2k-2j}$ if $j+1\le i\le r_1-1$. 
Here for each $i$, $1\leq i\leq r_1-1$, we denote
\begin{equation}\label{mat6}
u_{a,b,c,d}^i(X_i)=\begin{pmatrix} I_a&&&&&&\\ &I_b&X_i&&&&\\ &&I_c&&&&\\ &&&I_d&&&\\
&&&&I_c&X_i^*&&\\ &&&&&I_b&\\ &&&&&&I_a\end{pmatrix}
\end{equation}
and the indices $a_i,b_i,c_i,d_i$ are given as follows.
If $1\le i\le j-1$, then  $a_i=(2k-2j-1)(i-1)$, $b_i=c_i=2k-2j-1$; 
if $i=j$ then $a_j=(2k-2j-1)(j-1)$, $b_j=2k-2j-1$, $c_j=2k-2j$;
and 
if $j+1\le i\le r_1-1$, then $a_i=(2k-2j)(i-1)-j$, $b_i=c_i=2k-2j$.
In all cases $d_i=2(n+k(r-1)-jr)-2(a_i+b_i+c_i)$.
For $Y\in \Mat_{2(k-j)\times 2n}$  write  $Y=\begin{pmatrix} Y_1\\ Y_2\end{pmatrix}$ where $Y_1,Y_2\in \Mat_{(k-j)\times 2n}$. 
We define the group $U_j$ to be the subgroup of $U_{2k,r_1,n}^{1,j}$ consisting of all matrices of the form \eqref{mat5} such that $Y_2=0$. 
Note that $U_0=U_{2k,r_1,n}^0$. 

Let $\psi_{U_j}$ denote the character of the quotient $U_j(F)\backslash U_j({\A})$ defined as follows. Write $u\in U_j$ as in equation \eqref{mat5}. Then 
\begin{equation}\label{wh2}
\psi_{U_j}(u)=
\psi\Bigg(\text{tr}_0(X_j) + \text{tr}'(Z) +
\sum_{\substack{{i=1}\\i\ne j}}^{r_1-1}\text{tr}(X_i)\Bigg).
\end{equation}
Here if $X_j[a,b]$ denotes the $(a,b)$ entry of $X_j$, then 
$$\text{tr}_0(X_j)=X_j[1,1]+X_j[2,2]+\ldots +X_j[2k-2j-1,2k-2j-1]$$ (we remind the reader that $X_j$ is not a square matrix). 

We start the computation of \eqref{wh1} by first unfolding the theta series $\theta_{4kn}^{(2),\psi}$. 
Choosing $l_{4kn}(u_{2k,n}'(Y,Z))=(Y_2,Y_1,\text{tr}'(Z))$, we deduce that the embedding $\iota_1$ of $(v,g)\in V_{2k}\times Sp_{2n}$ in $Sp_{4kn}$ 
preserves this choice of polarization. Indeed,  
the group $\iota_1(V_{2k}\times Sp_{2n})$ is contained in the maximal parabolic subgroup of $Sp_{4kn}$ whose Levi part is $GL_{2kn}$. Thus, 
\begin{align}\notag 
\theta_{4kn}^{(2),\psi}(l_{4kn}(u_{2k,n}'(Y,Z))(v,g))=&
\sum_{\xi\in \Mat_{k\times 2n}(F)}\omega_\psi(l_{4kn}(u_{2k,n}'(Y,Z))\iota_1(v,g))\phi(\xi)\\
&= \sum_{\xi\in Mat_{k\times 2n}(F)}\omega_\psi(l_{4kn}(u_{2k,n}'(\xi '+Y,Z))\iota_1(v,g))\phi(0).\notag
\end{align}
Here $\xi'=\begin{pmatrix} 0\\ \xi\end{pmatrix}$. Plugging this into integral \eqref{wh1}, 
collapsing the summation over $\xi$ with the corresponding integration over $U_{2k,r_1,n}(F)\backslash U_{2k,r_1,n}({\A})$, integral \eqref{wh1} is equal to 
\begin{multline}\notag
\int\limits_{Mat_{k\times 2n}({\A})}
\int\limits_{Sp_{2n}(F)\backslash Sp_{2n}({\A})}
\int\limits_{V_{2k}(F)\backslash V_{2k}({\A})}
\int\limits_{U_{2k,r_1,n}^0(F)\backslash U_{2k,r_1,n}^0({\A})}
\overline{\varphi^{(r)}(g)}\\ 
\omega_\psi(l_{4kn}(u)(\iota_1(v,g)l_{4kn}(u'_{2k,n}(Y',0)))\phi(0)
\theta_{2n+2k(r-1)}^{(r)}(u\iota_2(v,g)u'_{2k,n}(Y',0))\\
\psi_{U_{2k,r_1,n}^0}(u)\psi_{V_{2k},\delta}(v)\,dv\,du\,dg\,dY'.
\end{multline}
Here, we write $Y'=\begin{pmatrix} 0\\ Y_2\end{pmatrix}$.

From the action of the Weil representation (see, for example, \cite{G-R-S2} Section 1, part {\bf 6}), we have the identity
$\omega_\psi(l_{4kn}(u)\iota_1(v,g))\phi(0)=\phi(0)$.
Since $\phi$ is an arbitrary Schwartz function, we deduce that the integral \eqref{wh1} is zero for all choices of data if and only if the integral
\begin{multline}\label{wh3}
\int\limits_{Sp_{2n}(F)\backslash Sp_{2n}({\A})}
\int\limits_{V_{2k}(F)\backslash V_{2k}({\A})}
\int\limits_{U_{2k,r_1,n}^0(F)\backslash U_{2k,r_1,n}^0({\A})}
\overline{\varphi^{(r)}(g)}\\
\theta_{2n+2k(r-1)}^{(r)}(u\iota_2(v,g))\psi_{U_{2k,r_1,n}^0}(u)
\psi_{V_{2k},\delta}(v)\,du\,dv\,dg
\end{multline}
is zero for all choices of data.

For every $0\le j\le \min(k-2, r_1-1)$,  define the integral
\begin{multline}\label{wh4}
L(j)=\int\limits_{Sp_{2n}(F)\backslash Sp_{2n}({\A})}
\int\limits_{V_{2k-2j}(F)\backslash V_{2k-2j}({\A})}
\int\limits_{U_j(F)\backslash U_j({\A})}\overline{\varphi^{(r)}(g)}
\\ \theta_{2(n+k(r-1)-jr)}^{(r)}(u\iota_{3,j}(v,g))\psi_{U_j}(u)
\psi_{V_{2k-2j},\delta}(v)\,du\,dv\,dg.
\end{multline}
Here  for $v=\left(\begin{smallmatrix} v_0&B\\ &1\end{smallmatrix}\right)\in V_{2k-2j}(\A)$, $g\in Sp_{2n}(\A)$, the matrix $\iota_{3,j}(v,g)\in Sp_{2(n+k(r-1)-jr)}(\A)$ is given by
\begin{equation}\label{iota-3-j}
\iota_{3,j}(v,g)=\text{diag}(v_0,\ldots,v_0,v\ldots,v,g,v^*,\ldots,v^*,v_0^*,\ldots,v_0^*),
\end{equation}
where on the right the matrix $v$ appears $r_1-j$ times and the matrix $v_0$ appears $j$ times. 
Notice that since $v$ is an upper unipotent matrix, so is $v_0$; however, $v_0$ is not an orthogonal matrix. 
Also, $\iota_{3,0}(v,g)$ is exactly the embedding introduced in \eqref{iota-2}: $\iota_{3,0}(v,g)=\iota_2(v,g)$.
In \eqref{wh4}, we first
extend $\iota_{3,j}(\cdot,\cdot)$ to covering groups in the usual way.  That is, we observe that $V_{2k-2j}(\A)$ is unipotent, so canonically embeds in its cover by means of the trivial section, and if 
$g=(g_1,\zeta)\in Sp_{2n}^{(r)}(\A)$, $g_1\in Sp_{2n}(\A)$, $\zeta\in\mu_r$, then
we set $\iota_{3,j}(v,g)=(\iota_{3,j}(v,g_1),\zeta)\in Sp^{(r)}_{2(n+k(r-1)-jr)}(\A)$.  
Then as a function of $g=(g_1,\zeta)\in Sp_{2n}^{(r)}(\A)$, the integrand in \eqref{wh4} is independent of $\zeta\in \mu_r$, so the integrand descends to the group $Sp_{2n}(\A)$ (and is then 
integrated over this group).  
Also, $L(0)$ is equal to the integral \eqref{wh3}.

We introduce additional notation.  For any index $a$ and composition $(n_1,\dots,n_k)$ of $a$, let ${\mathcal U}_{n_1,\dots,n_k}$ be the unipotent radical of the standard
parabolic subgroup of $GL_a$ with Levi part $GL_{n_1}\times\dots\times GL_{n_k}$.  Also we write $L_a$ for the full subgroup of upper triangular unipotent matrices of $GL_a$
(so $L_a=\mathcal{U}_{1,\dots,1}$, but we use the less adorned notation).

\begin{lemma}\label{lem1}
For each $j$ with $0\le j< \min(k-2, r_1-1)$, the integral  $L(j)$ is zero for all choices of data if and only if integral $L(j+1)$ is zero for all choices of data.
\end{lemma}
\begin{proof}
The first step is to perform a certain root exchange. For this notion, see \cite{G-R-S4}, Lemma 7.1.
In this process one needs to introduce two unipotent groups, and under the conditions stated in \cite{G-R-S4} Section 7.1 
one can then perform the root exchange. We will also use \cite{G-R-S4} Corollary 7.1 repeatedly to deduce that the two integrals which we study have the same vanishing properties. 

Fix $j$. We start by defining unipotent groups $M_1$ and $Q_1$ of $Sp_{2(n+k(r-1)-jr)}$.  These will 
be the first groups on which we carry out the root exchange process. 
We define $M_1$ to be the image of $\mathcal{U}_{1,2k-2j-3,1}$ inside $Sp_{2(n+k(r-1)-jr)}$ under the embedding $u_1\to \text{diag} (u_1,I,u_1^*)$ 
where $I$ is the identity matrix of size $2(n+(k-j)(r-3)-(j-1))$. 
The group $Q_1$ is defined to be the group of all matrices of the form
\begin{equation}\label{mat7}
u_{0,2k-2j-1,2k-2j-1,d_1}^1(X_1);\quad X_1=\begin{pmatrix} 0&&\\ d&0_{2k-2j-3}& \\ e&f&0\end{pmatrix},\quad  d^t, f\in \Mat_{1\times (2k-2j-3)}, e\in \Mat_1
\end{equation}
where $d_1=2(n+k(r-1)-4(2k-2j-1)-jr)$ and $0_\alpha$ denotes the zero square matrix of size $\alpha$.
Notice that $Q_1$ is a subgroup of $U_j$.

We now carry out root exchange between the groups $M_1$ and $Q_1$.
In the notation of \cite{G-R-S4} Lemma 7.1, we let $B=U_j$, $D$ denote the semi-direct product of $M_1$ and $Q_1\backslash U_j$, and $Y=Q_1$. 
We conclude that integral \eqref{wh4} is equal to
\begin{multline}\notag
\int\limits_{Q_1({\A})}\int
\int\limits_{M_1(F)\backslash M_1({\A})}\ \ \ 
\int\limits_{Q_1({\A})U_j(F)\backslash U_j({\A})}
\overline{\varphi^{(r)}(g)}\\
\theta_{2(n+k(r-1)-jr)}^{(r)}(um_1\iota_{3,j}(v,g)q_1)\psi_{j}(u)\psi_{V_{2k-2j},\delta}(v)\,du\,dm_1\,dv\,dg\,dq_1.
\end{multline}
Here, the domains of integration of the variables $v$ and $g$ are the same as for integral $L(j)$ in \eqref{wh4} above.
Applying Corollary 7.1 in \cite{G-R-S4}, we deduce that the integral
$L(j)$ is zero for all choices of data if and only if the integral
\begin{multline*}
\int
\int\limits_{M_1(F)\backslash M_1({\A})}\ \ \ 
\int\limits_{Q_1({\A})U_j(F)\backslash U_j({\A})}
\overline{\varphi^{(r)}(g)}\\
\theta_{2(n+k(r-1)-jr)}^{(r)}(um_1\iota_{3,j}(v,g))\psi_{U_j}(u)\psi_{V_{2k-2j},\delta}(v)\,du\,dm_1\,dv\,dg
\end{multline*}
is zero for all choices of data. 

For $2\le i\le r_1$ we define the subgroups $M_i$ of $Sp_{2(n+k(r-1)-jr)}$ as follows (these also depend on $j$ but we suppress this dependence).
First, for $2\le i\le j$, the group $M_i$ is the image of $\mathcal{U}_{1,2k-2j-3,1}$ in $Sp_{2(n+k(r-1)-jr)}$ under the embedding
$$u_i\to \text{diag} (I_\alpha,u_i,I_\beta,u_i^*,I_\alpha),\qquad \alpha=(2k-2j-1)(i-1)$$ with $\beta$ chosen so that the embedding is into $Sp_{2(n+k(r-1)-jr)}$.
The group $M_{j+1}$ is the image of $\mathcal{U}_{1,1,2k-2j-3,1}$  in $Sp_{2(n+k(r-1)-jr)}$
under the map
\begin{equation}\label{mat82}
u_{j+1}\to \text{diag} (I_\alpha,u_{j+1},I_\beta,u_{j+1}^*,I_\alpha), \qquad \alpha=(2k-2j-1)j.
\end{equation}
For $i$ in the range $j+2\le i\le r_1-1$, $M_i$ is the image of $\mathcal{U}_{1,2k-2j-2,1}$
under the map
$$u_i\to \text{diag} (I_\alpha,u_i,I_\beta,u_i^*,I_\alpha),\qquad \alpha=(2k-2j)(i-1)-j.$$
Finally, the group $M_{r_1}$ is the image of the group $\mathcal{U}_{1,2k-2j-1}$
under the embedding
$$u_{r_1}\to \text{diag} (I_\alpha,u_{r_1},I_\beta,u_{r_1}^*,I_\alpha)\qquad \alpha=(2k-2j)(r_1-1)-j.$$ 
In each embedding above, $\beta$ is chosen such that the resulting matrix is in $Sp_{2(n+k(r-1)-jr)}$. 

Next, for $2\le i\le r_1$ and $j$ fixed, we define groups $Q_i$. 
For $2\le  i<r_1$, $Q_i$ is the group of all matrices of the form $u_{a_i,b_i,c_i,d_i}^i(X_i)$ (see \eqref{mat6}), where $X_i$ is specified as follows. 
If $2\leq i<j$ then $X_i$ runs over all matrices of the form $X_1$ that appear in equation \eqref{mat7}. 
For $i=j,j+1$, $X_i$ consists of matrices of the following forms:
$$X_j=\begin{pmatrix} 0&&&\\ a&0_{2k-2j-3}&& \\ b&c&0&0\end{pmatrix},\quad
X_{j+1}=\begin{pmatrix} 0&&&\\ d&0_{2k-2j-3}&& \\ e&f&0&\\
g&h&s&0\end{pmatrix}$$
where $a^t,d^t,c,f,h\in \Mat_{1\times (2k-2j-3)}$ and $b,e,g,s\in \Mat_1$. For $j+1<i<r_1$,  
\begin{equation}\label{mat70}\notag
X_i=\begin{pmatrix} 0&&\\ a&0_{2k-2j-2}& \\ b&c&0\end{pmatrix}\qquad  a^t, c\in \Mat_{1\times (2k-2j-2)}, b\in \Mat_1.
\end{equation}
Finally, we define the group $Q_{r_1}$ to be all the matrices $u_{2k-2j,n}'(0,Z)$ (see \eqref{mat5}) where $Z$
is of the form 
$$Z=\begin{pmatrix} 0&&\\ a&0_{2k-2j-2}&\\ b&a^*&0\end{pmatrix}\qquad a\in  \Mat_{1\times (2k-2j-2)},  b\in \Mat_1.$$

Let $M_0$ be the group $M_0=\prod_{i} M_i$.
Perform a root exchange similar to the one above. Using \cite{G-R-S4} Section 7, we deduce that the integral $L(j)$ is zero for all choices of data if and only if the integral
\begin{multline}\notag
\int
\int\limits_{M_0(F)\backslash M_0({\A})}\ \ \ 
\int\limits_{Q_{r_1}({\A})\ldots Q_2({\A})Q_1({\A})U_j(F)\backslash U_j({\A})}
\overline{\varphi^{(r)}(g)}\\
\theta_{2(n+k(r-1)-jr)}^{(r)}(um_0\iota_{3,j}(v,g))\psi_{U_j}(u)\psi_{V_{2k-2j},\delta}(v)\,du\,dm_0\,dv\,dg
\end{multline}
is zero for all choices of data.  Here, the domains of integration of the variables $v$ and $g$ are the same as for integral $L(j)$ above.

Let $V'_{2k-2j}$ be the unipotent radical  of the standard parabolic of $SO_{2k-2j}$ with Levi part $GL_1\times SO_{2k-2j-2}$, and let $\psi_{V_{2k-2j}'}$ be the Whittaker
character restricted to this subgroup: $\psi_{V_{2k-2j}'}(v')=\psi(v'_{1,2})$.  Let $\iota_3:SO_{2k-2j-2}\to SO_{2k-2j}$ be the embedding $\iota_3(v)=\text{diag}(1,v,1)$.
Then every matrix $v$ in $V_{2k-2}$ has a unique factorization $v=v'\iota_3(v'')$ with $v'\in V'_{2k-2j}$ and $v''\in V_{2k-2j-2}$.
Since $j<\text{min}\{k-2, r_1-1\}$, we have $2k-2j>4$. Hence we have the factorization $\psi_{V_{2k-2j},\delta}(v)=\psi_{V_{2k-2j}'}(v')\psi_{V_{2k-2j-2},\delta}(v'')$.

Define the group $M=M_0V'_{2k-2j}$. 
After changing variables, we obtain that $L(j)$ is zero for all choices of data if and only if the integral  
\begin{multline}\label{wh7}
\int
\int\limits_{V_{2k-2j-2}(F)\backslash V_{2k-2j-2}({\A})}\ \ 
\int\limits_{M(F)\backslash M({\A})}\ \ \ 
\int\limits_{Q_{r_1}({\A})\ldots Q_2({\A})Q_1({\A})U_j(F)\backslash U_j({\A})}
\overline{\varphi^{(r)}(g)}\\
\theta_{2(n+k(r-1)-jr)}^{(r)}(um\iota_{3,j}(v,g))\psi_{U_j}(u)\psi_{V_{2k-2j-2},\delta}(v)\psi_M(m)\,du\,dm\,dv\,dg
\end{multline}
is zero for all choices of data.
Here if $m\in M_j(\A)\subseteq M(\A)$ is the image of $u_{j+1}$ under the map \eqref{mat82} then $\psi_M(m)=\psi(u_{j+1}[2k-2j-1,2k-2j])$, and the character $\psi_M$ is extended trivially from $M_j(\A)$ to $M(\A)$.

The next step is to define certain monomial matrices $w_0^j\in Sp_{2(n+k(r-1)-jr)}$ whose non-zero entries are $\pm 1$. These matrices are of the form
\begin{equation}\label{weyl1}
w_0^j=\begin{pmatrix} w_1&&w_2\\ &I_{2n}&\\ w_3&&w_4\end{pmatrix},\qquad w_i\in \Mat_{k(r-1)-jr} \text{~for~}1\le i\le 4.
\end{equation}
Since these matrices are symplectic, it is enough to specify the non-zero entries in the matrix $\begin{pmatrix} w_1&w_2\end{pmatrix}$.  Let $\alpha=2n+k(r-1)-jr$.
The nonzero entries of the first $r$ rows of this matrix are as follows.   
The matrix has entry $1$ at position $(i, (i-1)(2k-2j-1)+1)$ for $1\le i\le j$; at position $(i,(2k-2j-1)j+(2k-2j)(i-j-1)+1)$ for $j+1\le i\le r_1$; 
at position $(r_1+i,\alpha+(i-1)(2k-2j)+1)$ for $1\le i\le r_1-j$;
at position $(r-j, \alpha +(r_1-j-1)(2k-2j)+2)$; and 
at position $(r-j+i,\alpha+(r_1-j)(2k-2j)+(i-1)(2k-2j-1)+1)$ for $1\le i\le j$.
The next $k(r-1)-(j+1)r$ rows of the matrix $\begin{pmatrix} w_1&w_2\end{pmatrix}$ are zero in positions in $w_2$.
As for the matrix $w_1$, let  
$w_1^0$ denote the matrix obtained from $w_1$ by omitting  the first $r$ rows. Then 
$$w_1^0= \begin{pmatrix} w_{1,1}^0&0\\ 0&w_{2,2}^0\end{pmatrix}\qquad
w_{1,1}^0\in \Mat_{(j+1)\beta\times (j+1)(\beta+2)}, \quad
w_{2,2}^0\in \Mat_{(r_1-j-1)\gamma\times ((r_1-j-1)(\gamma+2)+1)},$$
with $\beta=2k-2j-3$ and $\gamma=2k-2j-2$.  The matrix $w_{1,1}^0$ is given by
\begin{equation}\label{w-one-one-zero} w_{1,1}^0=\begin{pmatrix} 0&I_\beta&0&0&0_\beta&\ldots&&&&&&\\
0&0_\beta&0&0&I_\beta&0&0&0_\beta&\ldots&&&\\ 0&0_\beta&0&0&0_\beta&0&0&I_\beta&0&\ldots&&\\
\vdots&\vdots&\vdots&\vdots&\vdots&\vdots&\vdots&\vdots&\vdots&\ldots&\\
 0&0_\beta&0&0&0_\beta&0&0&0_\beta&0&\ldots&0&I_\beta&0\end{pmatrix},
\end{equation}
where this block matrix has $j+1$ rows each of height $\beta$,
 the unadorned $0$ is the zero matrix in $\Mat_{\beta\times 1}$, and the identity matrix $I_\beta$ appears in  the $i$-th row of the block matrix above in 
the $(3i-1)$-th column, $1\le i\le j+1$.  The matrix $w_{2,2}^0$ is given by 
$$w_{2,2}^0=\begin{pmatrix} 0&0&I_\gamma&0&0&0_\gamma&\ldots&&&&&&\\
0&0&0_\gamma&0&0&I_\gamma&0&0&0_\gamma&\ldots&&&\\ 0&0&0_\gamma&0&0&0_\gamma&0&0&I_\gamma&0&\ldots&&\\
 \vdots&\vdots&\vdots&\vdots&\vdots&\vdots&\vdots&\vdots&\vdots&\vdots&\ldots&\\ 
0&0&0_\gamma&0&0&0_\gamma&0&0&0_\gamma&0&\ldots&0&I_\gamma&0
\end{pmatrix}.
$$
This block matrix has $r_1-j-1$ rows each of height $\gamma$, each unadorned $0$ is the zero matrix in $\Mat_{\gamma\times 1}$,
and for $1\le i\le r_1-j-1$, the identity matrix  $I_\gamma$ appears in the $i$-th row in the $3i$-th column.

We remark that though $w_0^j$ is defined here for $j$ in the range $1\le j< (r-3)/2$, the same description makes sense for $j=(r-3)/2$.  We will make use
of $w_0^{(r-3)/2}$ in Section~\ref{case2} below.

Since the function $\theta_{2(n+k(r-1)-jr)}^{(r)}$ is invariant under $Sp_{2(n+k(r-1)-jr)}(F)$, we have 
$$\theta_{2(n+k(r-1)-jr)}^{(r)}(um\iota_{3,j}(v,g))=\theta_{2(n+k(r-1)-jr)}^{(r)}(w_0^jum\iota_{3,j}(v,g)(w_0^j)^{-1}w_0^j).$$
After conjugation, we deduce that the integral \eqref{wh7} is zero for all choices of data if and only if the integral
\begin{multline}\label{wh8}
\int 
\overline{\varphi^{(r)}(g)}
\theta_{2(n+k(r-1)-jr)}^{(r)}\left (\begin{pmatrix} A&B&C\\ &I_a&B^*\\ &&A^*\end{pmatrix}\begin{pmatrix} I_r&&\\ D&I_a&\\ E&D^*&I_r\end{pmatrix}u\iota_{3,j+1}(v,g)\right )\\ 
\widetilde{\psi}(A)\psi_{U_{j+1}}(u)\psi_{V_{2k-2j-2},\delta}(v)\psi'(D)\ d(...)
\end{multline}
is zero for all choices of data.  Here $a=2(n+k(r-1)-(j+1)r)$, and the domains of integration and characters in this integral are given as follows. 

First, the variables $v$ and $g$ are integrated  as in \eqref{wh7}.  
The variable $u$ is integrated over  $U_{j+1}(F)\backslash U_{j+1}({\A})$. The embedding  of these groups in $Sp_{2(n+k(r-1)-jr)}$ is given by 
$$u\iota_{3,j+1}(v,g)\to \text{diag}(I_r,u\iota_{3,j}(v,g),I_r).$$ This means that we can view these groups as subgroups of $Sp_{2(n+k(r-1)-(j+1)r)}$ embedded in
$Sp_{2(n+k(r-1)-jr)}$ by the map $h\to \text{diag}(I_r,h,I_r)$ (we do not introduce notation).  Also, let $L_r$ denote the upper triangular maximal unipotent subgroup of $GL_r$. 
Then the variable $A$ is integrated over $L_r(F)\backslash L_r({\A})$. The character $\widetilde{\psi}$ is the Whittaker character of the group $L_r$, given for $A=(A_{i,j})\in L_r$ by 
$\widetilde{\psi}(A)=\psi(A_{1,2}+A_{2,3}+\cdots +A_{r-1,r})$. 

To give the domain of integration of the variable $C$, for a positive integer $b$, let $T_b$ denote the group of all upper triangular matrices in $\Mat_b$. 
Let $T_{b,0}$ denote the subgroup of $T_b$ consisting of all upper triangular matrices with zero entries on the diagonal. Let $T_{b,0}^0=T_{b,0}\cap \Mat_b^0$. 
Let $C(r)$ denote the subgroup of $\Mat_r^0$ consisting of all matrices 
$$C=\begin{pmatrix} C_1&C_2&C_3\\ &C_4&C_2^*\\ &&C_1^*\end{pmatrix},\qquad
C_1\in T_j, C_2\in \Mat_{j\times (r-2j)}, C_3\in \Mat_j^0, 
C_4\in T_{r-2j,0}^0.$$
The variable $C$ in integral \eqref{wh8} is integrated over $C(r)(F)\backslash C(r)({\A})$. 

To give the integration domain of the variable $E$, let $T_{b,0,0}$ be the group 
consisting of all matrices in $T_{b,0}$ such that the all entries of the diagonal immediately above the main diagonal are zero. Let $E(r)$ be the group of all matrices in $\Mat_r^0$ of the form
$$E=\begin{pmatrix} E_1&E_2&E_3\\ &E_4&E_2^*\\ &&E_1^*\end{pmatrix},\qquad
E_1\in T_{j+1,0,0}, E_2\in \Mat_{(j+1)\times (r-2j-2)}, E_3\in \Mat_{j+1}^0, E_4\in T_{r-2j-2,0}^0.$$
Then $E$ is integrated over $E(r)(F)\backslash E(r)({\A})$. 

Let $B(r,a)$ denote the subgroup of $\Mat_{r\times a}$ consisting of all matrices $B=(B_{\alpha,\beta})$ such that $B_{\alpha,\beta}=0$ for
the following pairs of integers $(\alpha,\beta)$. First, if $1\le\alpha\le j+2$, then $1\le\beta\le (\alpha-1)(2k-2j-3)$. 
When $j+3\le\alpha\le (r-1)/2$ we have $1\le \beta\le (\alpha-1)(2k-2j-2)-(j+1)$. When $\alpha=(r+1)/2$ we have $1\le\beta\le n+a/2$. 
For $1+(r+1)/2\le\alpha\le r-j-2$ we have $1\le\beta\le (2\alpha-r-1)(k-j-1)+n+a/2$. When $\alpha=r-j-1, r-j$ we have $1\le \beta\le (r-2\alpha -3)(k-j-1)+n+a/2$, 
and finally, for $r-j+1\le\alpha\le r$ we have $1\le\beta\le (r-3-2j)(k-j-1)+(\alpha+j-r)(2k-2j-3)$. Then, the variable $B$ in integral \eqref{wh8} is integrated over the quotient $B(r,a)(F)\backslash B(r,a)({\A})$. 

Finally, to define the integration domain of the variable $D$, let $D(a,r)$ denote the subgroup of $\Mat_{a\times r}$ defined as follows. 
Given $B\in B(r,a)$, let $B'$ denote the matrix obtained from $B$ by omitting the last row. 
Then a matrix $D\in \Mat_{a\times r}$ is in $D(a,r)$ if all entries of its first column are zeros, and for {\sl{all}} $B\in B(r,a)$ we have $D'B'=0$. Then $D$ in integral \eqref{wh8} is  integrated over
$D(a,r)(F)\backslash D(a,r)({\A})$. (In \eqref{wh8} there is also a character on the group $D(a,r)$; however, since it will not be needed below, we will not specify it.)

The situation here is very similar to the integral studied in \cite{G-R-S3} Lemma 2.4, equation (2.4). 
As in that reference we now perform a root exchange in integral \eqref{wh8}, exchanging the non-trivial columns in the matrices $D$ and $E$ with corresponding rows in 
the matrices $B$ and $C$. We give some details. 

Let $I$ be the identity matrix of size $2(n+k(r-1)-jr)$, and consider the two unipotent groups 
$\{I+\sum_{i=1}^{2k-2j-3}m_ie'_{r+i,3}+m_{2k-2j-2}e'_{a+r+1,3}\}$ and
$\{I+\sum_{i=1}^{2k-2j-3}l_ie'_{2,r+i}+l_{2k-2j-2}e'_{2,a+r+1}\}.$   Notice that the first group is a subgroup of the group of matrices of the form
$$\begin{pmatrix} I_r&&\\ D&I_a&\\ E&D^*&I_r\end{pmatrix},\qquad
D\in D(a,r), E\in E(r).$$
Then the conditions of \cite{G-R-S4}, Lemma 7.1, are satisfied. We perform a root exchange between these two groups. 
Proceeding in this way, and using the vanishing of the Fourier coefficients of the representation $\Theta_{2(n+k(r-1)-jr)}^{(r)}$ given in Theorem~\ref{theta05}, part~\ref{theta05-1}, 
we deduce that the integral \eqref{wh8} is equal to
\begin{multline}\label{wh9}
\int 
\overline{\varphi^{(r)}(g)}
\theta_{2(n+k(r-1)-jr)}^{(r)}\left (\begin{pmatrix} A&B&C\\ &I_a&B^*\\ &&A^*\end{pmatrix}u\iota_{3,j+1}(v,g)\begin{pmatrix} I_r&&\\ D&I_a&\\ E&D^*&I_r\end{pmatrix}\right )\\ 
\widetilde{\psi}(A)\psi_{U_{j+1}}(u)\psi_{V_{2k-2j-2},\delta}(v)\psi'(D)\,d(...)
\end{multline}
where now the variable $D$ is integrated over $D(a,r)({\A})$, and $E$ is integrated over $E(r)({\A})$. 
Also, the variable $B$ is now integrated over $\Mat_{r\times a}(F)\backslash \Mat_{r\times a}({\A})$, and $C$ is integrated over $\Mat_r^0(F)\backslash \Mat_r^0({\A})$.
All other variables in \eqref{wh9} are integrated as in integral \eqref{wh8}. We remark that the conjugation of $u\iota_{3,j+1}(v,g)$ across the matrix involving $D$ and $E$ is possible 
since the corresponding groups normalize the group generated by the symplectic matrices involving these two variables. 

Applying Corollary 7.1 in \cite{G-R-S4}, we deduce that integral \eqref{wh8} is zero for all choices of data if and only if the integral 
\begin{equation}\label{wh10}
\int 
\overline{\varphi^{(r)}(g)}
\theta_{2(n+k(r-1)-jr)}^{(r)}\left (\begin{pmatrix} A&B&C\\ &I_a&B^*\\ &&A^*\end{pmatrix}u\iota_{3,j+1}(v,g)\right ) 
\widetilde{\psi}(A)\psi_{U_{j+1}}(u)\psi_{V_{2k-2j-2},\delta}(v)\,d(...)
\end{equation}
is zero for all choices of data. Here $B$ is integrated over $\Mat_{r\times a}(F)\backslash \Mat_{r\times a}({\A})$, and $C$ is integrated over $\Mat_r^0(F)\backslash \Mat_r^0({\A})$. Notice that the group generated by all matrices 
$$\begin{pmatrix} I_r&B&C\\ &I_a&B^*\\ &&I_r\end{pmatrix},\qquad
B\in \Mat_{r\times a}, C\in \Mat_r^0$$
is the unipotent radical of the maximal parabolic subgroup of $Sp_{2(n+k(r-1)-jr)}$ whose Levi part is $GL_r\times Sp_a$.
Hence, we can apply Proposition 1 in \cite{F-G1}. Since the theta representation of the group $GL_r^{(r)}({\A})$ is generic, it follows that integral \eqref{wh10} is zero for all choices of data if and only if the integral 
\begin{equation}\label{wh11}
\int 
\overline{\varphi^{(r)}(g)}
\theta_{2(n+k(r-1)-(j+1)r)}^{(r)} (u\iota_{3,j+1}(v,g)) 
\psi_{U_{j+1}}(u)\psi_{V_{2k-2j-2},\delta}(v)\,du\,dv\,dg
\end{equation}
is zero for all choices of data. Here all variables are integrated as in integral \eqref{wh10}. Since integral \eqref{wh11} is equal to $L(j+1)$ the Lemma follows.
\end{proof}

As mentioned before Lemma \ref{lem1}, the integral $L(0)$ is equal to \eqref{wh3}, which is zero for all choices of data if and only if the integral \eqref{wh1} is zero for all choices of data.
Thus it follows from Lemma \ref{lem1} that \eqref{wh1} is zero for all choices of data if and only if the integral $L(\text{min}(k-2,r_1-1))$ is zero for all choices of data. We analyze the two cases 
for this minimum separately.

\section{The case $k\le (r+1)/2$}\label{case1}
Before stating the result we will prove in this Section, we fix some notation.  Recall that the 
 unipotent groups $U_{2k,r_1,n}^{1,j}$ were defined for $0\le j\le \min(k-2,r_1-1)=k-2$ above (following \eqref{mat4}). 
 For the computations we will carry out now, we need to extend the definition to $j=k-1$ and introduce a suitable character of this group.
 
If $k<(r+1)/2$, then we define $U_{2k,r_1,n}^{1,k-1}$ to be the unipotent radical of the standard parabolic subgroup of $Sp_{2(n+r-k)}$ whose Levi part is 
 $GL_1^{k-1}\times GL_2^{r_1-k+1}\times Sp_{2n}$. (Note that $r_1-k+1>0$.) This unipotent group has the same factorization as in \eqref{mat5}, and once again 
 we define $U_{k-1}$ to be the subgroup of $U_{2k,r_1,n}^{1,k-1}$ consisting of all matrices of the form \eqref{mat5} such that $Y_2=0$. 
 For $u\in U_{k-1}$, the factorization \eqref{mat5} is given by
\begin{equation}\label{mat9}
u=u_{2,n}'(Y,Z)u_{k-2,1,2,d_{k-1}}^{k-1}(X_{k-1})\prod_{i=1}^{k-2} u_{i-1,1,1,d_i}^{i}(X_i)\prod_{i=k}^{r_1-1} u_{2i-k-1,2,2,d_i}^{1,i}(X_i)u_1
\end{equation}
with $Y=\left(\begin{smallmatrix} Y_1\\ 0 \end{smallmatrix}\right)$, $Y_1\in \Mat_{1\times n}$. Here each $d_j$ is defined so that the matrix is in $Sp_{2(n+r-k)}$.
The matrix $X_{k-1}$ has size $1\times 2$, and we write $X_{k-1}=\begin{pmatrix} x_1&x_2\end{pmatrix}$. 
Define a character of the group $U_{k-1}$ as follows. Given $u\in U_{k-1}$ with the factorization \eqref{mat9}, set
\begin{equation}\label{whca11}
\psi_{U_{k-1,\delta}}(u)=
\psi\Bigg(\delta^{-1} x_1+x_2 + \text{tr}'(Z) +
\sum_{\substack{{i=1}\\i\ne k-1}}^{r_1-1}\text{tr}(X_i)\Bigg).
\end{equation}
Here $\psi(\text{tr}'(Z))$ is defined as in \eqref{wh2}, and $\delta\in F^\times$ (see \eqref{wh1}). 

When $k=(r+1)/2$, we define $U_{2k,r_1,n}^{1,k-1}$ to be the unipotent radical of the standard parabolic subgroup of $Sp_{2(n+k-1)}$ whose 
Levi part is $GL_1^{k-1}\times Sp_{2n}$. The corresponding factorization is now
\begin{equation}\label{mat10}
u=u_{1,n}'(Y,Z)\prod_{i=1}^{k-2} u_{1,r_1,n}^{1,i}(X_i)u_1
\end{equation}
where $X_i$ and  $Z$ are scalars. The subgroup $U_{2k,r_1,n}^{0,k-1}$ consists of all matrices $u$ as in \eqref{mat10} such that $Y=0$; this is also $U_{k-1}$. 
We define the character $\psi_{U_{2k,r_1,n,\delta}^{0,k-1}}$ by
\begin{equation}\label{whca12}
\psi_{U_{2k,r_1,n,\delta}^{0,k-1}}(u)=
\psi(\delta^{-1} Z)\psi(X_1+X_2+\cdots +X_{k-2}).
\end{equation}

In this section we prove the following Lemma.
\begin{lemma}\label{lem2}
Suppose that $k\le (r+1)/2$. Then the integral \eqref{wh1} is zero for all choices of data if and only if the integral 
\begin{equation}\label{whca13}
\int\limits_{Sp_{2n}(F)\backslash Sp_{2n}({\A})}\  
\int\limits_{U_{2k,r_1,n}^{0,k-1}(F)\backslash U_{2k,r_1,n}^{0,k-1}({\A})}\overline{\varphi^{(r)}(g)}
\theta_{2(n+r-k)}^{(r)}(u\iota_{3,k-1}(1,g))\psi_{U_{2k,r_1,n,\delta}^{0,k-1}}(u)\,
du\,dg
\end{equation}
is zero for all choices of data.
\end{lemma}

\begin{proof}
Since $k\le (r+1)/2$, we have $k-2\le r_1-1$. Thus it follows 
from Lemma \ref{lem1} that the integral \eqref{wh1} is zero for all choices of data if and only if the integral $L(k-2)$ is zero for all choices of data. 
Substituting $j=k-2$, we may write $L(k-2)$ as
\begin{multline}\label{whca14}
\int\limits_{Sp_{2n}(F)\backslash Sp_{2n}({\A})}\ 
\int\limits_{V_4(F)\backslash V_4({\A})}\ \ 
\int\limits_{U_{2k,r_1,n}^{0,k-2}(F)\backslash U_{2k,r_1,n}^{0,k-2}({\A})}\overline{\varphi^{(r)}(g)}\\
\theta_{2(n+2r-k)}^{(r)}(u\iota_{3,k-2}(v,g))\psi_{U_{2k,r_1,n}^{0,k-2}}(u)
\psi_{V_{4},\delta}(v)\,du\,dv\,dg.
\end{multline}

The proof is now similar to the proof of Lemma \ref{lem1}. 
First, for $1\le i\le r_1$ define the groups $M_i$ and $Q_i$ as in the proof of Lemma \ref{lem1}, with $j=k-2$. 
As in that Lemma define $M=V'_4\prod_i M_i$. 
Thus $M$ consists of all matrices 
\begin{equation}\label{mat80}
\text{diag}(m_1,m_2,\ldots,m_{k-2},m_{k-1},\ldots,m_{r_1},I_{2n},m_{r_1}^*,\dots)
\end{equation}
where the matrix $m_i$ (for $1\le i\le r_1$) is of the form
$$m_i=\begin{pmatrix} 1&a&b\\ &I_{\alpha-2}&c\\ &&1\end{pmatrix}\ \ \ \ 
a,c^t\in \Mat_{1\times (\alpha-2)};\ \ b\in \Mat_{1\times 1}$$
with $\alpha=3$ for $1\le i\le k-2$, $\alpha=4$ for $k-1\le i\le r_1$. 
Performing similar root exchanges, we deduce that the integral \eqref{whca14} is zero for all choices of data if and only if the integral  
\begin{multline}\label{whca15}
\int\limits_{Sp_{2n}(F)\backslash Sp_{2n}({\A})}\ 
\int\limits_{M(F)\backslash M({\A})}\ \ 
\int\limits_{Q_{r_1}({\A})\ldots Q_1({\A})U_{2k,r_1,n}^{0,k-2}(F)\backslash U_{2k,r_1,n}^{0,k-2}({\A})}\overline{\varphi^{(r)}(g)}\\
\theta_{2(n+2r-k)}^{(r)}(um\iota_{3,k-2}(1,g))\psi_{U_{2k,r_1,n}^{0,k-2}}(u)
\psi_{M,\delta}(m)\,du\,dm\,dg
\end{multline}
is zero for all choices of data.
The character $\psi_{M,\delta}$ is non-trivial only on the variable $m_{k-1}$ in \eqref{mat80}, and if
$$m_{k-1}=\begin{pmatrix} 1&a&b&c\\ &1&&d\\ &&1&e\\ &&&1\end{pmatrix}\in GL_4(\A)$$
then $\psi_{M,\delta}(m)=\psi_{M,\delta}(m_{k-1})=\psi(d+\delta e)$.

The next step is to define a Weyl element $w_0^{k-2}$ as in \eqref{weyl1}. In this case the $w_i$ are matrices in $\Mat_{2r-k}$, and it is enough to specify the 
matrix $\begin{pmatrix} w_1&w_2\end{pmatrix}$. We first describe the first $r$ rows of this matrix. This matrix has the value $1$ at the locations $(i,3(i-1)+1)$ for $1\le i\le k-2$;
$(i,4i-k-1)$ for $k-1\le i\le r_1$; $(i, 2n-k+4i-1)$ for $r_1+1\le i\le r-k+1$; $(r-k+2, 2n+4r-5k+5)$; and at $(i,2n+r-2k+3i-1)$ for $r-k+3\le i\le r$.  All other entries in these rows are $0$.
For the next $r-k$ rows of the matrix $\begin{pmatrix} w_1&w_2\end{pmatrix}$ the entries in $w_2$ are all $0$. As for $w_1$, let  
$w_1^0$ denote the matrix obtained from $w_1$ by omitting  the first $r$ rows. Then we choose
$$w_1^0= \begin{pmatrix} w_{1,1}^0&0\\ 0&w_{2,2}^0\end{pmatrix},\qquad
w_{1,1}^0\in \Mat_{(k-1)\times 3(k-1)},\quad
w_{2,2}^0\in \Mat_{(r-2k+1)\times (2r-4k+3)}.$$
For $1\le i\le k-1$, the matrix $w_{1,1}^0$ has the value $1$ at the $(i,3i)$ locations and $0$ otherwise. If $k<(r+1)/2$, the matrix $w_{2,2}^0$ is given by
$$w_{2,2}^0=\begin{pmatrix} 0&0&I_2&0&0&0&\ldots&&&&&&\\
0&0&0&0&0&I_2&0&0&0&\ldots&&&\\ 0&0&0&0&0&0&0&0&I_2&0&\ldots
&&\\ \vdots&\vdots&\vdots&\vdots&\vdots&\vdots&\vdots&\vdots&\vdots&\vdots&\ldots&\\ 0&0&0&0&0&0&0&0&0&0&\ldots&0&I_2&0
\end{pmatrix}.
$$
This block matrix has $r_1-k+1$ rows, and the identity matrix  $I_2$ at the $i$-th row is in the $3i$-th column. 
Each zero represents the zero matrix in $\Mat_{2\times 1}$ or $\Mat_{2\times 2}$. When $k=(r+1)/2$, then $w_1^0=\begin{pmatrix} w_{1,1}^0&0\end{pmatrix}$ 
where the zero represents the zero matrix in $\Mat_{(k-1)\times 1}$.

We use the left invariance property of the function $\theta_{2(n+2r-k)}^{(r)}$ to conjugate by the 
Weyl element $w_0^{k-2}$. In a similar way to arriving at \eqref{wh8}, we deduce that the integral \eqref{whca15} is zero for all choices of data if and only if the integral
\begin{multline}\label{whca16}
\int 
\overline{\varphi^{(r)}(g)}
\theta_{2(n+2r-k)}^{(r)}\left (\begin{pmatrix} A&B&C\\ &I_a&B^*\\ &&A^*\end{pmatrix}\begin{pmatrix} I_r&&\\ D&I_a&\\ E&D^*&I_r\end{pmatrix}u\iota_{3,k-2}(1,g)\right )\\ 
\widetilde{\psi}(A)\psi_{U_{2k,r_1,n,0}^{0,k-1}}(u)\psi_{1,\delta}(B)\psi'(D)\,d(...)
\end{multline}
is zero for all choices of data. Here $a=2(n+r-k)$. The variable  $g$ is integrated as in \eqref{whca15}, and the 
variable $u$ is integrated over the quotient $U_{2k,r_1,n}^{0,k-1}(F)\backslash U_{2k,r_1,n}^{0,k-1}({\A})$.
The character $\psi_{U_{2k,r_1,n,0}^{0,k-1}}$ is defined  by 
$$\psi_{U_{2k,r_1,n,0}^{0,k-1}}(u)=\begin{cases}
\psi\Big(x_2 + \text{tr}'(Z) +
\sum_{\underset{i\ne k-1}{i=1}}^{r_1-1}\text{tr}(X_i)\Big)&\text{if $k<(r+1)/2$}\\ 
\psi(X_1+X_2+\cdots +X_{k-2})&\text{if $k=(r+1)/2$.}\end{cases}$$
Here the notation is as in  \eqref{whca11} and \eqref{whca12}. The variable $A$ and the character $\widetilde{\psi}(A)$ are 
as in \eqref{wh8}, and the variables $B,C,D$ and $E$ are defined in a similar way to integral \eqref{wh8}. 
If $B=(B_{\alpha,\beta})\in \Mat_{r\times a}(\A)$ then $\psi_{1,\delta}(B)=\psi(\delta B_{r-k+1, 2n+2r-3k+2})$.
If $D=(D_{\alpha,\beta})\in \Mat_{a\times r}(\A)$ then $\psi'(D)=\psi(D_{2n+2r-3k+1, r-k+2})$
when $k<(r+1)/2$ and $\psi'(D)=\psi(D_{k-1, r-k+2})$ when $k=(r+1)/2$.

At this point the argument deviates from the prior case, more precisely in handling column $r-k+2$ of $D$, 
as the character $\psi_{1,\delta}(B)$ is not trivial 
on row $r-k+1$ of $B$. We proceed as follows. Define the matrix $x_1(\delta)=I_{2(n+2r-k)}+\delta^{-1}e'_{2n+3r-3k+2, r-k+2}$ in $Sp_{2(n+2r-k)}(F)$. 
By automorphicity, we have $\theta_{2(n+2r-k)}^{(r)}(h)=\theta_{2(n+2r-k)}^{(r)}(x_1(\delta)h)$ for all $h\in  Sp^{(r)}_{2(n+2r-k)}({\A})$. 
Conjugating the matrix $x_1(\delta)$ to the right and changing variables, integral \eqref{whca16} is equal to
\begin{multline}\label{whca17}
\int 
\overline{\varphi^{(r)}(g)}
\theta_{2(n+2r-k)}^{(r)}\left (\begin{pmatrix} A&B&C\\ &I_a&B^*\\ &&A^*\end{pmatrix}\begin{pmatrix} I_r&&\\ D&I_a&\\ E&D^*&I_r\end{pmatrix}u\iota_{3,k-2}(1,g)x_1(\delta)\right )\\ 
\widetilde{\psi}(A)\psi_{U_{2k,r_1,n,\delta}^{0,k-1}}(u)\psi'(D)\,d(...).
\end{multline}
Notice that here the character $\psi_{1,\delta}(B)$ has been omitted (it has been cancelled out), and also  $\psi_{U_{2k,r_1,n,0}^{0,k-1}}$ has been replaced by $\psi_{U_{2k,r_1,n,\delta}^{0,k-1}}$. 
Both of these changes are the result of the change of variables required to move the matrix $x_1(\delta)$ to the right.

We may now perform root exchanges similar to those carried out in analyzing \eqref{wh8}. 
As in that case, we proceed by the columns of the matrices $D$ and $E$ starting from the first non-zero column and then in increasing order. 
Using the smallness of the representation $\Theta_{2(n+2r-k)}^{(r)}$, i.e.\ Theorem~\ref{theta05}, part~\ref{theta05-1}, 
we deduce that the integral \eqref{whca17} is zero for all choices of data if and only if the integral 
\begin{equation*}
\int 
\overline{\varphi^{(r)}(g)}
\theta_{2(n+2r-k)}^{(r)}\left (\begin{pmatrix} A&B&C\\ &I_a&B^*\\ &&A^*\end{pmatrix}u\iota_{3,k-2}(1,g)\right ) 
\widetilde{\psi}(A)\psi_{U_{2k,r_1,n,\delta}^{0,k-1}}(u)\,d(...)
\end{equation*}
is zero for all choices of data. Here, $B$ is integrated over the quotient $\Mat_{r\times a}(F)\backslash \Mat_{r\times a}({\A})$, 
and $C$ is integrated over  $\Mat_r^0(F)\backslash \Mat_r^0({\A})$.  Now arguing as above (after \eqref{wh10}), the Lemma follows.
\end{proof}

Now we present a criterion for the vanishing of all Whittaker coefficients ${\mathcal W}_{k,\delta}(f)$. 
This criterion depends on the Descent Conjecture, which was described briefly in the Introduction.
To avoid disrupting the continuity of the proofs, we do not give additional details about it now, but defer its formulation and discussion to Section~\ref{4G}, Conjecture~\ref{conj20} below.
For the criterion, we require the unipotent group $U_{1,r_1-k+1,n}$ and its character $\psi_{U_{1,r_1-k+1,n}}$, defined in Section~\ref{basic}.
Also, we now write $l_n$ for the projection from the group $U_{1,r_1-k+1,n}$ onto the Heisenberg group ${\mathcal H}_{2n+1}$.
The criterion is this.
\begin{proposition}\label{prop0} Suppose that the Descent Conjecture holds.
Suppose that $k\le (r+1)/2$. Then the Whittaker coefficients ${\mathcal W}_{k,\delta}(f)$ are zero for all choices of data $f$ if and only if the integral 
\begin{multline}\label{whca19}
\int\limits_{Sp_{2n}(F)\backslash Sp_{2n}({\A})}\  
\int\limits_{U_{1,r_1-k+1,n}(F)\backslash U_{1,r_1-k+1,n}({\A})}\overline{\varphi^{(r)}(g)}
\theta_{2n}^{(2),\psi^\delta}(l_{2n}(u)g)\\\theta_{2n+r-2k+1}^{(2r),\psi^\delta}\left(u\begin{pmatrix}I_{\tfrac{r+1}{2}-k}&&\\&g&\\&&I_{\tfrac{r+1}{2}-k}\end{pmatrix}\right)\psi_{U_{1,r_1-k+1,n}}(u)\,
du\,dg
\end{multline}
is zero for all choices of data.
\end{proposition}
Equivalently, the Whittaker coefficient \eqref{wh1} is not zero for some choice of data if and only if the integral \eqref{intro4} that appears in Proposition~\ref{general-two} is not zero for some choice of data.

\begin{proof}
By Lemma~\ref{lem2}, it suffices to prove that the integral \eqref{whca13} is zero for all choices of data if and only if the integral \eqref{whca19} is zero for all choices of data. 

Suppose first that $k<(r+1)/2$. Let $a=2(n+r-k)$ and let $x(\delta)$ denote the 
unipotent element  $x(\delta)=I_a-\sum_{i=1}^{r_1-k}\delta^{-1}e'_{k+2i,k+2i+1}\in Sp_{2(n+r-k)}(F)$. By automorphicity, 
the function $\theta_{2(n+r-k)}^{(r)}(h)$ is left-invariant under this element. 
Moving this matrix from left to right and changing variables in $U_{2k,r_1,n}^{0,k-1}$, it follows that  \eqref{whca13} is zero for all choices of data if and only if the integral 
\begin{equation}\label{whca110}
\int\limits_{Sp_{2n}(F)\backslash Sp_{2n}({\A})}\  
\int\limits_{U_{2k,r_1,n}^{0,k-1}(F)\backslash U_{2k,r_1,n}^{0,k-1}({\A})}\overline{\varphi^{(r)}(g)}
\theta_{2(n+r-k)}^{(r)}(u\iota_{3,k-1}(1,g))\psi'_{U_{2k,r_1,n,\delta}^{0,k-1}}(u)\,du\,dg
\end{equation}
is zero for all choices of data. Here the character $\psi'_{U_{2k,r_1,n,\delta}^{0,k-1}}$ is defined, using the factorization \eqref{mat9}, by
\begin{equation}\label{whca111}\notag
\psi'_{U_{2k,r_1,n,\delta}^{0,k-1}}(u)=
\psi\Bigg(x_2 + \text{tr}'_\delta(Z) +
\sum_{\underset{i\ne k-1}{i=1}}^{r_1-1}\text{tr}(X_i)\Bigg)
\end{equation}
where for $Z=\begin{pmatrix} z_1&z_2\\ z_3&z_1\end{pmatrix}\in \Mat^0_2$, we define $\text{tr}'_\delta(Z)=z_1+\delta^{-1}z_3$.

Let $w\in Sp_{2(n+r-k)}$ be given by $w=\text{diag}(I_{k-1},w_1,I_{2n},w_1^*,I_{k-1})$ where $w_1\in \Mat_{r-2k+1}(F)$ is given as follows. 
Write $w_1=\begin{pmatrix} w_{1,1}\\ w_{1,2}\end{pmatrix}$ where $w_{1,i}\in \Mat_{(r_1-k+1)\times (r-2k+1)}(F)$. 
The matrix $w_{1,1}$ has a $1$ at the positions $(i,2i)$ and $w_{1,2}$ has a $1$ at the positions $(i,2i-1)$ for each $i$, $1\le i\le r_1-k+1$, and all other entries of the matrix $w_1$ are $0$.

Conjugating in \eqref{whca110} by $w$, we deduce that this integral is zero for all choices of data if and only if the integral
\begin{multline}\label{whca112}
\int 
\overline{\varphi^{(r)}(g)}
\theta_{2(n+r-k)}^{(r)}\left (\begin{pmatrix} A&B&C\\ &I_a&B^*\\ &&A^*\end{pmatrix}\begin{pmatrix} I_{r_1}&&\\ D&I_a&\\ &D^*&I_{r_1}\end{pmatrix}\begin{pmatrix} I_{r_1}&&\\ &u&\\ &&I_{r_1}\end{pmatrix}
\begin{pmatrix}I_{r-k}&&\\&g&\\&&I_{r-k}\end{pmatrix}\right )\\ 
\widetilde{\psi}(A)\psi_{U_{1,r_1-k+1,n}}(u)\psi'(B)\psi_\delta(C)\,d(...)
\end{multline}
is zero for all choices of data.
Here $a=2n+r-2k+1$. The domain of integration and characters here are described as follows. The variable $u$ is integrated over  $U_{1,r_1-k+1,n}(F)\backslash U_{1,r_1-k+1,n}({\A})$, and
$g$ is integrated over $Sp_{2n}(F)\backslash Sp_{2n}({\A})$. The variable $A$ is integrated over $L_{r_1}(F)\backslash L_{r_1}({\A})$, where we recall
that $L_{r_1}$ is the maximal upper unipotent subgroup of $GL_{r_1}$. The character $\widetilde{\psi}(A)$ is the Whittaker character, defined following \eqref{wh8}. Let $B({r_1},a)$ denote the 
subgroup of $\Mat_{{r_1}\times a}$ of matrices  $B=(B_{\alpha,\beta})$ such that $B_{\alpha,\beta}=0$ for all $k-1\le\alpha\le r_1-1$ 
and $1\le\beta\le \alpha-k+2$, and for $\alpha={r_1}$ and $1\le\beta\le 2n+r_1-k+1$. Then the variable $B$ is integrated over 
$B({r_1},a)(F)\backslash B({r_1},a)({\A})$. The character $\psi'(B)$ is given by $\psi'(B)=\psi(B_{r_1,2n+r_1-k+2})$. The variable $C$ is integrated over 
$\Mat_{r_1}^0(F)\backslash \Mat_{r_1}^0({\A})$, and the character $\psi_\delta$ is given by $\psi_\delta(C)=\psi(\delta^{-1}C_{r_1,1})$. Finally, let $D(a,{r_1})$ denote the 
subgroup of $\Mat_{a\times {r_1}}$ consisting of all matrices of the form
$D=\begin{pmatrix} 0&D'\\ 0&0\end{pmatrix}$ such that $D'\in \Mat_{r_1-k+1}$ and $D'_{\alpha,\beta}=0$ for all $1\le \alpha\le r_1-k+1$ and $\beta\le\alpha$. 
Then $D$ is integrated over $D(a,{r_1})(F)\backslash D(a,{r_1})({\A})$.

Next, we perform a root exchange that is similar to the one performed before \eqref{wh9}. 
This implies that the integral \eqref{whca112} is zero for all choices of data if and only if the integral
\begin{multline}\label{whca113}
\int 
\overline{\varphi^{(r)}(g)}
\theta_{2(n+r-k)}^{(r)}\left(\begin{pmatrix} A&B&C\\ &I_a&B^*\\ &&A^*\end{pmatrix}\begin{pmatrix} I_{r_1}&&\\ &u&\\ &&I_{r_1}\end{pmatrix}
\begin{pmatrix}I_{r-k}&&\\&g&\\&&I_{r-k}\end{pmatrix}\right )\\
\widetilde{\psi}(A)\psi_{U_{1,r_1-k+1,n}}(u)\psi'(B)
\psi_\delta(C)\,d(...)
\end{multline}
is zero for all choices of data. Here the variable $B$ is integrated over  $\widetilde{B}({r_1},a)(F)\backslash \widetilde{B}({r_1},a)({\A})$ where $\widetilde{B}({r_1},a)$ is the 
subgroup of $\Mat_{{r_1}\times a}$ consisting of all matrices $B=(B_{\alpha,\beta})$ such that $B_{{r_1},\beta}=0$ for all $1\le\beta\le 2n+r_1-k+1$.
All other variables are integrated as in \eqref{whca112}.

Write $\widetilde{B}({r_1},a)=\widetilde{B}_1({r_1},a)\widetilde{B}_2({r_1},a)$ where $\widetilde{B}_1({r_1},a)$ consists of all matrices in $\widetilde{B}({r_1},a)$ with bottom row zero, and 
$\widetilde{B}_2({r_1},a)$ consists of all matrices $(B_{\alpha,\beta})\in \widetilde{B}({r_1},a)$ such that $B_{\alpha,\beta}=0$ for all $\alpha$ with $1\le \alpha\le r_1-1$, and for $\alpha=r_1$ and all 
$\beta$ with $1\le \beta\le 2n+r_1-k+1$. 
Recall that the group $U_{1,r_1,n-k+r_1+1}$ was defined in Section~\ref{basic} and 
each element $u\in U_{1,r_1,n-k+r_1+1}$ has a factorization \eqref{mat3}. Let $U_{1,r_1,n-k+r_1+1}^0$ denote the subgroup of 
$U_{1,r_1,n-k+r_1+1}$ consisting of $u$ such that in the factorization \eqref{mat3}, $Y=0$. 
It is not hard to check that this is the subgroup of $Sp_{2(n+r-k)}$ consisting of all matrices 
\begin{equation}\label{whca114} v=
\begin{pmatrix} A&B&C\\ &I_a&B^*\\ &&A^*\end{pmatrix},\qquad
A\in L_{r_1},\quad B\in \widetilde{B}_1({r_1},a),\quad
\ C\in \Mat_{r_1}^0.
\end{equation}
Thus integral \eqref{whca113}
is equal to
\begin{multline}\label{whca115}
\int 
\overline{\varphi^{(r)}(g)}
\theta_{2(n+r-k)}^{(r)}\left (v\begin{pmatrix} I_{r_1}&B&\\ &I_a&B^*\\ &&I_{r_1}\end{pmatrix}\begin{pmatrix} I_{r_1}&&\\ &u&\\ &&I_{r_1}\end{pmatrix}
\begin{pmatrix}I_{r-k}&&\\&g&\\&&I_{r-k}\end{pmatrix}\right )\\
\psi_{U_{1,r_1,n-k+r_1+1}^{0,\delta}}(v)
\psi_{U_{1,r_1-k+1,n}}(u)\psi'(B)\,dv\,dB\,du\,dg.
\end{multline}
Here $v$ is integrated over $U_{1,r_1,n-k+r_1+1}^0(F)\backslash U_{1,r_1,n-k+r_1+1}^0({\A})$ and 
$$\psi_{U_{1,r_1,n-k+r_1+1}^{0,\delta}}(v)=
\widetilde{\psi}(A)\psi_\delta(C)$$ on matrices $v$ given by \eqref{whca114}. 
The variable $B$ is integrated over $\widetilde{B}_2({r_1},a)(F)\backslash 
\widetilde{B}_2({r_1},a)({\A})$ and the character $\psi'$ is the restriction of the character of $B({r_1},a)$ to the subgroup 
$\widetilde{B}_2({r_1},a)$. The variables $u$ and $g$ are integrated as before. 

We observe that the set of all matrices of the form 
$$\begin{pmatrix} I_{r_1}&B&\\ &I_a&B^*\\ &&I_{r_1}\end{pmatrix}\begin{pmatrix} I_{r_1}&&\\ &u&\\ &&I_{r_1}\end{pmatrix}
\begin{pmatrix}I_{r-k}&&\\&g&\\&&I_{r-k}\end{pmatrix}$$
with $B\in \widetilde{B}_2({r_1},a), u\in U_{1,r_1-k+1,n},
g\in Sp_{2n}$
is a subgroup of ${\mathcal H}_{a+1}\rtimes Sp_a$, and the integration in \eqref{whca115} over the adelic quotient $U_{1,r_1,n-k+r_1+1}^0(F)\backslash U_{1,r_1,n-k+r_1+1}^0({\A})$ is a 
Fourier coefficient which corresponds to the unipotent orbit $((r-1)1^a)$. Thus we may apply Ikeda's work on Fourier-Jacobi coefficients \cite{I1}. This implies that the space of functions
\begin{multline}\label{whca116}
(B,u,g)\to\theta_{2n+r-2k+1}^{(2),\psi^\delta}
\left(l_{2n+r-2k+1}(B)u\begin{pmatrix}I_{\tfrac{r+1}{2}-k}&&\\&g&\\&&I_{\tfrac{r+1}{2}-k}\end{pmatrix}\right)\\
\int\limits_{U_{1,r_1,n+r_1-k}(F)\backslash U_{1,r_1,n+r_1-k}({\A})}
\overline{\theta_{2n+r-2k+1}^{(2),\psi^\delta}}\left(l_{2n+r-2k+1}(v')u\begin{pmatrix}I_{\tfrac{r+1}{2}-k}&&\\&g&\\&&I_{\tfrac{r+1}{2}-k}\end{pmatrix}\right)\\
\theta_{2(n+r-k)}^{(r)}\left(v'u\begin{pmatrix}I_{r-k}&&\\&g&\\&&I_{r-k}\end{pmatrix}\right)\psi_{U_{1,r_1,n+r_1-k}}(v')\,dv'
\end{multline}
is a dense subspace of the space of functions 
\begin{equation}\label{whca117}
(B,u,g)\to \int
\theta_{2(n+r-k)}^{(r)}\left (v\begin{pmatrix} I_{r_1}&B&\\ &I_a&B^*\\ &&I_{r_1}\end{pmatrix}\begin{pmatrix} I_{r_1}&&\\ &u&\\ &&I_{r_1}\end{pmatrix} \begin{pmatrix}I_{r-k}&&\\&g&\\&&I_{r-k}\end{pmatrix}\right)
\psi_{U_{1,r_1,n-k+r_1+1}^{0,\delta}}(v)\,dv
\end{equation}
where $v$ is integrated as in \eqref{whca115}.
It follows from Conjecture~\ref{conj20} that, as a function of $ug$, the integral in \eqref{whca116} is equal to a function of the form
$$\theta_{2n+r-2k+1}^{(2r),\psi^\delta}\left(u\begin{pmatrix}I_{\tfrac{r+1}{2}-k}&&\\&g&\\&&I_{\tfrac{r+1}{2}-k}\end{pmatrix}\right)$$
with $\theta_{2n+r-2k+1}^{(2r),\psi^\delta}$ in the space attached to the corresponding  theta representation.

We deduce that the integral \eqref{whca115} is zero for all choices of data if and only if the integral
\begin{multline}\label{whca118}
\int\limits_{Sp_{2n}(F)\backslash Sp_{2n}({\A})}\  
\int\limits_{U_{1,r_1-k+1,n}(F)\backslash U_{1,r_1-k+1,n}({\A})}\overline{\varphi^{(r)}(g)}\\
\theta_{2n+r-2k+1}^{(2r),\psi^\delta}\left(u\begin{pmatrix}I_{\tfrac{r+1}{2}-k}&&\\&g&\\&&I_{\tfrac{r+1}{2}-k}\end{pmatrix}\right)\psi_{U_{1,r_1-k+1,n}}(u)\\
\int\limits_{\widetilde{B}_2({r_1},a)(F)\backslash 
\widetilde{B}_2({r_1},a)({\A})}
\theta_{2n+r-2k+1}^{(2),\psi^\delta}
\left(l_{2n+r-2k+1}(B)u\begin{pmatrix}I_{\tfrac{r+1}{2}-k}&&\\&g&\\&&I_{\tfrac{r+1}{2}-k}\end{pmatrix}\right))\psi'(B)\,dB\,du\,dg
\end{multline}
is zero for all choices of data. 

To complete the proof of the Proposition we need to compute the inner integration in \eqref{whca118}. 
We do this by unfolding the theta series. Recall that the action of the Weil representation is normalized as in \cite{G-R-S2}, Section 1, part {\bf 6}. From the definition of 
$\widetilde{B}_2({r_1},a)$ we deduce that 
$$l_{2n+r-2k+1}(B)=(0_{2n+r_1-k+1},B_{r_1,2n+r_1-k+2},B_{r_1,2n+r_1-k+3},\ldots,B_{r_1,2n+r-2k+2},0).$$
With this notation, $\psi'(B)=\psi(B_{r_1,2n+r_1-k+2})$. 
Unfolding the theta series, the inner integration in integral \eqref{whca118} is equal to
\begin{multline}\label{whca119}
\int\sum_{\xi\in F^{n+r_1-k+1}}\omega_{\psi^\delta}\left(l_{2n+r-2k+1}(B)u\begin{pmatrix}I_{\tfrac{r+1}{2}-k}&&\\&g&\\&&I_{\tfrac{r+1}{2}-k}\end{pmatrix}\right) \phi(\xi)\psi'(B)\,dB\\
=\sum_{\xi\in F^n}\omega_{\psi^\delta}\left(u\begin{pmatrix}I_{\tfrac{r+1}{2}-k}&&\\&g&\\&&I_{\tfrac{r+1}{2}-k}\end{pmatrix}\right)\phi(0_{r_1-k},\delta,\xi)
\end{multline}
where we have used the action of the Weil representation to derive this last equality. However, for an appropriate Schwartz function, the right hand side of 
equation \eqref{whca119} is equal to $\theta_{2n}^{(2),\psi^\delta}(l_{2n}(u)g)$. This completes the proof of the Proposition for the case when $k<(r+1)/2$. 

When $k=(r+1)/2$, the proof is simpler. Starting with integral \eqref{whca13}, from the definition of the character \eqref{whca12} we see that no conjugation by a 
Weyl element is needed. The first step is to use the result of \cite{I1} as described in 
\eqref{whca116} and  \eqref{whca117}. Notice that in this case $B=0$. Thus we obtain \eqref{whca19} as claimed.
\end{proof}

\section{The case $k> (r+1)/2$}\label{case2}
We study this case using a similar approach to the case $k\leq (r+1)/2$.  First, for $0\le j\le \min(n,k-(r+1)/2)$ we define a family of integrals denoted by ${\mathcal L}(j)$. 
Then we prove that ${\mathcal L}(j)$ is zero for all choices of data if and only if ${\mathcal L}(j+1)$ is zero for all choices of data. 
However, this case requires a new ingredient, namely the Fourier expansion along a unipotent subgroup of a symplectic group, and makes use of the cuspidality of $\pi$ in an essential way.

We start by fixing some notation. Fix an integer $j$ in the range $1\le j\le \min(n,k-(r+1)/2)$.  Recall that the group $U_{n,j}$ and its character $\psi_{U_{n,j}}$ were defined in Section~\ref{intro}.
We let $l_{2n-2j}$ denote the homomorphism from $U_{n,j}$ onto ${\mathcal H}_{2n-2j+1}$, the Heisenberg group of $2n-2j+1$ variables, given 
on $u=(u_{\alpha,\beta})\in U_{n,j}$ by 
\begin{equation}\label{the-hom-ell}l_{2n-2j}(u)=
\begin{cases}(u_{j,j+1},u_{j,j+2},\ldots,u_{j,2n-2j},u_{j,2n-2j+1})&\text{for $j<n$}\\u_{n,n+1}&\text{for $j=n$.}\end{cases}
\end{equation}
We identify ${\mathcal H}_{2n-2j+1}$ with an upper subgroup of $Sp_{2n-2j+2}$ as in \cite{G-R-S4}, p.\ 8. 

Recall that $L_{2k-r-2j+2}$ denotes the maximal unipotent subgroup of $GL_{2k-r-2j+2}$ and 
$V_{2k-r-2j+1}$ denotes the maximal unipotent subgroup of $SO_{2k-r-2j+1}$.
Let $V_{0,2k-r-2j+2}$ be the subgroup of $L_{2k-r-2j+2}$ generated by all matrices of the form
\begin{equation}\label{end0}
v_0=\begin{pmatrix} 1&a\\ &v\end{pmatrix},\qquad a\in \Mat_{1\times (2k-r-2j+1)},  v\in V_{2k-r-2j+1}.
\end{equation}
Let $\psi_{V_{0,2k-r-2j+2},\delta}$ be the character of $V_{0,2k-r-2j+2}(F)\backslash V_{0,2k-r-2j+2}(\A)$ which is the trivial extension of the Whittaker character of $V_{2k-r-2j+1}(\A)$, i.e.\ $\psi_{V_{0,2k-r-2j+2},\delta}(v_0)=\psi_{V_{2k-r-2j+2},\delta}(v)$. 
Also, let $\psi'_{0,V_{2k-r-2j},\delta}$ be the character of the group $V_{0,2k-r-2j}(F)\backslash V_{0,2k-r-2j}(\A)$ defined as follows. Write $v_0$ as in 
\eqref{end0} with $j+1$ in place of $j$. Then we set
\begin{equation}\label{psi-prime}
\psi'_{0,V_{2k-r-2j},\delta}(v_0)=\psi(a_{1,1})\psi_{V_{2k-r-2j-1},\delta}(v).
\end{equation}

Finally, we define a unipotent subgroup of a symplectic group, denoted $U_{0,j}$, and a character $\psi_{U_{0,j}}$ of the adelic quotient
of this group, as follows. Consider the unipotent group $U_{a,b,c}$ where $a=2k-r-2j+2, b=r_1$ and $c=n-j+1$. 
For $u\in U_{a,b,c}$, we have the factorization \eqref{mat3}. Let $R$ denote the subgroup of $\Mat_{(2k-r-2j+2)\times 2(n-j+1)}$ of all matrices 
$Y=(Y_{\alpha,\beta})\in \Mat_{(2k-r-2j+2)\times 2(n-j+1)}$ such that $Y_{\alpha,\beta}=0$ for all $(2k-r-2j+5)/2\le \alpha\le 2k-r-2j+2$ and $1\le \beta\le 2n-2j+1$, and such that
$Y_{\alpha,1}=0$ for all $2\le\alpha\le (2k-r-2j+3)/2$. 
We define $U_{0,j}$ to be the subgroup of $U_{a,b,c}$ such that if $u$ is factored as in \eqref{mat3}, then $Y\in R$.
Also, for $Y\in R(\A)$ let $\psi_R(Y)=\psi(Y_{1,1}+Y_{2k-r-2j+2,2n-2j+2})$ and for 
$Z=(Z_{\alpha,\beta})\in \Mat_{2k-r-2j+2}^0(\A)$, let $\psi_0(Z)=\psi(Z_{2,1}+\cdots +Z_{(2k-r-2j+3)/2,(2k-r-2j+1)/2})$.
Then we let $\psi_{U_{0,j}}$ be the character of $U_{0,j}(F)\backslash U_{0,j}(\A)$ given by
\begin{equation}\label{psi-111}
\psi_{U_{0,j}}(u)=\psi(\text{tr}(X_1+\cdots +X_{r_1}))\psi_R(Y)\psi_0(Z).
\end{equation}

We now introduce the integrals ${\mathcal L}(j)$ for each $j$, 
$0\le j\le \min(n,k-(r+1)/2)$.
First, ${\mathcal L}(0)$ is defined to be the integral $L(r_1-1)=L((r-3)/2)$; see equation \eqref{wh4}.  Next, for $1\le j\le \min(n,k-(r+1)/2)$, 
let $a_j=2(n-j)+(2k-r-2j+2)(r-1)+2$, and for $v_0\in V_{0,2k-r-2j+2}$ and $h\in Sp_{2(n-j+1)}$, let
\begin{equation*}
\iota_{4,j}(v_0,h)=\text{diag}(v_0,\ldots,v_0,h,v_0^*\ldots,v_0^*)\in Sp_{a_j}.
\end{equation*} 
Here, $v_0$ appears $r_1$ times. 
The map $\iota_{4,j}$ is identical to $\iota_2$ except for the size of the groups (which, in each variable, depends on $j$) and that $v_0$ is not necessarily orthogonal.
Since $v_0$ is upper unipotent, it is split in the covering group by the
trivial section, and the map $\iota_{4,j}$  
extends to a one-to-one homomorphism (which we continue to denote $\iota_{4,j}$) from $V_{0,2k-r-2j+2}(\A)\times Sp^{(r)}_{2(n-j+1)}(\A)$ to $Sp^{(r)}_{a_j}(\A)$.
Then we define
\begin{multline}\label{end1}
{\mathcal L}(j)=\int\limits_{Sp_{2n-2j}(F)\backslash Sp_{2n-2j}({\A})}\ 
\int\limits_{U_{n,j}(F)\backslash U_{n,j}({\A})}\  
\int\limits_{V_{0,2k-r-2j+2}(F)\backslash V_{0,2k-r-2j+2}({\A})}\ 
\int\limits_{U_{0,j}(F)\backslash U_{0,j}({\A})} \\ 
\overline{\varphi^{(r)}}\left(u\begin{pmatrix}I_j&&\\&g&\\&&I_j\end{pmatrix}\right)\theta^{(r)}_{a_j}\left(u_0\iota_{4,j}\left(v_0,l_{2n-2j}(u)\begin{pmatrix}1&&\\&g&\\&&1\end{pmatrix}\right)\right)\\
\psi_{U_{0,j}}(u_0)\psi_{V_{0,2k-r-2j+2},\delta}(v_0)\psi_{U_{n,j}}(u)\,du_0\,dv_0\,du\,dg.
\end{multline}
 In \eqref{end1}, the integrand is a function of $g=(g_1,\zeta)\in Sp^{(r)}_{2n}(\A)$, but, as earlier in this paper,
the integrand is independent of $\zeta\in\mu_r$.  Thus we write the integral over the group and need not keep additional track of the cover.

The analysis of these integrals is by induction, and is given by the following key Lemma.
\begin{lemma}\label{lem4}
For $0\le j< \min(n,k-(r+1)/2)$, the integral ${\mathcal L}(j)$ is zero for all choices of data if and only if the integral ${\mathcal L}(j+1)$ is zero for all choices of data.
\end{lemma}
\begin{proof}
We start by establishing the Lemma for $j=0$. That is, we prove that ${\mathcal L}(0)$ is zero for all choices of data if and only if ${\mathcal L}(1)$ is zero for all choices of data. 
By definition,  ${\mathcal L}(0)=L((r-3)/2)$. Using the value $j=(r-3)/2$ in \eqref{wh4},  we have
\begin{multline}\label{end2}
{\mathcal L}(0)=\int\limits_{Sp_{2n}(F)\backslash Sp_{2n}({\A})}
\int\limits_{V_{2k-r+3}(F)\backslash V_{2k-r+3}({\A})}
\int\limits_{U_{r_0}(F)\backslash U_{r_0}({\A})}\overline{\varphi^{(r)}(g)}
\\ \theta_{a_0}^{(r)}(u\iota_{3,r_0}(v,g))\psi_{U_{r_0}}(u)
\psi_{V_{2k-r+3},\delta}(v)\,du\,dv\,dg.
\end{multline}
Here $a_0=2n+2k(r-1)-r(r-3)$, $r_0=r_1-1$, and for $v\in V_{2k-r+3}$ and $g\in Sp_{2n}$ we recall that  $\iota_{3,r_0}(v,g)=\text{diag}(v_0,\ldots,v_0,v,g,v^*,v_0^*\ldots,v_0^*)$
where each $v_0$ is repeated $r_0$ times (see \eqref{iota-3-j}).

The first part of the computation here is similar to the first part of the proof of Lemma~\ref{lem1}. 
For $\alpha=1,2$, let $M_\alpha$ denote the unipotent subgroup of $L_{2k-r+1+\alpha}$ consisting of all matrices of the form
$$\begin{pmatrix} 1&a\\ &I_{2k-r+\alpha}\end{pmatrix},\qquad a\in \Mat_{1\times (2k-r+\alpha)}.$$
As in the proof of Lemma~\ref{lem1}, let $M$  be the subgroup of $Sp_{a_0}$ of matrices of the form
$$m=\text{diag}(m_1,m_2,\cdots,m_{r_1},I_{2n},m_{r_1}^*,\cdots,m_{2}^*,m_{1}^*)$$
with $m_i\in M_1$ for $1\le i\le (r-3)/2$ and $m_{r_1}\in M_2$.
Also, let $Q_i$, $1\le i\le r_1$, denote the following subgroups of $Sp_{a_0}$. 
For $1\le i\le (r-5)/2$, $Q_i$ is the subgroup of all matrices of the form
$$u^i_{(2k-r+2)(i-1),2k-r+2,2k-r+2,a_i}(X_i),\qquad X_i=\begin{pmatrix} 0&\\ a&0_{2k-r+1}\end{pmatrix},\quad a\in \Mat_{(2k-r+1)\times 1}$$
with $a_i=a_0-2(2k-r+2)(i+1)$ (see \eqref{mat5});
for $i=(r-3)/2$, the group $Q_i$ is generated by
$$u^i_{(2k-r+2)(r-5)/2,2k-r+2,2k-r+3,a_i}(X_i),\qquad X_i=\begin{pmatrix} 0&\\ a&0_{2k-r+2}\end{pmatrix},\quad a\in \Mat_{(2k-r+2)\times 1}$$
with $a_i=a_0-(2k-r+2)(r-3)-2(2k-r+3)$; and for $i=(r-1)/2$, $Q_i$ is the subgroup of matrices of the form
$$u'_{2k-r+3,n}(0,Z),\qquad Z=\begin{pmatrix} 0&&\\ a&0_{2k-r+1}&\\ b&a^*&0
\end{pmatrix}, \quad b\in \Mat_1,\quad a\in \Mat_{(2k-r+1)\times 1}.$$

Performing root exchange between these two groups, we obtain that ${\mathcal L}(0)$ is zero for all choices of data if and only if an integral similar to \eqref{wh7} is zero for all 
choices of data with $j=(r-3)/2$. We write this integral after conjugating by a certain Weyl element. 
The Weyl elements $w_0^j$ were defined above in equation \eqref{weyl1} for $j$ in the range $1\le j< (r-3)/2$.  As mentioned there, 
that description still determines the matrix uniquely if $j=(r-3)/2$. We use the Weyl element $w_0^{(r-3)/2}$ now.

Performing the above root exchange, and then conjugating by $w_0^{(r-3)/2}$, we deduce that integral \eqref{end2} is zero for all choices of data if and only if the integral 
\begin{multline}\label{end3}
\int 
\overline{\varphi^{(r)}(g)}\widetilde{\psi}(A)\psi_{U}(u)\psi'_{0,V_{2k-r},\delta}(v_0)\psi'(D)\\
\theta_{a_0}^{(r)}\left (\begin{pmatrix} A&B&C\\ &I_{a_1}&B^*\\ &&A^*\end{pmatrix}\begin{pmatrix} I_r&B(v_0)&\\ &I_{a_1}&B^*(v_0)\\ &&I_r\end{pmatrix}
\begin{pmatrix} I_r&&\\ D&I_{a_1}&\\ E&D^*&I_r\end{pmatrix}\begin{pmatrix}I_r&&\\&u\iota_{4,1}(v_0,g)&\\&&I_r\end{pmatrix}\right ) \,
d(...)
\end{multline}
is zero for all choices of data. In \eqref{end3}, the variables $A,B,C,D$ and $E$ are integrated over the same domain as in \eqref{wh8} 
with $j=(r-3)/2$. We describe them explicitly. First $A$ is integrated over $L_r(F)\backslash L_r({\A})$. The character $\widetilde{\psi}(A)$ is the Whittaker character. 
Let $B(r,a_0)$ denote the subgroup of $\Mat_{r\times (a_0-2r)}$ which consists of all matrices $B=B_{\alpha,\beta}$ such that $B_{i,(i-1)(2k-r)}=0$ for all $1\le i\le r_1$ and
$B_{i,(i-1)(2k-r)+2n}=0$ for all $r_1+1\le i\le r$. The variable $B$ is integrated over $B(r,a_0)(F)\backslash B(r,a_0)({\A})$.
Let $C(r)$ denote the subgroup of $\Mat^0_r$ consisting of all matrices 
$$C=\begin{pmatrix} C_1&C_2\\ C_3&C^*_1\end{pmatrix},\qquad C_1\in \Mat_{r_1,r_1+1},\quad C_2\in \Mat^0_{r_1},\quad C_3\in \Mat^0_{r_1+1}$$
such that $C_3=0$, and such that, writing $C_1=(C_1[\alpha,\beta])$, one has
$C_1[\alpha,\beta]=0$ for all $1\le \alpha\le r_1$ and 
$\beta<\alpha$. The variable $C$ is integrated over $C(r)(F)\backslash C(r)({\A})$. The matrices $D$ and $E$ are integrated over $D(r,a_0)(F)\backslash D(r,a_0)({\A})$ 
and $E(r)(F)\backslash E(r)({\A})$, respectively, defined following \eqref{wh8}. The character $\psi'(D)$ is given by $\psi'(D)=\psi(D_{\alpha,\beta})$ 
where $\alpha=1+(2k-r)(r-3)/2$ and $\beta=(r+3)/2$. The variable $u$ is integrated over $U(F)\backslash U({\A})$, where $U$ is the subgroup of 
$U_{2k-r,r_1,n}\subseteq Sp_{2n+(2k-r)(r-1)}$ consisting of all 
matrices \eqref{mat3} such that $Y\in \Mat_{(2k-r)\times 2n}$ has bottom $(2k-r-1)/2$ rows all zeroes.
The character $\psi_U$ is given by 
\begin{equation}\label{psi-sub-U}\psi_U(u)=\psi\left(\text{tr}(X_1+\cdots+X_{r_1-1})\right)\psi_0(Z)
\end{equation} 
(using the factorization \eqref{mat3} of $u$ with $a=2k-r$, $b=r_1$, $c=n$).
The variable $v_0$ is integrated 
over $V_{0,2k-r}(F)\backslash V_{0,2k-r}({\A})$;
the character $\psi'_{0,V_{2k-r},\delta}$ was defined in \eqref{psi-prime}. The $g$ variable is integrated over $Sp_{2n}(F)\backslash Sp_{2n}({\A})$. 

Finally, we describe the matrix $B(v_0)\in \Mat_{r\times (a_0-2r)}$. 
The first $n+a_0/2$ columns of $B(v_0)$ are all zero. Write $B=\begin{pmatrix} 0&B_1\end{pmatrix}$ where $B_1\in \Mat_{r\times (a_0/2-n-2r)}$. Then the first $(r+1)/2$ rows of $B_1$ 
are zero. Write $B_1=\left(\begin{smallmatrix} 0\\ B_{1,1}\end{smallmatrix}\right)$ where $B_{1,1}\in \Mat_{(r-1)/2\times (a_0/2-n-2r)}$. Recall that in \eqref{end2}, $v=(v_{\alpha,\beta})$ is a 
matrix in $V_{2k-r+3}$. For such a matrix, write 
\begin{equation}\label{a-v-zero}
a(v_0)=(v_{2,2k-r+2},v_{3,2k-r+2},\ldots,v_{2k-r+1,2k-r+2})\in \Mat_{1\times (2k-r)}.
\end{equation}
Then 
$$B_{1,1}=\begin{pmatrix} a(v_0)&0&0&\ldots&0\\ 0&a(v_0)&0&\ldots&0\\ \vdots&\vdots&\vdots&\ddots&\vdots\\ 0&0&0&\ldots &a(v_0)\end{pmatrix}.$$
To simplify the notation and for later use as well, for $0\le j < \min(n,k-(r+1)/2)$, given a matrix $B(v_0)\in \Mat_{r\times a_{j+1}}(\A)$, we write
\begin{equation}\label{mat100}
p_j(B(v_0))=\begin{pmatrix} I_r&B(v_0)&\\ &I_{a_{j+1}}&B^*(v_0)\\ &&I_r\end{pmatrix}\in Sp_{a_j}(\A).
\end{equation}

We now proceed in the same way as we did in analyzing \eqref{wh8} above. That is, we perform root exchange between variables in the $D$ and $E$ matrices and suitable upper triangular  matrices. 
The process for the first $r_1$ rows is similar to the one performed for \eqref{wh8}.  
Similarly, we exchange the last $a_0/2+n$ entries of the $(r+1)/2$-th row. 
However, in the case at hand it is not possible to perform root exchange for the entries in positions $((r+1)/2, a_0/2-n+\alpha)$ of the left-most 
matrix in the argument of $\theta_{a_0}^{(r)}$ in the integral \eqref{end3}, for $1\le \alpha\le 2n$. 

We conclude that the integral ${\mathcal L}(0)$ is zero for all choices of data if and only if an integral, denoted $I$, is zero for all choices of data. The integral $I$ has
the same form as integral \eqref{end3} but with a different domain of integration.   
For $I$, the variable $B$ is integrated over  $B_1(r,a_0)(F)\backslash B_1(r,a_0)({\A})$, where $B_1(r,a_0)$ is defined as follows. 
If $B=(B_{\alpha,\beta})\in B_1(r,a_0)$, then $B_{(r+1)/2,\beta}=0$ for all $1\le\beta\le 2n+(2k-r)r_1$, and the bottom $r_1$ rows of $B$ satisfy the same vanishing
conditions as for $B(r,a_0)$. 
The variable $C$ is now integrated over $C_1(r)(F)\backslash C_1({\A})$ where $C_1(r)$ is the subgroup of $\Mat^0_r$ such that $C_{\alpha,\beta}=0$ for 
all $r_1\le\alpha\le r$ and $1\le\beta\le r_1$. The variable $D$ is integrated over $D_1(r,a_0)(F)\backslash D_1(r,a_0)({\A})$ where $D_1(r,a_0)$ is the 
subgroup of $D(r,a_0)$ of all matrices whose first $r_1+1$ columns are zeros, and similarly we define the group $E_1(r)$. 

Let $S_0$ denote the subgroup of $Sp_{a_0}$ consisting of all matrices of the form 
$$I_{a_0}+\sum_{i=1}^{2n}\alpha_ie'_{r_1+1,a_0/2-n+i}+\beta e_{r_1+1,a_0-r_1}.$$
Its center, denoted by $Z(S_0)$, is the group generated by all the above matrices such that $\alpha_i=0$ for all $i$. It is a subgroup of $C_1(r)$. 
This means that we may expand integral $I$ along $S_0(F)Z(S_0)({\A})\backslash S_0({\A})$. The group $Sp_{2n}(F)$ acts on this expansion with two orbits. 

First, we consider the contribution from the trivial orbit.  In this case, we can further perform a root exchange similar to the analysis of integral \eqref{wh8} above. 
Arguing as in \eqref{wh9}-- \eqref{wh11}, the vanishing of this contribution is equivalent to the vanishing of the integral  
\begin{equation}\label{end4}
\int \overline{\varphi^{(r)}(g)}
\theta_{a_1}^{(r)}(u\iota_{4,1}(v_0,g))\psi_{U}(u)\psi'_{0,V_{2k-r},\delta}(v_0)\,
du\,dv_0\,dg,
\end{equation}
with the domain of integration as in \eqref{end3}. 
The proof that integral \eqref{end4} is zero for all choices of data is similar to the computations we performed above. 
It will be convenient to postpone the proof and give the details later. We do so after Proposition~\ref{prop8} below.

Assuming this vanishing, we conclude that the integral $I$ is equal to the integral
\begin{multline}\label{end5}
\int\limits_{U_{n,1}(F)Sp_{2n-2}(F)\backslash Sp_{2n}({\A})}\ \int 
\overline{\varphi^{(r)}(g)}\widetilde{\psi}(A)\psi_{U}(u)\psi'_{0,V_{2k-r},\delta}(v_0)\psi'(D)\psi_1(B)\\
\theta_{a_0}^{(r)}\left (\begin{pmatrix} A&B&C\\ &I_{a_1}&B^*\\ &&A^*\end{pmatrix}p_0(B(v_0))\begin{pmatrix} I_r&&\\ D&I_{a_1}&\\ E&D^*&I_r\end{pmatrix}
\begin{pmatrix}I_r&&\\&u\iota_{4,1}(v_0,g)&\\&&I_r\end{pmatrix}\right ) \,
d(...).
\end{multline}
Here, all variables except the $g$ and $B$ variables, are integrated as in the integral $I$. The $B$ variable is integrated over $B_2(r,a_0)(F)\backslash B_2(r,a_0)({\A})$, 
where $B_2(r,a_0)=B_1(r,a_0)(Z(S_0)\backslash S_0)$. The group $B_2(r,a_0)$ may be identified with all $B=(B_{\alpha,\beta})\in \Mat_{r,a_0}$ such that $B_{r_1+1,\beta}=0$ for all $1\le \beta\le (2k-r)r_1$. 
All other rows are the same as for matrices in $B_1(r,a_0)$. Also, the character $\psi_1(B)=\psi(B_{r_1+1,(2k-r)r_1+1})$. 

Write 
$$\int\limits_{U_{n,1}(F)Sp_{2n-2}(F)\backslash Sp_{2n}({\A})}=
\int\limits_{U_{n,1}({\A})Sp_{2n-2}({\A})\backslash Sp_{2n}({\A})}\ \ 
\int\limits_{U_{n,1}(F)Sp_{2n-2}(F)\backslash U_{n,1}({\A})Sp_{2n-2}({\A})}.$$
Arguing as in \cite{G-S}, Section~7, we deduce that the integral \eqref{end5} is zero for all choices of data if and only if the integral 
\begin{multline}\label{end6}
\int\limits_{Sp_{2n-2}(F)\backslash Sp_{2n-2}({\A})}\int\limits_{U_{n,1}(F)\backslash U_{n,1}({\A})} \int 
\overline{\varphi^{(r)}}\left(u_1\begin{pmatrix}1&&\\&g&\\&&1\end{pmatrix}\right)\widetilde{\psi}(A)\psi_{U}(u)\psi'_{0,V_{2k-r},\delta}(v_0)\psi'(D)\psi_1(B)\\
\theta_{a_0}^{(r)}\left (\begin{pmatrix} A&B&C\\ &I_{a_1}&B^*\\ &&A^*\end{pmatrix}p_0(B(v_0))\begin{pmatrix} I_r&&\\ D&I_{a_1}&\\ E&D^*&I_r\end{pmatrix}
\begin{pmatrix}I_r&&\\&u\iota_{4,1}(v_0,l_{2n-2}(u_1)g)&\\&&I_r\end{pmatrix}\right ) \,
d(...)
\end{multline}
is zero for all choices of data.

Let $x(1)=I_{a_0}-e'_{1+a_0/2,r_1+2}$. Then $\theta_{a_0}^{(r)}(h)=\theta_{a_0}^{(r)}(x(1)h)$. Using this in integral \eqref{end6} and then conjugating $x(1)$ to the right, we obtain the integral
\begin{multline}\label{end7}
\int\limits_{Sp_{2n-2}(F)\backslash Sp_{2n-2}({\A})}\int\limits_{U_{n,1}(F)\backslash U_{n,1}({\A})} \int 
\overline{\varphi^{(r)}}\left(u_1\begin{pmatrix}1&&\\&g&\\&&1\end{pmatrix}\right)\widetilde{\psi}(A)\psi'_{U}(u)\psi'_{0,V_{2k-r},\delta}(v_0)\psi'(D)\\ 
\theta_{a_0}^{(r)}\Bigg(\begin{pmatrix} A&B&C\\ &I_{a_1}&B^*\\ &&A^*\end{pmatrix}p_0(B(v_0))\begin{pmatrix} I_r&&\\ D&I_{a_1}&\\ E&D^*&I_r\end{pmatrix}
\begin{pmatrix}I_r&&\\&uu'_{2k-r,n}(Y(v_0),0)&\\&&I_r\end{pmatrix}\\ \iota_{4,1}(v_0,l_{2n-2}(u_1)g)x(1)\Bigg)\,d(...).
\end{multline}
We highlight the changes in the characters. First $\psi_1(B)$ is cancelled. 
Second, the character $\psi'_{U}$ is defined as follows. Given the decomposition of $u\in U_{2k-r,r_1,n}(\A)$ as in \eqref{mat3}, we have $\psi'_{U}(u)=\psi_{U}(u)\psi(Y_{1,1})$. 
Also, we note that the matrix $u'_{2k-r,n}(Y(v_0),0)$ has the property that $Y(v_0)\in \Mat_{2k-r,2n}(\A)$ has all entries zero except the last column; this
last column is equal to $J_{2k-r}a(v_0)^t$. 
(The vector $a(v_0)$ is given in \eqref{a-v-zero}.)

At this point we carry out a root exchange between variables in $D$ and $E$ with suitable upper triangular matrices. 
This is similar to the steps in the proof of Lemma~\ref{lem1} that go from \eqref{wh8} to \eqref{wh9}.  We then argue as in that proof, from \eqref{wh9} to \eqref{wh11}.
We conclude that in the case at hand, the integral \eqref{end7} is zero for all choices of data if and only if the integral
\begin{multline}\label{end8}
\int\limits_{Sp_{2n-2}(F)\backslash Sp_{2n-2}({\A})}\int\limits_{U_{n,1}(F)\backslash U_{n,1}({\A})}  \int 
\overline{\varphi^{(r)}}\left(u_1\begin{pmatrix}1&&\\&g&\\&&1\end{pmatrix}\right)\psi'_{U}(u)\psi'_{0,V_{2k-r},\delta}(v_0)\\ 
\theta_{a_1}^{(r)}(uu'_{2k-r,n}(Y(v_0),0)\iota_{4,1}(v_0,l_{2n-2}(u_1)g))\,d(...)
\end{multline}
is zero for all choices of data. (All other domains of integration here are the same as in \eqref{end7}.)

We perform one final root exchange, as follows. Let $R_1$ denote the subgroup of $V_{0,2k-r}$ consisting of all matrices  of the form
$$\begin{pmatrix} 1&a&0\\ &I_{(2k-r-1)/2}&0\\ &&I_{(2k-r-1)/2}\end{pmatrix},\qquad a\in \Mat_{1\times (2k-r-1)/2}. $$ 
Let $R_2$ denote the subgroup of $U$ consisting of all elements of the form
$$u'_{2k-r,n}(Y,0),\qquad Y=\alpha_1e_{2,1}+\alpha_2e_{3,1}+\cdots +\alpha_{(2k-r-1)/2}e_{(2k-r-1)/2,1}$$
where $e_{\alpha,\beta}$ is the $(2k-r)\times 2n$ matrix whose $(\alpha,\beta)$ entry is $1$ and with all other entries $0$.
Using $\psi'_{U}$, we perform a root exchange between $R_1$ and $R_2$. From this we deduce that the integral \eqref{end8} is zero for all choices of data if and only if the integral
\begin{multline}\label{end9}
\int
\int\limits_{R_2({\A})U(F)\backslash U({\A})}
\int\limits_{R(F)\backslash R({\A})}
\overline{\varphi^{(r)}}\left(u_1\begin{pmatrix}1&&\\&g&\\&&1\end{pmatrix}\right)\psi'_{U}(u)\psi_{0,V_{2k-r},\delta}(v_0)
\psi_R(u'_{2k-r,n}(Y_0,0))\\ 
\theta_{a_1}^{(r)}(uu'_{2k-r,n}(Y_0,0)\iota_{4,1}(v_0,l_{2n-2}(u_1)g))\,d(...)
\end{multline}
is zero for all choices of data. Here $R\cong \Mat_{(2k-r-1)/2\times 1}$  is the group of all matrices of the form 
$$u'_{2k-r,n}(Y_0,0)=\alpha_1e_{(2k-r+3)/2,2n}+\alpha_2e_{(2k-r+5)/2,2n}+\cdots +\alpha_{2k-r}e_{2k-r,2n},$$
and $\psi_R(u'_{2k-r,n}(Y_0,0))=\psi(\alpha_{2k-r})$. (Notice that in \eqref{end9}, the character 
$\psi'_{0,V_{2k-r},\delta}$ of \eqref{end8} is replaced by $\psi_{0,V_{2k-r},\delta}$.) 
The domains of integration of the variables $v_0$ and $g$ are the same as in \eqref{end8}. 

The group generated by $R_2\backslash U$ and $R$ is isomorphic to $U_{0,1}$ (defined before integral \eqref{end1} above). Also, under this identification the product of the 
characters $\psi'_{U}$ and $\psi_R$ is equal to the character $\psi_{U_{0,1}}$. Thus integral \eqref{end9} is equal to integral ${\mathcal L}(1)$. 
This completes the proof that ${\mathcal L}(0)$ is zero for all choices of data if and only if ${\mathcal L}(1)$ is zero for all choices of data.

Next we prove that for $1\le j<\min(n,k-(r+1)/2)$, ${\mathcal L}(j)$ is zero for all choices of data if and only if ${\mathcal L}(j+1)$ is zero for all choices of data.
Many steps are similar to the proof for $j=0$ above. Starting with ${\mathcal L}(j)$, we again define the group $M$ and the groups $Q_i$, and  
perform a root exchange between them.  We then
define a Weyl element $w_0(j)\in Sp_{a_j}(F)$. Its definition is slightly different for $j>0$, so we give it here. We let
\begin{equation*}
w_0(j)=\begin{pmatrix} w_1&&w_2\\ &I_{2(n-j)}&\\  w_3&&w_4\end{pmatrix},\qquad
w_1,w_2,w_3,w_4\in \Mat_{r_1(2k-r-2j+2)+1}.
\end{equation*}
The matrix $w_0(j)$ is symplectic, so it suffices to specify $\begin{pmatrix} w_1& w_2\end{pmatrix}$.
The first $r$ rows of this matrix have entry $1$ at positions $(i,(i-1)(2k-r-2j+2)+1)$ for $1\le i\le r_1+1$, and at positions $(r_1+i, (r_1+i-1)(2k-r-2j+2)+1)$ for $2\le i\le r_1$.
For the next $r_1(2k-r-2j+2)+1-r$ rows, all entries in $w_2$ are zero, while for $w_1$, these rows form the matrix $w_{1,1}^0$ in \eqref{w-one-one-zero} with $\beta=2k-r-2j$. 

Introducing $w_0(j)$ and conjugating, we conclude that ${\mathcal L}(j)$ is zero for all choices of data if and only if the integral
\begin{multline*}
\int 
\overline{\varphi^{(r)}}\left(u\begin{pmatrix}I_j&&\\&g&\\&&I_j\end{pmatrix}\right)\widetilde{\psi}(A)\psi_{U}(u_0)\psi'_{0,V_{2k-r},\delta}(v_0)\psi_{U_{n,j}}(u)\psi'(D)\\
\theta_{a_j}^{(r)}\Bigg(\begin{pmatrix} A&B&C\\ &I_{a_{j+1}}&B^*\\ &&A^*\end{pmatrix}p_j(B(v_0))\begin{pmatrix} I_r&&\\ D&I_{a_{j+1}}&\\ E&D^*&I_r\end{pmatrix}
\begin{pmatrix}I_r&&\\& u_0\iota_{4,j+1}(v_0,1)&\\&&I_r\end{pmatrix}\\
w_0(j)\iota_{4,j}(1,l_{2n-2j}(u)g)\Bigg) \,
d(...)
\end{multline*}
is zero for all choices of data. 
The variables $A, B, C, D$ and $E$, and their characters are integrated over the same regions as in \eqref{end3}. 
The variable $v_0$ is integrated over the quotient  $V_{0,2k-r-2j}(F)\backslash V_{0,2k-r-2j}({\A})$; its character 
$\psi'_{0,V_{2k-r-2j},\delta}$ is defined similarly to \eqref{psi-prime}.
The variable $u_0$ is integrated over $U(F)\backslash U({\A})$, where $U$ is the subgroup of $U_{2k-r-2j,r_1,n-j}\subset Sp_{a_{j+1}}$ consisting of
$u_0$ with factorization \eqref{mat3} such that $Y\in \Mat_{(2k-r-2j)\times (2n-2j)}$ has bottom $(2k-r-2j-1)/2$ rows all zero. 
Notice that $u_0\in Sp_{a_{j+1}}$. The character $\psi_U$ is given by the same formula as \eqref{psi-sub-U}, using the factorization \eqref{mat3} of $u_0$
with $a=2k-r-2j$, $b=r_1$, $c=n-j$.
Finally, $u$ and $g$ are integrated as in \eqref{end1}. 

We again proceed as we did in analyzing \eqref{wh8} above (also see following \eqref{mat100}). We carry out a root exchange between root groups appearing in the 
$D, E$ variables and suitable upper triangular matrices. 
Then let $S_j$ be the subgroup of $Sp_{a_j}$ of matrices of the form
$$I_{a_0}+\sum_{i=1}^{2n-2j}\alpha_ie'_{r_1+1,a_0/2-n-j+i}+\beta e_{r_1+1,a_0-r_1}.$$
The center of $S_j$, $Z(S_j)$, is the group generated by the matrices above such that $\alpha_i=0$ for all $i$. Then $Z(S_j)$ is a subgroup of $C_1(r)$, the same group as defined 
above in treating ${\mathcal L}(0)$. 
Hence we may expand the integral along $S_j(F)Z(S_j)({\A})\backslash S_j({\A})$. The group $Sp_{2n-2j}(F)$ acts on this expansion with two orbits. 
First we consider the contribution from the constant term. The group $S_j$ is isomorphic to the Heisenberg group ${\mathcal H}_{2n-2j+1}$. 
Also, we have $w_0(j)(1,l_{2n-2j}(u))_{a_j}w_0(j)^{-1}\in S_j$. Changing variables in $S_j$, we obtain the constant term of $\overline{\varphi^{(r)}}$ along the unipotent radical of the maximal 
parabolic subgroup of $Sp_{2n}$ whose Levi part is $GL_j\times Sp_{2n-2j}$. From the cuspidality of $\varphi^{(r)}$, it follows that the contribution from the constant term along $S_j$ is zero. 

We are left with the contribution of the non-trivial orbit. Arguing as in the case $j=0$ (beginning with \eqref{end5}) 
we obtain the integral ${\mathcal L}(j+1)$. 
This concludes the proof of Lemma~\ref{lem4}.
\end{proof}

We conclude from Lemma~\ref{lem4}  that for $k>(r+1)/2$, the Whittaker coefficient \eqref{wh1} is zero for all choices of data if and only if the integral 
${\mathcal L}(\min(n,k-(r+1)/2))$ is zero for all choices of data.  We analyze the two possibilities for this minimum separately.

\begin{proposition}\label{prop1}
Suppose that $k\ge n+(r+3)/2$. Then the Whittaker coefficient \eqref{wh1} is zero for all choices of data.
\end{proposition}

\begin{proof}
The integral ${\mathcal L}(\min(n,k-(r+1)/2))={\mathcal L}(n)$ is equal to
\begin{equation*}
\int\limits_{U_{n,n}(F)\backslash U_{n,n}({\A})}\int
\overline{\varphi^{(r)}(u)}
\theta_{a_n}^{(r)}(u_0\iota_{4,n}(v_0,l_0(u)))
\psi_{U_{0,n}}(u_0)\psi'_{0,V_{2k-r-2n+2},\delta}(v_0)\psi_{U_{n,n}}(u) \,
d(...).
\end{equation*}
Here the variables $u_0$ and $v_0$ are integrated as in \eqref{end1} with $j=n$. By definition (see \eqref{the-hom-ell}), if $u=(u_{i,j})\in U_{n,n}$, then $l_0(u)=u_{n,n+1}\in \mathcal{H}_1$; 
in the above expression we embed 
this degenerate Heisenberg group in the group $Sp_2$ as the matrix $\left(\begin{smallmatrix}1&u_{n,n+1}\\&1\end{smallmatrix}\right)$.
Thus we have $\iota_{4,n}(1,l_0(u))=I_{a_n}+u_{n,n+1}e_{a_n/2,a_n/2+1}$.  Expand the above integral 
along the one parameter subgroup $S_n$ of $Sp_{a_n}$  defined by $S_n=\{I_{a_n}+\alpha e_{a_n/2,a_n/2+1}\}$. Then the contribution from each nontrivial orbit is zero. Indeed, each such 
contribution may be expressed as an integral of a Fourier coefficient of $\theta_{a_n}^{(r)}$ which corresponds to the unipotent orbit $((r+1)1^{a_n-r-1})$. By Theorem~\ref{theta05}, part~\ref{theta05-1},
this Fourier coefficient is zero. As for the constant term along $S_n$, after a suitable change of variables, we obtain as inner integration the constant term of $\overline{\varphi^{(r)}}$ 
along the unipotent radical of the parabolic subgroup of $Sp_{2n}$ whose Levi part is $GL_n$. By the cuspidality of $\varphi^{(r)}$, this constant term is zero. Thus ${\mathcal L}(n)$ is 
zero for all choices of data.  
\end{proof}

Next we examine the case $n\ge k-(r+1)/2$. From our work above, the Whittaker coefficient \eqref{wh1} is zero for all choices of data if and only if
the  integral ${\mathcal L}(k-(r+1)/2)$ is zero for all choices of data. 
In this case, $2n-2j=2n-2k+r+1$,  $2k-r-2j+2=3$, and $a_{k-(r+1)/2}=2(n-k+2r)$. Thus
\begin{multline}\label{arc2}
{\mathcal L}(k-(r+1)/2)=\int
\overline{\varphi^{(r)}}\left(u\begin{pmatrix}I_{k-(r+1)/2}&&\\&g&\\&&I_{k-(r+1)/2}\end{pmatrix}\right)\psi_{U_{n,k-(r+1)/2}}(u)\\ 
\theta_{2(n-k+2r)}^{(r)}(u_0\iota_{4,k-(r+1)/2}(v_0,l_{2n-2k+r+1}(u)g))
\psi_{U_{0,k-(r+1)/2},\delta}(u_0)\psi_{0,V_3}(v_0)\,du_0\,du\,dv_0\,dg,
\end{multline}
where 
$g$ is integrated over the quotient  $Sp_{2n-2k+r+1}(F)\backslash Sp_{2n-2k+r+1}({\A})$,
$u$ is integrated over $U_{n,k-(r+1)/2}(F)\backslash U_{n,k-(r+1)/2}({\A})$, 
$u_0$ is integrated over $U_{0,k-(r+1)/2}(F)\backslash U_{0,k-(r+1)/2}({\A})$, 
and $v_0$ is integrated over $V_{0,3}(F)\backslash V_{0,3}({\A})$ with
$$V_{0,3}=\left\{v_0=\begin{pmatrix} 1&x_1&x_2\\ &1&0\\ &&1\end{pmatrix}\right\}.$$
The characters in \eqref{arc2} are given by $\psi_{0,V_3}(v_0)=\psi(x_2)$ (see \eqref{wh1}) and 
$$\psi_{U_{0,k-(r+1)/2},\delta}(u)=\psi(\text{tr}(X_1+\cdots +X_{r_1}))\psi_{R,\delta}(Y)\psi_0(Z)\qquad $$
with $\psi_{R,\delta}(Y)=\psi(Y_{1,1}+\delta Y_{3,2n-2k+r+3})$ (compare \eqref{psi-111}).

To simplify the notation, let $a=2(n-k+2r)$. As in the previous cases, we start with a root exchange. For $1\le i\le r_1-1$, let $M_i$ denote the subgroup of $Sp_a$ consisting of all 
matrices $I_a+\alpha e'_{3i-2,3i-1}+\beta e'_{3i-2,3i}$. Let $M=V_{0,3}M_1\ldots M_{r_1-1}$. Let $\psi_{M}$ denote the character of $M$ 
defined by $\psi_{M}(m)=\psi_M(I_a+ \beta e'_{3r_1-2,3r_1})=\psi(\beta)$. For $1\le i\le r_1-1$, let $Q_i$ denote the subgroup of 
$U_{0,k-(r+1)/2}$ generated by all matrices 
$$u^i_{3(i-1),3,3,a-6(i+1)}(X_i),\qquad X_i=\begin{pmatrix} 0&\\ b&0_2\end{pmatrix},\quad b\in \Mat_{2\times 1}.$$
Then \eqref{arc2} is zero for all choices of data if and only if the integral
\begin{multline}\label{arc3}
\int
\overline{\varphi^{(r)}}\left(u\begin{pmatrix}I_{k-(r+1)/2}&&\\&g&\\&&I_{k-(r+1)/2}\end{pmatrix}\right)
\theta_{a}^{(r)}(u_0m\iota_{4,k-(r+1)/2}(1,l_{2n-2k+r+1}(u)g))\\
\psi_{U_{0,k-(r+1)/2}}(u_0)\psi_{M,\delta}(m)\psi_{U_{n,k-(r+1)/2}}(u) \,du_0\,du\,dm\,dg 
\end{multline}
is zero for all choices of data, 
where $M$ is integrated over $M(F)\backslash M({\A})$ and $u_0$ is integrated over 
$Q_1({\A})\ldots Q_{r_1-1}({\A})U_{n,k-(r+1)/2}(F)\backslash U_{n,k-(r+1)/2}({\A})$. 

Let $x(1)=I_a+e'_{3r_1,3r_1+1}$.  
Let $w_0$ denote the Weyl element of $Sp_a$ 
$$w_0=\text{diag}(w,I_{2n-2k+r+1},w^*)$$ 
where $w$ is the Weyl element of $GL_{3r_1+1}$ specified as follows. The first $r_1+1$ rows of $w$ have the entry $1$ at position $(i,2i+1)$ and $0$ elsewhere. 
The next $2r_1$ rows are given by the matrix
$$\begin{pmatrix} 0&I_2&0&0&\cdots&0\\ 0&0&0&I_2&\cdots&0\\ \vdots&\vdots &
\vdots&\vdots&\cdots&\vdots\\ 0&0&0&0&\cdots&I_2\end{pmatrix}.$$ 

We have $\theta_a^{(r)}(h)=\theta_a^{(r)}(w_0x(1)h)$. Conjugating $w_0x(1)$ to the right, integral \eqref{arc3} is equal to
\begin{multline}\label{arc4}
\int 
\overline{\varphi^{(r)}}\left(u\begin{pmatrix}I_{k-(r+1)/2}&&\\&g&\\&&I_{k-(r+1)/2}\end{pmatrix}\right)
\widetilde{\psi}(A)\psi_{U}(u_0)\psi_{U_{n,k-(r+1)/2}}(u)\psi_\delta(B)\\
\theta_{a}^{(r)}\left (\begin{pmatrix} A&B&C\\ &I_{a'}&B^*\\ &&A^*\end{pmatrix}\begin{pmatrix} I_{r_1+1}&&\\ D&I_{a'}&\\ &D^*&I_{r_1+1}\end{pmatrix}
\begin{pmatrix}I_{r_1+1}&&\\&u_0&\\&&I_{r_1+1}\end{pmatrix}
w_0x(1)\iota_{4,k-(r+1)/2}(1,l_{2n-2k+r+1}(u)g)\right ) \\ \,d(...).
\end{multline}
Here $a'=2n-2k+3r-1$ and $A$ is integrated over $L_{r_1+1}(F)\backslash L_{r_1+1}({\A})$. The character $\widetilde{\psi}(A)$ is the Whittaker character. 
The matrix $B$ is integrated over $B(r_1+1,a')(F)\backslash B(r_1+1,a')({\A})$, where $B(r_1+1,a')$ is the group of matrices 
$(B_{\alpha,\beta}) \in \Mat_{(r_1+1)\times a'}$ such that $B_{\alpha,2\alpha-2}=0$ for all $1\le \alpha\le r_1$, and $B_{r_1+1,\beta}=0$ for all $1\le \beta\le 2(n-k+r)$. 
The character $\psi_\delta(B)$ is given by  
$$\psi_\delta(B)=\psi(B_{r_1+1,2(n-k+r)+1}+\delta B_{r_1+1,2(n-k+r)+2}).$$ 
The variable $C$ is integrated over $C(r_1+1)(F)\backslash C(r_1+1)({\A})$, 
where $C(r_1+1)$ consists of all $C\in \Mat^0_{r_1+1}$ such that $C_{r_1+1,1}=0$. 
The variable $D$ is integrated over $D(a',r_1+1)(F)\backslash D(a',r_1+1)({\A})$ where $D(a',r_1+1)$ is the subgroup of $\Mat_{a'\times (r_1+1)}$ defined in a similar 
way to the group that appears following \eqref{wh8}. Finally, the variable $u_0$ is integrated over $U^0_{2,r_1,n-k+r_1+1}(F)\backslash U^0_{2,r_1,n-k+r_1+1}({\A})$, where the 
group $U^0_{2,r_1,n-k+r_1+1}\subset Sp_{a'}$ was defined after \eqref{mat4}. Its character $\psi_{U}(u_0)$ is defined similarly to equation \eqref{wh2}. 
The variables $u$ and $g$ are integrated as before. 

Carrying out a root exchange between the group $D(a',r_1+1)$ and a suitable group of upper triangular matrices involving roots in the first $r_1$ rows, 
we deduce that  \eqref{arc4} is zero for all choices of data if and only if the integral 
\begin{multline}\label{arc5}
\int 
\overline{\varphi^{(r)}}\left(u\begin{pmatrix}I_{k-(r+1)/2}&&\\&g&\\&&I_{k-(r+1)/2}\end{pmatrix}\right)
\widetilde{\psi}(A)\psi_{U}(u_0)\psi_{U_{n,k-(r+1)/2}}(u)\psi_\delta(B)\\
\theta_{a}^{(r)}\left (\begin{pmatrix} A&B&C\\ &I_{a'}&B^*\\ &&A^*\end{pmatrix}
\begin{pmatrix}I_{r_1+1}&&\\&u_0&\\&&I_{r_1+1}\end{pmatrix}
w_0x(1)\iota_{4,k-(r+1)/2}(1,l_{2n-2k+r+1}(u)g)\right ) \,d(...)
\end{multline}
is zero for all choices of data. Here $C$ is integrated over $\Mat^0_{r_1+1}(F)\backslash \Mat^0_{r_1+1}({\A})$ and $B$ is integrated over 
$B_1(r_1+1,a')(F)\backslash B_1(r_1+1,a')({\A})$, where $B_1(r_1+1,a')$ is the subgroup of $\Mat_{(r_1+1)\times a'}$ of all matrices $B$ such that 
$B_{r_1+1,\beta}=0$ for all $1\le \beta\le 2(n-k+r)$. All other variables are integrated as before.

Let $y(\delta)=I_a+\sum_{i=1}^{r_1}\delta e'_{2i-1,2i}$. Let $w_1$ denote the Weyl element of $Sp_a$ given by
$$w_1=\begin{pmatrix} I_{r_1+1}&&&&\\ &w_{1,1}&&w_{1,2}&\\ &&I_{2n-2k+r+1}&&\\
&w_{2,1}&&w_{2,2}&\\  &&&&I_{r_1+1}\end{pmatrix}.$$
Here, $w_{\alpha,\beta}\in\Mat_{r-1}$ for all $\alpha,\beta\in\{1,2\}$. To define $w_1$ we need only specify the matrices $w_{1,1}$ and $w_{1,2}$. We do so as follows.
The matrix $w_{1,1}$ has the entry $1$ at position $(r_1+i,2i)$ for $1\le i\le r_1$ and all other entries zero.
The matrix $w_{1,2}$ has the 
entry $1$ at position $(i,2i+1)$ for $1\le i\le r_1$ and all other entries zero. 

Introducing $w_1y(\delta)$ and conjugating, integral \eqref{arc5} is equal to 
\begin{multline}\label{arc6}
\int 
\overline{\varphi^{(r)}}\left(u\begin{pmatrix}I_{k-(r+1)/2}&&\\&g&\\&&I_{k-(r+1)/2}\end{pmatrix}\right)
\widetilde{\psi}(A)\psi_{U,\delta}(u_0)\psi_{U_{n,k-(r+1)/2}}(u)\psi'(D)\\
\theta_{a}^{(r)}\left (\begin{pmatrix} A&B&C\\ &I_{a''}&B^*\\ &&A^*\end{pmatrix}\begin{pmatrix} I_r&&\\ D&I_{a''}&\\ E&D^*&I_r\end{pmatrix}
\begin{pmatrix}I_{r}&&\\&u_0&\\&&I_{r}\end{pmatrix}
w_1y(\delta)w_0x(1)\iota_{4,k-(r+1)/2}(1,l_{2n-2k+r+1}(u)g)\right ) \\
\, d(...)
\end{multline}
Here, $a''=a-2r=2(n-k+r)$. The variable $A$ is now integrated over $L_r(F)\backslash L_r({\A})$, and $\widetilde{\psi}(A)$ is the Whittaker character of $L_r$. 
The variable $B$ is integrated over a subgroup of $\Mat_{r\times a''}$. Similarly, for the variables $C, D$ and $E$. We omit the precise descriptions. The variable $u_0$ is integrated 
over $U^0_{1,r_1,n-k+r_1+1}(F)\backslash  U^0_{1,r_1,n-k+r_1+1}({\A})$. 
For $u_0=(u_0[\alpha,\beta])\in U^0_{1,r_1,n-k+r_1+1}(\A)\subset Sp_{a''}(\A)$, the character $\psi_{U,\delta}$ is given for $k\leq n+(r+1)/2$ by
\begin{equation}\label{psi-u-delta}
\psi_{U,\delta}(u_0)=
\psi(u_0[1,2]+\cdots + u_0[r_1-1,r_1]+\delta
{u_0}[r_1,2n-2k+3(r+1)/2]). 
\end{equation}

We now carry out a root exchange between the root subgroups appearing in $D, E$ and 
suitable upper triangular matrices. We deduce that integral \eqref{arc6} is zero for all choices of data if and only if the integral
\begin{multline}\label{arc7}
\int 
\overline{\varphi^{(r)}}\left(u\begin{pmatrix}I_{k-(r+1)/2}&&\\&g&\\&&I_{k-(r+1)/2}\end{pmatrix}\right)
\widetilde{\psi}(A)\psi_{U,\delta}(u_0)\psi_{U_{n,k-(r+1)/2}}(u)\\
\theta_{a}^{(r)}\left (\begin{pmatrix} A&B&C\\ &I_{a''}&B^*\\ &&A^*\end{pmatrix}
\begin{pmatrix}I_{r}&&\\&u_0&\\&&I_{r}\end{pmatrix}
w_1y(\delta)w_0x(1)\iota_{4,k-(r+1)/2}(1,l_{2n-2k+r+1}(u)g)\right )\, d(...)
\end{multline}
is zero for all choices of data. In \eqref{arc7}, 
$B$ is integrated over $\Mat_{r\times a''}(F)\backslash \Mat_{r\times a''}({\A})$, and $C$ is integrated over $\Mat^0_r(F)\backslash \Mat^0_r({\A})$. 
Observe that the integration over the $B$ and $C$ variables give rise to the constant term along the unipotent radical of the parabolic subgroup of $Sp_a$ 
whose Levi part is $GL_r\times Sp_{a''}$. 

Next, we conjugate $w_1y(\delta)w_0x(1)$ across $\iota_{4,k-(r+1)/2}(1,l_{2n-2k+r+1}(u)g)$.
Note that the matrix $\iota_{4,k-(r+1)/2}(1,g)$ commutes with $w_1y(\delta)w_0x(1)$. Also, when conjugating the element $\iota_{4,k-(r+1)/2}((1,l_{2n-2k+r+1}(u))$ by 
$w_1y(\delta)w_0x(1)$, we obtain $u'\iota_{4,k-(r+1)/2}((1,l_{2n-2k+r+1}(u))$ where $u'$ is an element in the unipotent radical that we are integrating over in \eqref{arc7}. 
We conclude that the integral \eqref{arc7} is zero for all choices of data if and only if the integral
\begin{multline}\label{arc8}
\int 
\overline{\varphi^{(r)}}\left(u\begin{pmatrix}I_{k-(r+1)/2}&&\\&g&\\&&I_{k-(r+1)/2}\end{pmatrix}\right) \\
\theta_{2(n-k+r)}^{(r)} \left(u_0\begin{pmatrix}I_{r_1-1}&&\\&\iota_{4,k-r_1}(1,l_{2n-2k+r+1}(u)g)&\\&&I_{r_1-1}\end{pmatrix}\right)
\psi_{U,\delta}(u_0)\psi_{U_{n,k-(r+1)/2}}(u)\,du_0\,du\,dg
\end{multline}
is zero for all choices of data.
Here, all variables are integrated as in \eqref{arc7}.
In particular, the domain of integration of the variable $u_0$ is the quotient $U^0_{1,r_1,n-k+r_1+1}(F)\backslash  U^0_{1,r_1,n-k+r_1+1}({\A})$. 
The Fourier coefficient of $\theta_{2(n-k+r)}^{(r)}$ given by this integration is attached to the unipotent orbit $((r-1)1^{2n-2k+r+1})$. 
Note that this is the same Fourier coefficient as in integral \eqref{whca115}. (There we denoted the integration variable by $v$, but the integration domain and the character are the same.)

We now obtain consequences from the expression \eqref{arc8}.  Suppose first
 that $2n-2k+r+1=0$; that is $k=n+(r+1)/2$. In this case there is no $g$ integration and $l_{2n-2k+r+1}(u)=l_0(u)=u_{n,n+1}$. 
 The character $\psi_{U_{n,n}}(u)$, given by \eqref{character-U-nm} with $m=n$, is independent of $u_{n,n+1}$, but after making the variable change $u_0[n,n+1]\mapsto u_0[n,n+1]-u_{n,n+1}$,
 we obtain the generic (Whittaker) character attached to $Sp_{2n}$ given by 
\begin{equation*}
\psi_{Wh,-\delta}(u)=\psi\left(u_{1,2}+\dots+u_{n-1,n}-\delta u_{n,n+1}\right).
\end{equation*}
 Let $W_{\overline{\varphi^{(r)}},-\delta}(g)$ denote the value of the Whittaker coefficient of $\overline{\varphi^{(r)}}$ with respect to $\psi_{Wh,-\delta}$. 
Thus integral \eqref{arc8} is thus equal to
\begin{equation*}
W_{\overline{\varphi^{(r)}},-\delta}(e)
\int\limits_{U^0_{1,r_1,n-k+r_1+1}(F)\backslash  U^0_{1,r_1,n-k+r_1+1}({\A})} 
\theta_{2(n-k+r)}^{(r)} (u_0)
\psi_{U,\delta}(u_0)\,du_0.
\end{equation*}

The complex conjugate of $\psi_{Wh,-\delta}$ is in the same class as $\psi_{Wh,\delta}$ modulo the conjugation action of the rational torus.
Also, since $k=n+(r+1)/2$, the character $\psi_{U,\delta}$ is a Whittaker character of $Sp_{2(n-k+r)}$.  We arrive at the following statement.

\begin{proposition}\label{prop7} Suppose that $k=n+(r+1)/2$ and $r>1$. 
Fix a nontrivial additive character $\psi$ of $F\backslash \A$.
Then the Whittaker coefficient ${\mathcal W}_{k,\delta}(f)$ is not zero for some choice of data if and only if 
\begin{enumerate} \item $\Theta^{(r)}_{r-1}$ has a nonzero Whittaker coefficient
with respect to the Whittaker character $\psi_{U,\delta}$ of $Sp_{r-1}$ and 
\item $\pi^{(r)}$ is generic with respect to the Whittaker character $\psi_{Wh,\delta}$ of $Sp_{2n}$.
\end{enumerate}
\end{proposition}

Next, suppose that $2n-2k+r+1>0$. Since we are in the case $k>(r+1)/2$, this implies that $(r+3)/2\le k\le n+(r-1)/2$. In this case, as noted above, the integration over the 
quotient $U^0_{1,r_1,n-k+r_1+1}(F)\backslash  U^0_{1,r_1,n-k+r_1+1}({\A})$ is attached to the unipotent orbit $((r-1)1^{2n-2k+r+1})$. 
We may use the result of \cite{I1} and  Conjecture~\ref{conj20} and then argue as in the treatment of equations \eqref{whca116} and \eqref{whca117}.
We obtain the following result (we suppress the details, as they are similar to the treatment there).
\begin{proposition}\label{prop8} 
Suppose that $(r+3)/2\le k\le n+(r-1)/2$ and that the Descent Conjecture (Conjecture~\ref{conj20}) holds. Then the Whittaker coefficient \eqref{wh1} is not zero for some choice of data if and only if 
the integral \eqref{intro5} is not zero for some choice of data.
\end{proposition}

To conclude this Section, we now return to the point we deferred above, and prove that the integral \eqref{end4} is zero for all choices of data.
 To do this we will define a family of integrals $K(j)$, where $0\le j\le \min(r_1-1,(2k-r-3)/2)$. 
 Recall that $k>(r+1)/2$. Hence, $2k-r-3\ge 0$, and the set of such $j$ is non-empty. We will prove that for $0\le j< \min(r_1-1,(2k-r-3)/2)$, the integral  $K(j)$ is zero for all choices of data if and 
 only if the integral $K(j+1)$ is zero for all choices of data. The idea is similar to the proof of Lemma~\ref{lem1}. 

First, we define $K(0)$ to be the integral \eqref{end4}. 
To define the integrals $K(j)$ for $j\ge 1$, we introduce unipotent groups $U^{2,j}_{2k-r,r_1,n}$. 
Let $b_j=2n+(2k-r)(r-1)-2jr $ (so $b_j=a_j-2r$).
Let $Q^j_{2k-r,r_1,n}$ denote the parabolic subgroup of $Sp_{b_j}$ whose Levi part is $GL_{2k-r-2j}^{r_1-j}\times GL_{2k-r-2j-1}^j\times Sp_{2n}$. 
Let $U^{2,j}_{2k-r,r_1,n}$ denote the unipotent radical of $Q^j_{2k-r,r_1,n}$.  The matrices in $U^{2,j}_{2k-r,r_1,n}$ have a factorization that is similar to \eqref{mat5}.
We define the subgroup $U''_{2,j}$ of $U^{2,j}_{2k-r,r_1,n}$ by imposing the same condition $Y_2=0$ used to specify the subgroup $U_j$ of $U^{1,j}_{2k,r_1,n}$.
Moreover, let $\psi_{U''_{2,j}}$ be the character of the 
quotient $U''_{2,j}(F)\backslash U''_{2,j}({\A})$ given as in \eqref{wh2} with respect to the factorization here.  

Recall that the group $V_{2k-2j-r-1}$ is the upper triangular maximal unipotent subgroup of $SO_{2k-2j-r-1}$. Embedding 
$V_{2k-2j-r-1}$ in $V_{2k}$ by $v\mapsto \diag(I_{r_1+1},v,I_{r_1+1})$, we define the character $\psi_{V_{2k-2j-r-1},\delta}$ to be the restriction of $\psi_{V_{2k},\delta}$ 
(see equation \eqref{whit-char}) to the embedded image of the group $V_{2k-2j-r-1}$. 
We consider the semidirect product of the groups $V_{2k-2j-r-1}$ and $\Mat_{1\times (2k-2j-r-1)}$ realized as the set of
matrices $v(a)=\left(\begin{smallmatrix} 1&a\\ &v\end{smallmatrix}\right)\in  L_{2k-2j-r}$ with 
$v\in V_{2k-2j-r-1}$, and $a\in \Mat_{1\times (2k-2j-r-1)}$. 
If $2k-2j-r\neq3$, define a character on $V_{2k-2j-r-1}(\A)\ltimes \Mat_{1\times (2k-2j-r-1)}(\A)$ by
$$\psi'_{V_{2k-2j-r-1},\delta}(v(a))=\psi(a_{1,1})
\psi_{V_{2k-2j-r-1},\delta}(v).$$
If $2k-2j-r=3$, then $V_{2k-2j-r-1}$ consists of only the identity matrix, and in that case we define $\psi'_{V_{2k-2j-r-1},\delta}(v(a))=\psi(\delta a_{1,1}+a_{1,2})$. 
We define an embedding $\iota_{5,j}:V_{2k-2j-r-1}\ltimes \Mat_{1\times (2k-2j-r-1)}\to  Sp_{b_j}$ by the formula
$$\iota_{5,j}(v(a))=\text{diag}(v(a),\ldots,v(a),v,\ldots,v,I_{2n},v^*,\ldots,v^*,v(a)^*\ldots,v(a)^*),$$
where $v$ appears $j$ times and $v(a)$ appears $r_1-j$ times. (Since $1\le j\le \min(r_1-1,(2k-r-3)/2)$, all indices 
appearing in $\iota_{5,j}$ are positive integers.)

For $1\le j\le \min(r_1-1,(2k-r-3)/2)$ define 
\begin{equation*}
K(j)=\int \overline{\varphi^{(r)}(g)}\theta_{b_j}^{(r)}\left(u\iota_{5,j}(v(a))\begin{pmatrix}I_{c_j}&&\\&g&\\&&I_{c_j}\end{pmatrix}\right)\psi_{U''_{2,j}}(u)
\psi'_{V_{2k-2j-r-1},\delta}(v(a))\,du\,dv(a)\,dg.
\end{equation*}
Here, $u$ is integrated over $U''_{2,j}(F)\backslash U''_{2,j}({\A})$, the variable $v(a)$ 
is integrated over $$V_{2k-2j-r-1}(F)\ltimes \Mat_{1\times (2k-2j-r-1)}(F)\backslash V_{2k-2j-r-1}({\A})\ltimes \Mat_{1\times (2k-2j-r-1)}({\A}),$$
$g$ is integrated over $Sp_{2n}(F)\backslash Sp_{2n}({\A})$, and the integer $c_j=(2k-r)r_1-jr $.

We now prove that for $0\le j< \min(r_1-1,(2k-r-3)/2)$, the integral $K(j)$ is zero for all choices of data if and only if integral $K(j+1)$ is zero for all choices of data. 
If $j<(2k-r-3)/2$ it follows that $2k-r-2j\ge 5$. We start with a root exchange that is similar to the one performed just prior to
 equation \eqref{arc3}. Let $\tt{M}_1$ denote the subgroup of $L_{2k-2j-r}$ consisting of all matrices of the form $\left(\begin{smallmatrix} 1&a\\ &I\end{smallmatrix}\right)$ 
where $a\in \Mat_{1\times (2k-2j-r-1)}$.
(Thus we may identify the group $V_{2k-2j-r-1}\ltimes \Mat_{1\times (2k-2r-2j-1)}$ with $V_{2k-2j-r-1}\cdot \tt{M}_1$.)
Let $\tt{M}$ be the subgroup of $Sp_{b_j}$ consisting of 
all matrices of the form $\text{diag}(m_1,\ldots,m_{r_1-j},I_c,m_{r_1-j}^*,\ldots,m_1^*)$ 
where each $m_i\in \tt{M}_1$ and $c=2(n+2k-2j-r-1)$.  
Define the subgroups $Q_i$ of $U''_{2,j}$  for $1\le i\le r_1-j-1$ similarly to the definition of $Q_i$ above \eqref{arc3}. 
Each such $Q_i$ is isomorphic to $\Mat_{(2k-2j-r-1)\times 1}$. Performing this root exchange, 
we deduce that $K(j)$ is zero for all choices of data if and only if the integral 
\begin{equation}\label{cons2}
\int \overline{\varphi^{(r)}(g)}\theta_{b_j}^{(r)}\left(um\iota_{5,j}(v(0))\begin{pmatrix}I_{c_j}&&\\&g&\\&&I_{c_j}\end{pmatrix}\right)
\psi_{U''_{2,j}}(u)
\psi_{V_{2k-2j-r-1},\delta}(v)\psi_{\tt M}(m)\,du\,dm\,dv\,dg
\end{equation}
vanishes for all choices of data.
Here, $m$ is integrated over $\tt{M}(F)\backslash \tt{M}({\A})$, and the character $\psi_{\tt M}$ is defined as follows.
Let $m=\text{diag}(m_1,\ldots,m_{r_1-j},I_c,m_{r_1-j}^*,\ldots,m_1^*)\in \tt{M}(\A)$. Then $\psi_{\tt M}(m)=\psi(m_{r_1-j}[1,2])$. The variable $u$ is integrated over 
$Q_1({\A})\ldots Q_{r_1-j-1}({\A})U''_{2,j}(F)\backslash U''_{2,j}({\A})$. The variable $v$ is integrated over $V_{2k-2j-r-1}(F)\backslash V_{2k-2j-r-1}({\A})$, and the character 
$\psi_{V_{2k-2j-r-1},\delta}$ is the restriction of the Whittaker character \eqref{whit-char}. The variable $g$ is integrated as before. 

Now we repeat the same steps as in the proof of Lemma~\ref{lem1}. First, we define a Weyl element $w_0^{,j}$ of $Sp_{b_j}$ as in \eqref{weyl1}. 
However, in the case at hand we need to interchange $j$ and $r_1-j$, to interchange $\beta$ and $\gamma$, and to replace $2k$ by $2k-r$. 
After conjugating by this Weyl element, we obtain 
\begin{multline}\label{cons3}
\int  \overline{\varphi^{(r)}(g)}
\theta_{b_j}^{(r)}\left (\begin{pmatrix} A&B&C\\ &I_{b_{j+1}}&B^*\\ &&A^*\end{pmatrix}
\begin{pmatrix} I_r&&\\ D&I_{b_{j+1}}&\\ E&D^*&I_r\end{pmatrix}\begin{pmatrix}I_r&&\\&u\iota_{5,j+1}(v(a))&\\&&I_r\end{pmatrix}
\begin{pmatrix}I_{c_j}&&\\&g&\\&&I_{c_j}\end{pmatrix}\right)\\ 
\widetilde{\psi}(A)\psi_{U''_{2,j+1}}(u)\psi'_{V_{2k-2j-r-3},\delta}(v(a))\,d(...)
\end{multline}
(this is similar to \eqref{wh8}). Here $u$ is integrated over $U''_{2,j+1}(F)\backslash U''_{2,j+1}({\A})$, the variable $v(a)$ is integrated over 
$$V_{2k-2j-r-3}(F)\ltimes \Mat_{1\times (2k-2j-r-3)}(F)\backslash V_{2k-2j-r-3}({\A})\ltimes \Mat_{1\times (2k-2j-r-3)}({\A}),$$ and 
$g$ is integrated over $Sp_{2n}(F)\backslash Sp_{2n}({\A})$. 
The variables $A, B, C, D$ and $E$ are integrated over groups which are defined in a similar way to \eqref{wh8}. 

After performing further root exchanges, we deduce that  \eqref{cons3} is zero for all choices of data if and only if the integral
\begin{multline}\label{cons4}
\int 
\overline{\varphi^{(r)}(g)}
\theta_{b_j}^{(r)}\left (\begin{pmatrix} A&B&C\\ &I_{b_{j+1}}&B^*\\ &&A^*\end{pmatrix}
\begin{pmatrix}I_r&&\\&u\iota_{5,j+1}(v(a))&\\&&I_r\end{pmatrix} 
\begin{pmatrix}I_{c_j}&&\\&g&\\&&I_{c_j}\end{pmatrix}\right)\\ 
\widetilde{\psi}(A)\psi_{U''_{2,j+1}}(u)\psi'_{V_{2k-2j-r-3},\delta}(v(a))\,d(...)
\end{multline}
is zero for all choices of data. Here $B$ is integrated over $\Mat_{r\times b_{j+1}}(F)\backslash \Mat_{r\times b_{j+1}}({\A})$, and $C$ over $\Mat_r^0(F)\backslash \Mat_r^0({\A})$. 
The integration over $B$ and $C$ gives the constant term along the unipotent radical of the parabolic group whose Levi part is 
$GL_r\times Sp_{b_{j+1}}$. (See the discussion following \eqref{wh10}.)
We deduce that the integral \eqref{cons4} is zero for all choices of data if and only if $K(j+1)$ is zero for all choices of data. 
This completes the proof that $K(j)$ is zero for all choices of data if and only if $K(j+1)$ is zero for all choices of data. 

We deduce that the integral \eqref{end4} is zero for all choices of data if and only if the integral $K(\min(r_1-1,(2k-r-3)/2))$ is zero for all choices of data. 
We now consider each case for the minimum.
Suppose first that $(2k-r-3)/2\le r_1-1$. Then
\begin{multline}\label{cons5}
K((2k-r-3)/2)=\int \overline{\varphi^{(r)}(g)}\theta_{2(n+2r-k)}^{(r)}\left(u\iota_{5,(2k-r-3)/2}(v(a))\begin{pmatrix}I_{2r-k}&&\\&g&\\&&I_{2r-k}\end{pmatrix}\right)\\ \psi_{U''_{2,(2k-r-3)/2}}(u)
\psi'_{V_2,\delta}(v(a))\,du\,dv(a)\,dg.
\end{multline}
In this integral $u$ is integrated over adelic quotient of the unipotent group $U''_{2,(2k-r-3)/2}$.  We recall that the group $U''_{2,(2k-r-3)/2}$
 is a subgroup of the unipotent radical of the parabolic subgroup of $Sp_{2n+4r-2k}$ whose Levi part is $GL_3^{r-k+1}\times GL_2^{(2k-r-3)/2}\times Sp_{2n}$. 
The group $V_2$ in \eqref{cons5} is the identity group while $a$ ranges over $1\times 2$ matrices, and the character $\psi'_{V_2,\delta}(v(a))=\psi(\delta a_{1,1}+a_{1,2})$. 
Defining the group $M$ and the groups $Q_i$, and then performing root exchange as immediately before integral \eqref{cons2}, we obtain 
as inner integration the Fourier coefficient of $\theta_{2(n+2r-k)}^{(r)}$ which corresponds to the unipotent orbit $((r+1)(2k-r-1)1^{2(n+2r-2k)})$. 
By Theorem~\ref{theta05}, part~\ref{theta05-1}, this Fourier coefficient is zero. Thus, integral $K((2k-r-3)/2)$, and hence integral \eqref{end4}, are both zero for all choices of data. 

Suppose instead that $(2k-r-3)/2> r_1-1$. Then we must analyze
\begin{multline*}
K(r_1-1)=\int \overline{\varphi^{(r)}(g)}\theta_{b_{r_1-1}}^{(r)}\left(u\iota_{5,r_1-1}(v(a))\begin{pmatrix}I_{c_{r_1-1}}&&\\&g&\\&&I_{c_{r_1-1}}\end{pmatrix}\right)\\
\psi_{U''_{2,r_1-1}}(u) \psi'_{V_{2(k-r+1),\delta}}(v(a))\,du\,dv(a)\,dg
\end{multline*}
where $b_{r_1-1}=2n+(2k-r)(r-1)-r(r-3)$. Here $u$ is integrated over $U''_{2,r_1-1}$, a subgroup of the unipotent radical of the parabolic subgroup of $Sp_{b_{r_1-1}}$ whose Levi part is 
$GL_{2k-2r+3}\times GL_{2(k-r+1)}^{r_1-1}\times Sp_{2n}$. We do not need to do any root exchange at the first step. Defining a suitable Weyl element, we deduce that 
the integral $K(r_1-1)$ is equal to an integral similar to integral \eqref{end3}. In other words, we obtain a matrix similar  to $B(v_0)$. However, in this case,
after a suitable change of variables, we obtain as inner integration an integral of the type $\int\psi(x)\,dx$, with $x$ integrated over $F\backslash {\A}$. 
This integral is zero. We omit the details. We deduce that integral $K(r_1-1)$ is zero for all choices of data. 
This completes the proof that integral \eqref{end4} is zero for all choices of data.

\section{The vanishing of the Whitaker coefficient of the lift in the case $k\le n-\frac{r+1}{2}$.}\label{4G}
In this section we study the case $k\le n-\frac{r+1}{2}$.  By the results of Section~\ref{whcoeff}, the integral \eqref{wh1} is zero for all choices 
of data if and only if the integral \eqref{whca115} is zero for all choices of data.
Our goal is to prove the vanishing of this integral when $k\le n-\frac{r+1}{2}$.
We explain how this follows from a Descent Conjecture; that is, we formulate and establish Proposition~\ref{general-two}, part~\ref{g4-2}. We also indicate how one may establish slightly weaker results without
this.

The proofs of Proposition~\ref{general-one}, part~\ref{g3-2} and Proposition~\ref{general-two}, part~\ref{g4-1}, relied on studying the integral \eqref{whca116} (see \eqref{whca115} to \eqref{whca119} above). 
This is a descent integral in the sense of \cite{G-R-S4} and of \cite{F-G1}, Section~3. With the notation as in \eqref{whca116}, consider the representation $\rho$ 
of $Sp_{2n+r-2k+1}^{(2r)}({\A})$ whose underlying vector space is generated by the functions
\begin{multline}\label{whca200}
f(m)=\\ \int\limits_{U_{1,r_1,n+r_1-k}(F)\backslash U_{1,r_1,n+r_1-k}({\A})}
\overline{\theta_{2n+r-2k+1}^{(2),\psi^\delta}}(l_{2n+r-2k+1}(v')m)
\theta_{2(n+r-k)}^{(r)}(v'm)\psi_{U_{1,r_1,n+r_1-k}}(v')\,dv'.
\end{multline}
Note that since $r$ is odd, each function $f$ is a genuine function on the $2r$-fold cover, and each $f$ is invariant under $Sp_{2n+r-2k+1}(F)$.
This construction was analyzed in \cite{F-G1}, and we established the following result there (\cite{F-G1}, Proposition 4.1).

\begin{proposition} Suppose $r$ is an odd integer such that the Conjecture~\ref{conj1} (the Orbit Conjecture) holds. 
Then the representation $\rho$ is a subrepresentation of the Hilbert space 
$$L^2(Sp_{2n+r-2k+1}^{(2r)}({F})\backslash Sp_{2n+r-2k+1}^{(2r)}({\A}))$$
and it has nonzero projection to the residual spectrum.
\end{proposition}

We make the stronger Descent Conjecture:

\begin{conjecture}[The Descent Conjecture] \label{conj20}
The representation $\rho$ is in the residual spectrum, and 
each irreducible summand of $\rho$ is the representation $\Theta_{2n+r-2k+1}^{(2r)}$.
\end{conjecture}

\noindent
In \cite{F-G1}, Conjecture~4.2, we conjectured that $\rho$ is in fact {\it exactly} the theta representation $\Theta_{2n+r-2k+1}^{(2r)}$.
The slightly weaker statement of Conjecture~\ref{conj20} suffices for our applications.
We also remark that the local version of Conjecture~\ref{conj20} is true.  This is given precisely in Lemma~\ref{local-descent} below.
As a consequence the Descent Conjecture would follow from a strong multiplicity one theorem for $Sp_{2n+r-2k+1}^{(2r)}(\A)$.

We recall that the theta representation 
$\Theta_{2n+r-2k+1}^{(2r)}$ on $Sp_{2n+r-2k+1}^{(2r)}({\A})$ is obtained from the residue of the Eisenstein series $E^{(2r)}_{2n+r-2k+1}(\cdot,s)$ that is 
associated with the induced representation $\text{Ind}_{P({\A})}^{Sp_{2n+r-2k+1}^{(2r)}({\A})}\Theta_{GL_{n+r_1-k+1}}^{(r)}\delta_P^s$. 
Here $P$ is the maximal parabolic subgroup of $Sp_{2n+r-2k+1}$ whose Levi part is $GL_{n+r_1-k+1}$ and
$\Theta_{GL_{n+r_1-k+1}}^{(r)}$ is the theta representation of $GL_{n+r_1-k+1}^{(r)}({\A})$. See \cite{F-G1}, pg.\ 94.
Using this, Conjecture~\ref{conj20}  allows us to establish the following result.

\begin{proposition}\label{th20}
Assume Conjecture~\ref{conj20} holds. Then Proposition~\ref{general-two}, part~\ref{g4-2} is true.
\end{proposition}
\begin{proof}
Under this hypothesis, the argument is similar to the linear case $r=1$ (see \eqref{intro2} and \eqref{intro3}).  
That is, we replace the theta function by an Eisenstein series and unfold in order to obtain zero.
Indeed, it follows from Conjecture~\ref{conj20} that the integral \eqref{wh1} is zero for all choices of data if and only if
the  integrals \eqref{intro4} and \eqref{intro5} are zero for all choices of data.  Suppose $2\leq k\leq \tfrac{r+1}{2}.$
Consider the integral 
\begin{multline}\label{whca221}
\int\limits_{Sp_{2n}(F)\backslash Sp_{2n}({\A})}\ \  
\int\limits_{U_{n+\frac{r+1}{2}-k,\frac{r+1}{2}-k}(F)\backslash U_{n+\frac{r+1}{2}-k,\frac{r+1}{2}-k}({\A})}
\overline{\varphi^{(r)}(g)}\\ \overline{\theta_{2n}^{(2),\psi^\delta}}(l(u)g) 
E^{(2r)}_{2n+r-2k+1}\left(u\begin{pmatrix}I_{\tfrac{r+1}{2}-k}&&\\&g&\\&&I_{\tfrac{r+1}{2}-k}\end{pmatrix},s\right)
\psi_{U_{n+\frac{r+1}{2}-k,\frac{r+1}{2}-k}}(u)\,du\,dg.
\end{multline}
We can unfold this integral, and the unfolding process is exactly as in the linear group case. See \cite{G-R-S2}. 
Doing so, and using the cuspidality of $\varphi^{(r)}$, we deduce that integral \eqref{whca221} unfolds to an integral which has
as inner integration the Whittaker coefficient of the representation 
$\Theta_{GL_{n+r_1-k+1}}^{(r)}$. However, if $r<n+r_1-k+1$, then by \cite{K-P}, Theorem I.3.5 and II.2.1, this representation is not generic.  
A similar argument applies to integral \eqref{intro5}. The Proposition follows.
\end{proof}

To conclude, we observe that one may obtain somewhat weaker results towards Proposition~\ref{general-two}, part~\ref{g4-2} without Conjecture~\ref{conj20}. 

\begin{proposition}\label{weaker-prop-4}
Suppose that
$\pi^{(r)}$ is an irreducible cuspidal representation of $Sp_{2n}^{(r)}({\A})$ with the property that it has at least one unramified constituent which is in general position.
Suppose that $k\le n-\frac{r+1}{2}$. Then the representation  $\sigma_{n,k}^{(r)}$ is not generic.
That is, Proposition~\ref{general-two}, part~\ref{g4-2} holds for $\pi^{(r)}$.
\end{proposition}

We sketch the proof.  First, suppose that the Whittaker coefficient $W_{k,\delta}(f)$ is nonzero for some choice of data and some $\delta\in F^\times$.
The arguments above show that the integral \eqref{wh1} is nonzero if and only if integrals similar to \eqref{intro4} and \eqref{intro5} are nonzero, where in those integrals the functions
$\theta_{2n+r-2k+1}^{(2r),\psi^\delta}$ are known only to be functions of the form $f(m)$ as in \eqref{whca200}, that is, functions obtained by the descent process that uses the
theta representation $\Theta_{2(n+r-k)}^{(r)}$ to construct automorphic functions on  $Sp_{2n+r-2k+1}^{(2r)}({\A})$.  
The nonvanishing of this integral allows us to conclude that a local Hom space is nonzero.  Indeed, choose vectors such that the integral is nonzero, and choose a finite
place $\nu$ such that all data are unramified.  Since the local groups at $\nu$ act on the representations, we obtain a nonzero trilinear form.  

Suppose that $2\leq k\leq \tfrac{r+1}{2}$ so that
we consider \eqref{intro4}; the 
case corresponding to \eqref{intro5} is treated in a similar way.  Let $\Theta$ now denote the local theta representations with the groups and covers notated as in the global case.
Let $U_{1,r_1,n+r_1-k}(F_\nu)$ act on $\Theta_{2n+r-2k+1}^{(2),\psi^\delta}$  via the embedding $l_{2n+r-2k+1}$ of this group in $\mathcal{H}_{2n+r-2k+2}(F_\nu)$.
Let
$$J_{U_{1,r_1,n+r_1-k}(F_\nu)}\left(\Theta_{2(n+r-k)}^{(r)}\right)=V/W$$  
be the $Sp^{(2r)}_{2n+r-2k+1}(F_\nu)$-module that arises from the local descent, that is, the twisted Jacquet module.
Here $V$ is the vector space generated by vectors of the form $v_1\otimes v_2$, $v_1\in \Theta_{2n+r-2k+1}^{(2),\psi^\delta}$, $v_2\in 
\Theta_{2(n+r-k)}^{(r)}$, and $W$ is the subspace
$$\{(u\cdot v_1)\otimes (u\cdot v_2)-\psi_{U_{1,r_1,n+r_1-k}}(u) v_1\otimes v_2 \mid v_1\in \Theta_{2n+r-2k+1}^{(2),\psi^\delta}, v_2\in 
\Theta_{2(n+r-k)}^{(r)}, u \in U_{1,r_1,n+r_1-k}(F_\nu)\}.$$
Then $J_{U_{1,r_1,n+r_1-k}(F_\nu)}(\Theta_{2(n+r-k)}^{(r)})$ is an $Sp_{2n}(F_\nu)$-module via the embedding 
$$h\mapsto \diag(I_{\frac{r+1}{2}-k},h,I_{\frac{r+1}{2}-k})$$
of $Sp_{2n}(F_\nu)$ in $Sp_{2n+r-2k+1}(F_\nu)$.
Under the assumption of nonvanishing of the integral \eqref{intro4}, we conclude that the Hom space
$$\Hom_{Sp_{2n}(F_\nu)\times U_{n+\frac{r+1}{2}-k,\frac{r+1}{2}-k}(F_\nu)}(\overline{\pi^{(r)}_\nu}\otimes \Theta_{2n}^{(2),\psi^\delta}\otimes J_{U_{1,r_1,n+r_1-k}(F_\nu)}(\Theta_{2(n+r-k)}^{(r)})
 \psi_{U_{n+\frac{r+1}{2}-k,\frac{r+1}{2}-k}},
\C)$$
is nonzero, where the group $U_{n+\frac{r+1}{2}-k,\frac{r+1}{2}-k}(F_\nu)$ acts on $\Theta_{2n}^{(2),\psi^\delta}$ via $l_{2n}$.

We now use information about the local descent of $\Theta^{(r)}_{2(n+r-k)}$, namely that the local version of Conjecture~\ref{conj20} is true.  This is given the following Lemma.

\begin{lemma}\label{local-descent} Suppose that all data are unramified.  Then 
$$J_{U_{1,r_1,n+r_1-k}(F_\nu)}\left(\Theta_{2(n+r-k)}^{(r)}\right)\cong \Theta^{(2r)}_{2n+r-2k+1}.$$
\end{lemma}

Though this is not formally stated in \cite{F-G1}, it follows from the local versions of the arguments in Section~4 there.  Let $\nu$ be an unramified place and let $O_\nu$ be the ring of integers of $F_\nu$.
The computation there is equivalent to showing that the Jacquet module at $\nu$ is nonzero and
its exponent matches that of $\Theta^{(2r)}_{2n+r-2k+1}$. Recall that the local theta representation $\Theta_{2(n+r-k)}^{(r)}$ (resp.\ $\Theta^{(2r)}_{2n+r-2k+1}$)
 is generated by a nonzero $K_1$-fixed vector 
with $K_1\cong Sp_{2(n+r-k)}(O_\nu)$ a
compact open subgroup of $Sp_{2(n+r-k)}^{(r)}(F_\nu)$ (resp.\ a nonzero $K_2$-fixed vector with $K_2\cong Sp_{2n+r-2k+1}(O_\nu)$ a compact open subgroup of $Sp^{(2r)}_{2n+r-2k+1}(F_\nu)$).
Since the image of the $K_1$-fixed vector in the Jacquet module is $K_2$-fixed and the exponents match, the Lemma follows.

It is sufficient to show that the representations in questions do not support 
such a local trilinear form. 
This may be established the same way as the analysis of integral \eqref{whca221} above, but ported to local fields. For example, instead of unfolding we use Frobenius
reciprocity, and instead of a Fourier expansion we use the Geometrical Lemma of \cite{B-Z}, p.\ 448.  See \cite{F-G3}, Section~6, for an example of such an argument.
The contribution from the terms involving the constant terms will vanish by the assumption of general position.  The result follows.


\begin{thebibliography}{AAAAA}

\bibitem[Bak]{Bak} Baki\'c, P.: Theta lifts of generic representations for dual pairs $(Sp_{2n},O(V))$.  Manuscripta Math. {\bf 165} (2021), no. 3-4, 291--338. 

\bibitem[B-Z]{B-Z} Bernstein, I.N.; Zelevinsky, A.V.: Induced representations of reductive $p$-adic groups, I.
Ann. scient. \'Ec. Norm. Sup. {\bf 10} (1977), no.\ 4, 441--472.

\bibitem[C-M]{C-M}  Collingwood, D. H.; McGovern, W. M.:
Nilpotent orbits in semisimple Lie algebras.
Van Nostrand Reinhold Mathematics Series. Van Nostrand Reinhold Co., New York (1993).

\bibitem[F-G1]{F-G1}  Friedberg, S.; Ginzburg, D.: Theta functions on covers of symplectic groups. 
In: Automorphic Forms, L-functions and Number Theory, a special issue of BIMS in honor of Prof. Freydoon Shahidi's 70th birthday.  Bull. Iranian Math. Soc. {\bf 43} (2017), no. 4, 89--116.

\bibitem[F-G2]{F-G2}  Friedberg, S.; Ginzburg, D.: Descent and theta functions for metaplectic groups. Journal of the European Mathematical Society (JEMS) {\bf 20} (2018), no. 8, 1913--1957.

\bibitem[F-G3]{F-G3} Friedberg, S.; Ginzburg, D.:  Classical theta lifts for higher metaplectic covering groups. Geometric and Functional Analysis (GAFA) {\bf 30} (2020), no. 6, 1531--1582. 

\bibitem[G-S]{G-S} Gan, W.-T.; Savin, G.: Real and global lifts from $PGL(3)$ to $G_2$. Int. Math. Res. Not. {\bf 2003} (2003), no. 50, 2699--2724.

\bibitem[Gao]{Gao} Gao, F.: Distinguished theta representations for certain covering groups.  Pac. J. Math. {\bf 290}, no. 2, 333--379.

\bibitem[G-T]{G-T} Gao, F.; Tsai, W.-Y.: On the wavefront sets associated with theta representations.  arXiv:2008.03630. 

\bibitem[G]{G} Ginzburg, D.: Certain conjectures relating unipotent orbits to automorphic representations. Israel J. Math. {\bf 151} (2006), 323--355.

\bibitem[G-R-S1]{G-R-S1} Ginzburg, D.; Rallis, S.; Soudry, D.: Periods, poles of $L$-functions, and symplectic-orthogonal lifts.  J. reine angew. Math. {\bf 487} (1997), 85--114.

\bibitem[G-R-S2]{G-R-S2} Ginzburg, D.; Rallis, S.; Soudry, D.: L-functions for symplectic groups. Bull. Soc. Math. France {\bf 126} (1998), no. 2, 181--244.

\bibitem[G-R-S3]{G-R-S3} Ginzburg, D.; Rallis, S.; Soudry, D.: On Fourier coefficients of automorphic forms of symplectic groups. Manuscripta Math. {\bf 111} (2003), no. 1, 1--16. 

\bibitem[G-R-S4]{G-R-S4}  Ginzburg, D.; Rallis, S.; Soudry, D.: The descent map from automorphic representations of $GL(n)$ to classical groups. 
World Scientific Publishing Co. Pte. Ltd., Hackensack, NJ (2011). 

\bibitem[Ho]{Ho} Howe, R.: $\theta$-series and invariant theory. Automorphic forms, representations and L-functions (Proc. Sympos. Pure Math., Oregon State Univ., Corvallis, Ore., 1977), Part 1, 
pp. 275--285, Proc. Sympos. Pure Math., XXXIII, Amer. Math. Soc., Providence, R.I., 1979.

\bibitem[I1]{I1}  Ikeda, T.: On the theory of Jacobi forms and Fourier-Jacobi
coefficients of Eisenstein series. J. Math. Kyoto Univ. {\bf 34} (1994), no. 3, 615--636.

\bibitem[K-P]{K-P}  Kazhdan, D. A.; Patterson, S. J.: Metaplectic forms. Inst. Hautes \'Etudes Sci. Publ. Math., No. 59 (1984), 35--142.

\bibitem[Ku1]{Ku1}  Kudla, S.: Splitting metaplectic covers of dual reductive pairs.  Israel J. Math. {\bf 87} (1994), no. 1-3, 361--401.

\bibitem[Mat]{Mat} Matsumoto, H.:
Sur les sous-groupes arithm\'etiques des groupes semi-simples d\'eploy\'es. 
Ann. Sci. \'Ecole Norm. Sup. (4) 2 (1969), 1--62.

\bibitem[M-W]{M-W} M\oe glin, C.; Waldspurger, J.-L.: 
Spectral decomposition and Eisenstein series.
Une paraphrase de l'\'Ecriture [A paraphrase of Scripture]. Cambridge Tracts in Mathematics, 113. Cambridge University Press, Cambridge, 1995

\bibitem[Pat]{Pat} Patterson, S.J.: A cubic analogue of the theta series.
J. Reine Angew. Math. {\bf 296} (1977), 125--161.

\bibitem[R]{R} Rallis, S.: On the Howe duality conjecture. Compositio Math. {\bf 51} (1984), no. 3, 333--399.

\bibitem[Ro]{Ro} Roberts, B. Nonvanishing of global theta lifts from orthogonal groups. J. Ramanujan Math. Soc. {\bf 14} (1999), no.\ 2, 131--194.

\bibitem[Sw]{Sw} Sweet, W. J. Jr.: The metaplectic case of the Weil-Siegel formula. 
Thesis, Univ.\ of Maryland, 1990.

\end{thebibliography}
\end{document}